\documentclass[11pt, reqno]{amsart}
\usepackage{amsfonts,latexsym,enumerate}
\usepackage{amsmath}
\usepackage{amscd}
\usepackage{float,amsmath,amssymb,mathrsfs,bm,multirow,graphics}
\usepackage[dvips]{graphicx}
\usepackage[percent]{overpic}
\usepackage[pdftex]{color}
\usepackage{amsaddr}
\usepackage[numbers,sort&compress]{natbib}

\newcommand{\R}{{\Bbb R}}

\newcommand{\C}{{\Bbb C}}
\newcommand{\Z}{{\Bbb Z}}

\newcommand{\stepproofbegin}{\noindent{\it Proof. }}
\newcommand{\proofendcontinue}{\hfill \raisebox{.8mm}[0cm][0cm]{$\bigtriangledown$}\bigskip}
\newcommand{\proofend}{\hfill$\Box$\bigskip}

\newcommand{\tr}{\text{\upshape tr\,}}

\newcommand{\im}{\text{\upshape Im\,}}

\def\XXint#1#2#3{{\setbox0=\hbox{$#1{#2#3}{\int}$}
\vcenter{\hbox{$#2#3$}}\kern-.5\wd0}}

\newtheorem{theorem}{Theorem}[section]

\newtheorem{corollary}[theorem]{Corollary}
\newtheorem{lemma}[theorem]{Lemma}
\newtheorem{definition}[theorem]{Definition}

\newtheorem{remark}[theorem]{Remark}

\newtheorem{figuretext}{Figure}

\numberwithin{equation}{section}

\input epsf
\date{\today}
\title[Nonlinear Fourier Transforms]
{Nonlinear Fourier Transforms for the sine-Gordon equation in the Quarter Plane}

\author{Lin Huang and Jonatan Lenells}
\address{Department of Mathematics, KTH Royal Institute of Technology, \\ 100 44 Stockholm, Sweden.}
\email{linhuang@kth.se}
\email{jlenells@kth.se}

\begin{document}
\begin{abstract} 
\noindent
The solution of the sine-Gordon equation in the quarter plane can be expressed in terms of the solution of a matrix Riemann-Hilbert problem whose definition involves four spectral functions $a,b,A,B$. The functions $a(k)$ and $b(k)$ are defined via a nonlinear Fourier transform of the initial data, whereas $A(k)$ and $B(k)$ are defined via a nonlinear Fourier transform of the boundary values. 
In this paper, we provide an extensive study of these nonlinear Fourier transforms and the associated eigenfunctions under weak regularity and decay assumptions on the initial and boundary values. The results can be used to determine the long-time asymptotics of the sine-Gordon quarter-plane solution via nonlinear steepest descent techniques. 
\end{abstract}

\maketitle

\noindent
{\small{\sc AMS Subject Classification (2010)}: 37K15, 41A60, 35P25.}

\noindent
{\small{\sc Keywords}: Nonlinear Fourier transform, spectral function, sine-Gordon equation, inverse scattering, initial-boundary value problem.}

\setcounter{tocdepth}{1}
\tableofcontents

\section{Introduction}
The standard method for solving the initial-value problem for a {\it linear} partial differential equation (PDE) is to Fourier transform in space. Since the Fourier transform diagonalizes the differential operator, the time evolution in Fourier space is particularly simple and reduces to solving a linear ordinary differential equation for each frequency component.

In the context of {\it nonlinear} PDEs, the Fourier transform does not provide the same kind of simplification of the time evolution. Nonlinear PDEs can therefore not, in general, be analyzed in such a simple manner. However, there exists a class of nonlinear PDEs (called {\it integrable} PDEs) for which an analog of the above procedure exists. The first equation discovered to be integrable was the Korteweg--de Vries (KdV) equation \cite{GGKM1967}, soon followed by the nonlinear Schr\"odinger \cite{ZS1972} and sine-Gordon equations \cite{AKNS1973, K1975, ZTF1974}.

The class of integrable PDEs is set apart by the fact that it is possible to implement a `nonlinear version' of the Fourier transform which reduces the time evolution to a set of linear ordinary differential equations.
This nonlinear Fourier transform has to be tailor-made for each specific equation and takes, in general, a different form for each equation. For the initial value problem on the line, the transform is usually referred to as the Inverse Scattering Transform (IST) due to the fact that its implementation for the KdV equation involves a scattering problem for the Schr\"odinger equation in quantum mechanics. 

The IST formalism provides a powerful framework for analyzing integrable nonlinear PDEs. For example, it is straightforward via IST techniques to construct large families of exact solutions, such as multi-soliton solutions. Soliton-generating techniques, being largely algebraic in nature, usually require a limited amount of analytical groundwork. On the other hand, the implementation of the IST relevant for the solution of the general initial value problem, or for the study of asymptotics, relies on detailed analytical estimates. In particular, estimates on the associated nonlinear Fourier transform and its inverse are required. To establish these estimates is a challenging and technical enterprise even for relatively simple cases. In the case of the KdV equation on the line, whose nonlinear Fourier transform is determined by the one-dimensional Schr\"odinger operator $-\partial_x^2 + q(x)$, Deift and Trubowitz put the IST analysis on a rigorous footing in the elegant but long paper \cite{DT1979}.

An important development in recent years has been the extension of the IST formalism to initial-boundary value (IBV) problems (see \cite{F2002, Fbook} and references therein).
Although it is possible to treat more complicated domains, we here restrict attention to the case of IBV problems posed in the quarter-plane domain $\{x \geq 0, t \geq 0\}$, i.e., problems that involve a single boundary located at $x = 0$. Employing the machinery of \cite{F2002}, the solution of an integrable PDEs with a $2 \times 2$-matrix Lax pair in such a domain, can be expressed in terms of the solution of a matrix Riemann-Hilbert (RH) problem whose formulation involves four spectral functions $a(k)$, $b(k)$, $A(k)$, and $B(k)$. The functions $\{a(k), b(k)\}$ are defined via a nonlinear Fourier transform of the initial data on the half-line $\{x \geq 0, t = 0\}$, whereas $\{A(k), B(k)\}$ are defined via a nonlinear Fourier transform of the boundary values on $\{x = 0, t \geq 0\}$. 

The purpose of the present paper is to study the nonlinear Fourier transforms required for the analysis of the sine-Gordon equation in laboratory coordinates,
\begin{align}\label{sg}
u_{tt}- u_{xx} + \sin u=0,
\end{align}
in the quarter plane. The sine-Gordon equation has numerous applications. It is the Gauss-Codazzi equation for surfaces of constant negative curvature embedded in three-dimensional Euclidean space, it was the first equation for which B\"acklund transformations were discovered, it is the continuous limit of the Frenkel-Kontorova model in condensed matter physics, it models the magnetic flux propagation in Josephson junctions, and it can be used to describe several phenomena in nonlinear optics such as self-induced transparency, see the review \cite{M2014} and references therein. For several of these applications, boundaries play an important role, motivating the study of IBV problems for (\ref{sg}) in addition to the pure initial-value problem. 

In order to consider the quarter-plane problem for equation (\ref{sg}), we let $u_0, u_1$ denote the initial data and let $g_0$ and $g_1$ denote the Dirichlet and Neumann boundary values, respectively:
\begin{align}\label{uugg}
\begin{cases} u_0(x) = u(x,0), \\ u_1(x) = u_t(x,0), \end{cases}  x \geq 0; \qquad\qquad
\begin{cases} g_0(t) = u(0,t), \\ g_1(t) = u_x(0,t), \end{cases}  t \geq 0.
\end{align} 
The spectral functions $a,b,A,B$ that enter the formulation of the RH problem are defined by 
$$X(0,k) = \begin{pmatrix}
\overline{a(\bar{k})} 	&	b(k)	\\
-\overline{b(\bar{k})}	&	a(k)
\end{pmatrix}, \qquad
T(0,k) = \begin{pmatrix}
\overline{A(\bar{k})} 	&	B(k)	\\
-\overline{B(\bar{k})}	&	A(k)
\end{pmatrix},$$
where $X(x,k)$  and $T(t,k)$ are the eigenfunction solutions of the $x$- and $t$-parts of the associated Lax pair, normalized to approach the identity matrix $I$ as $x \to \infty$ and $t \to \infty$, respectively \cite{F2002}. 
Of particular importance are the combinations $c := bA -aB$ and $d := a\bar{A} + b\bar{B}$, which enter the product matrix $T(0,k)^{-1}S(0,k)$.
The RH formalism also involves two eigenfunctions $Y(x,k)$ and $U(t,k)$, which are similar to $X$ and $T$, but normalized at the origin.

It is crucial for the implementation of the RH approach to have a good understanding of the above eigenfunctions and spectral functions. 
In this paper, we provide an extensive study of these functions under weak regularity and decay assumptions. The main results of the paper can be summarized as follows:

\begin{itemize}

\item Theorem \ref{xth1} establishes analyticity properties and estimates for the eigenfunctions $X$ and $Y$ and their derivatives.

\item Theorem \ref{xth2} and Corollary \ref{xcor} establish the asymptotics of $X$ and $Y$ as $k \to \infty$ and $k \to 0$, respectively.

\item Theorem \ref{tth1}, Theorem \ref{tth2}, and Corollary \ref{tcor} establish analogous results for the eigenfunctions $T$ and $U$.

\item Theorem \ref{abth}, Theorem \ref{ABth}, and Theorem \ref{cdth} establish analyticity properties and asymptotics as $k \to \infty$ and $k \to 0$ of the spectral functions $\{a(k), b(k)\}$, $\{A(k), B(k)\}$, and $\{c(k), d(k)\}$, respectively. 

\end{itemize}

The study of equation (\ref{sg}) in the quarter plane was initiated by Fokas and Its \cite{FI1992} who showed that the solution can be related to the solution of a RH problem. 
Analogs of the above eigenfunctions and spectral functions appear already in this reference, which also contains a brief discussion of their asymptotics (mostly to leading order); see \cite{F2002, Fbook} for some further discussion. The main contributions of the present paper are: 
\begin{enumerate}[\hspace{-.1cm}$(a)$]
\item We establish asymptotic expansions of the eigenfunctions, the spectral functions, and their derivatives to all orders\footnote{More precisely, to all orders permitted by the regularity; roughly speaking, the asymptotic expansions of the eigenfunctions and spectral functions exist to order $N$ if the potential is $N+1$ times weakly differentiable.} as $k \to \infty$. The proofs are inspired by the approach of \cite{LNonlinearFourier}, where the nonlinear Fourier transforms associated with the mKdV equation on the half-line were analyzed. However, by improving the arguments of \cite{LNonlinearFourier}, we are able to obtain an $L^1$ theory which allows for weaker solutions, cf. (\ref{xth1assump}) and (\ref{ujassump}).

\item We establish analogous expansions as $k \to 0$. The point $k = 0$ is special for the sine-Gordon equation because the Lax pair is singular at this point. 
Corollaries \ref{xcor} and \ref{tcor} provide the asymptotics as $k \to 0$ of the eigenfunctions $X,Y,U$, and $T$ to all orders. If the higher-order terms are ignored, the expansions reduce to the formulas for the leading-order behavior of $X$ and $U$ obtained in Section 16.A of \cite{Fbook}.

\item We show that that if the functions $u_0,u_1,g_0,g_1$ are compatible at the origin to a given order, then $c(k)$ vanishes to that same order at $k = \infty$  and at $k = 0$ (see Theorem \ref{cdth}). If $u_0,u_1,g_0,g_1$ are the initial and boundary values of some sufficiently well-behaved solution in the quarter plane with decay as $x\to \infty$, then $c(k)$ vanishes identically  in the domain $D_1 = \{\im k > 0\} \cap \{|k| > 1\}$; this fact is referred to as the {\it global relation} and is one of the fundamental observations underlying the RH approach to IBV problems, see \cite{Fbook}. 
Theorem \ref{cdth} clarifies the role of the global relation by showing that $c(k)$ vanishes to all orders at $k = \infty$ for an {\it arbitrary} choice of the functions $u_0,u_1,g_0,g_1$ as long as they are compatible with (\ref{sg}) at the origin.\footnote{Note that the function $c(k)$ can be nonzero even if its asymptotic expansion as $k \to \infty$ vanishes to all orders, because the point $k = \infty$ lies on the boundary of its analyticity domain.}
Thus, identical vanishing of $c(k)$ in $D_1$ is tied to the functions $u_0,u_1,g_0,g_1$ providing consistent initial and boundary values for a quarter-plane solution, whereas vanishing of the asymptotic expansion of $c(k)$ as $k \to \infty$ is tied to compatibility at the origin. To the best of our knowledge, this is the first time this connection has been noticed for an IBV problem for an integrable equation (see Appendix B of \cite{FLnonlinearizable} for some different but related observations).

\end{enumerate}

A salient feature of the sine-Gordon equation is that it supports solutions with nonzero winding number (also known as topological charge). This is related to the fact that equation (\ref{sg}) is invariant under shifts of the form $u \to u + 2\pi j$, $j \in \Z$, so that $u(x,t)$ is naturally viewed as a map $u:[0,\infty) \times [0,\infty) \to S^1$, where $S^1 = \R / 2\pi \Z$ denotes the unit circle. For the problem on the line, it is customary to impose the boundary conditions $\lim_{|x| \to \infty} u(x,t) = 0$ (mod $2\pi$).
These conditions allow for solutions $u(x,t)$ which approach different multiples of $2\pi$ as $x$ approaches plus and minus infinity. The integer
$$\frac{1}{2\pi}\Big(\lim_{x \to \infty} u(x,t) - \lim_{x \to -\infty} u(x,t)\Big)$$
is referred to as the topological charge of the solution. Analogously, for the half-line problem, we impose the conditions 
$$\lim_{x \to \infty} u_0(x) = 2\pi N_x,  \qquad \lim_{t \to \infty} g_0(t) = 2\pi N_t, \qquad N_x, N_t \in \Z,$$
and refer to the integer $N_x - N_t \in \Z$ as the topological charge of the solution. 

Let us finally point out that a major motivation for writing this paper was to open up for the derivation of asymptotic formulas for (\ref{sg}) on the half-line. In fact, by employing the results derived in this paper, we have obtained a rather complete picture of the long-time asymptotics of the quarter-plane solutions of (\ref{sg}), establishing formulas for the leading-order solitonic contributions, the subleading radiative corrections, as well as the topological charge of the solution \cite{HLasymptotics}. In \cite{HLasymptotics}, we also utilize the results proved here to construct new quarter-plane solutions of the sine-Gordon equation, both in terms of prescribed spectral data, and in terms of given initial and boundary values satisfying the global relation. 
For the above and other applications, the behavior of the eigenfunctions and spectral functions near the points $k = \infty$ and $k = 0$ is of fundamental importance. For example, the behavior of $c(k)$ at the origin determines the asymptotics of the solution $u(x,t)$  in the transition region between the superluminal ($|x/t| > 1$) and subluminal ($|x/t| < 1$) sectors. Theorem \ref{cdth} grew out of a desire to understand the asymptotics in this sector.

\subsection{Notation} Throughout the paper, we use $\C_\pm = \{k \in \C \, | \, \im k \gtrless 0\}$ to denote the open upper and lower half-planes. The closed half-planes are denoted by $\bar{\C}_+ = \{k \in \C \, | \, \im k \geq 0\}$ and $\bar{\C}_- = \{k \in \C \, | \, \im k \leq 0\}$.
We let $\{D_j\}_1^4$ denote the open domains of the complex $k$-plane defined by (see Figure \ref{Gamma.pdf})
\begin{align}\nonumber
D_1 = \{\im k > 0\} \cap \{|k|>1\},  \qquad
D_2 = \{\im k > 0\} \cap \{|k|<1\},
	\\ \nonumber
D_3 = \{\im k < 0\} \cap \{|k|<1\},  \qquad
D_4 = \{\im k < 0\} \cap \{|k|>1\},
\end{align}
and let $\Gamma = \R \cup \{|k| = 1\}$ denote the contour separating the $D_j$, oriented so that $D_1 \cup D_3$ lies to the left as in Figure \ref{Gamma.pdf}. We use $C > 0$ to denote a generic constant which may change within a computation.

\begin{figure}
\begin{center}
\bigskip\bigskip\bigskip
\begin{overpic}[width=.55\textwidth]{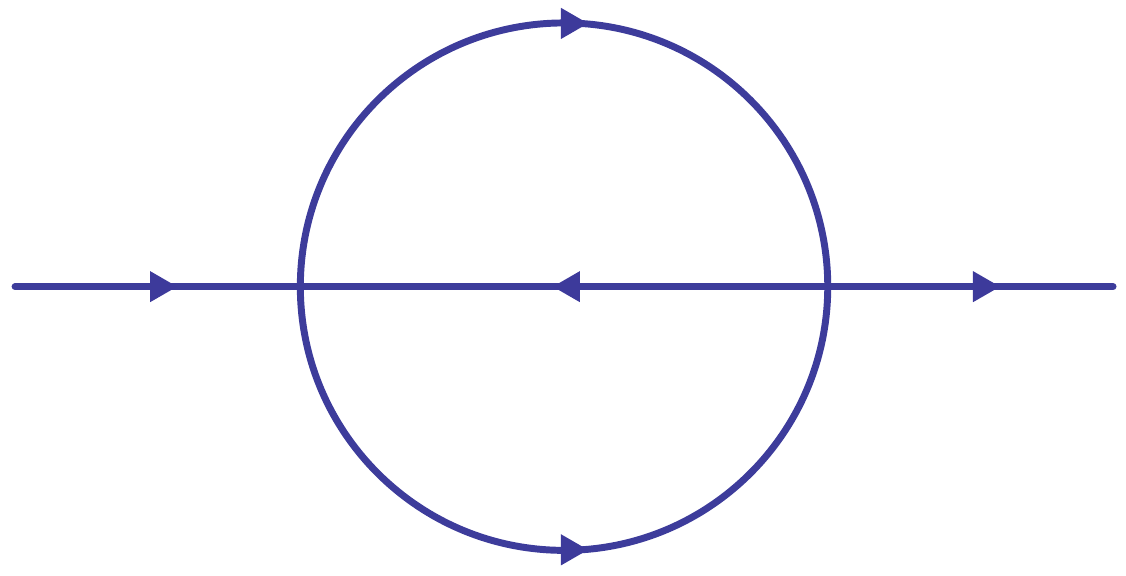}
      \put(74.5,21.5){\small $1$}
      \put(20,21.5){\small $-1$}
      \put(48,57){\small $D_1$}
      \put(48,35){\small $D_2$}
      \put(48,13){\small $D_3$}
      \put(48,-7){\small $D_4$}
      \put(102,24.5){\small $\Gamma$}
      \end{overpic}\bigskip\bigskip\bigskip
     \begin{figuretext}\label{Gamma.pdf}
       The contour $\Gamma$ and the domains $\{D_j\}_1^4$ in the complex $k$-plane.
     \end{figuretext}
     \end{center}
\end{figure}

\section{Background}\label{backgroundsec}

\subsection{Lax pair}
Let $\{\sigma_j\}_1^3$ denote the three Pauli matrices defined by
\begin{align*}
\sigma_1=\begin{pmatrix}
           0 & 1 \\
           1 & 0
         \end{pmatrix},\quad
 \sigma_2=\begin{pmatrix}
      0 & -i \\
       i & 0
\end{pmatrix},\quad 
\sigma_3=\begin{pmatrix}
          1 & 0 \\
          0 & -1
 \end{pmatrix}.
\end{align*}
Equation (\ref{sg}) is the compatibility condition of the Lax pair
\begin{align}\label{lax}
\begin{cases}
\mu_x +\frac{i}{4} (k-\frac{1}{k})[\sigma_3, \mu] = Q(x,t,k) \mu,
	\\
\mu_t + \frac{i}{4} (k+\frac{1}{k} )[\sigma_3, \mu] = Q(x,t,-k)\mu,
\end{cases} 
\end{align}
where $k \in \C$ is the spectral parameter, $\mu(x,t,k)$ is a $2\times 2$-matrix valued eigenfunction, and $Q$ is defined by
\begin{align}\label{Qdef}
  Q(x,t,k) = Q_0(x,t) + \frac{Q_1(x,t)}{k}
  \end{align}
with
$$  Q_0(x,t) = -\frac{i(u_x+u_t)}{4}\sigma_2, \qquad
 Q_1(x,t) = \frac{i\sin u}{4}\sigma_1 + \frac{i(\cos u-1)}{4}\sigma_3.$$

 Since $Q$ has a pole at $k = 0$, the Lax pair (\ref{lax}) is not well-adapted for determining the behavior of $\mu$ as $k \to 0$. 
We therefore introduce a second Lax pair as follows. Define $G(x,t)$ by
\begin{align}\label{Gxtdef}
G(x,t)= \begin{pmatrix}
  \cos \frac{u}{2} & -\sin \frac{u}{2} \\
  \sin \frac{u}{2}  & \cos \frac{u}{2}
 \end{pmatrix}.
\end{align}
If $\mu$ satisfies (\ref{lax}), then the function $\hat{\mu}$ defined by
$$\mu(x,t,k) = G(x,t)\hat{\mu}(x,t,k)$$
satisfies
\begin{align*}
\begin{cases}
\hat{\mu}_x +\frac{i}{4} (k-\frac{1}{k})[\sigma_3, \hat{\mu}] = (\hat{Q}_0 + k\hat{Q}_1)\hat{\mu},
	\\
\hat{\mu}_t + \frac{i}{4} (k+\frac{1}{k} )[\sigma_3, \hat{\mu}] = (-\hat{Q}_0 + k\hat{Q}_1)\hat{\mu},
\end{cases} 
\end{align*}
where
$$\hat{Q}_0(x,t) = \frac{i(u_x-u_t)}{4}\sigma_2, \qquad 
\hat{Q}_1(x,t) = \frac{i\sin u}{4}\sigma_1 - \frac{i(\cos u-1)}{4}\sigma_3.$$

\subsection{Spectral functions}
We define two spectral functions $\{a(k), b(k)\}$ in terms of the initial data $\{u_0(x), u_1(x)\}$ and two spectral functions $\{A(k), B(k)\}$ in terms of the boundary values $\{g_0(t), g_1(t)\}$ as follows:
Define $\theta_1(k)$ and $\theta_2(k)$ by
\begin{align}\label{thetadef}
\theta_1(k) = \frac{1}{4}\bigg(k - \frac{1}{k}\bigg), \qquad  \theta_2(k) = \frac{1}{4}\bigg(k + \frac{1}{k}\bigg).
\end{align}
Let $\mathsf{U}(x,k)$ and $\mathsf{V}(t,k)$ be given by (cf. definition (\ref{Qdef}) of $Q(x,t,k)$)
\begin{align}\label{mathsfUdef}
& \mathsf{U}(x,k) = -\frac{i(u_{0x}+u_1)}{4}\sigma_2 + \frac{i\sin u_0}{4k}\sigma_1 + \frac{i(\cos u_0-1)}{4k}\sigma_3,
	\\ \label{mathsfVdef}
& \mathsf{V}(t,k)= -\frac{i(g_1+g_{0t})}{4}\sigma_2 - \frac{i\sin g_0}{4k}\sigma_1 - \frac{i(\cos g_0-1)}{4k}\sigma_3,
\end{align}
and define the $2 \times 2$-matrix valued functions $X(x,k)$ and $T(t,k)$ as the unique solutions of the linear Volterra integral equations
\begin{align}\label{Xvolterra}
& X(x,k) = I + \int_{\infty}^x e^{i\theta_1(x'-x)\hat{\sigma}_3} (\mathsf{U}X)(x',k) dx',
	\\ \label{Tvolterra}
&  T(t,k) = I + \int_{\infty}^t e^{i\theta_2(t'-t)\hat{\sigma}_3} (\mathsf{V}T)(t',k) dt',
\end{align}
where $\hat{\sigma}_3$ acts on a $2\times 2$ matrix $A$ by $\hat{\sigma}_3A = [\sigma_3, A]$, i.e., $e^{\hat{\sigma}_3} A = e^{\sigma_3} A e^{-\sigma_3}$.
The symmetries
$$X(x,k) =\sigma_2 \overline{X(x, \bar{k})}\sigma_2, \qquad  T(t,k)=\sigma_2 \overline{T(t, \bar{k})}\sigma_2,$$
imply that we can define $a,b,A,B$ by
\begin{align}\label{abABdef}
s(k) = \begin{pmatrix}
\overline{a(\bar{k})} 	&	b(k)	\\
-\overline{b(\bar{k})}	&	a(k)
\end{pmatrix}, \qquad
S(k) = \begin{pmatrix}
\overline{A(\bar{k})} 	&	B(k)	\\
-\overline{B(\bar{k})}	&	A(k)
\end{pmatrix},
\end{align}
where $s(k) := X(0,k)$ and $S(k) := T(0,k)$.
We think of the maps $\{u_0(x), u_1(x)\} \mapsto \{a(k), b(k)\}$ and $\{g_0(x), g_1(x)\} \mapsto \{A(k), B(k)\}$ as (half-line) nonlinear Fourier transforms. 

The importance of the above spectral functions lies in the fact that the solution of (\ref{sg}) in the quarter plane can be represented in terms of the solution of a $2 \times 2$-matrix RH problem, whose formulation involves $a,b,A,B$, see \cite{F2002}. Of particular interest for the RH formulation are the combinations $c(k)$ and $d(k)$ defined by 
\begin{subequations}\label{cddef}
\begin{align}
& c(k) = b(k)A(k) - a(k) B(k), \qquad k \in \bar{D}_1 \cup \R,
	\\
& d(k) = a(k)\overline{A(\bar{k})} +  b(k) \overline{B(\bar{k})}, \qquad k \in \bar{D}_2 \cup \R.
\end{align}
\end{subequations}

\section{Spectral analysis of the $x$-part}\label{xpartapp}
Consider the $x$-part of the Lax pair (\ref{lax}) evaluated at $t = 0$:
\begin{align}\label{xpartF}
  X_x + i\theta_1[\sigma_3, X] = \mathsf{U} X,
\end{align}
where $\theta_1 := \theta_1(k) = \frac{1}{4}(k - \frac{1}{k})$ and $\mathsf{U}(x,k)$ is given by (\ref{mathsfUdef}), i.e.,
\begin{align*}
  \mathsf{U}(x,k) = \mathsf{U}_0(x)+\frac{1}{k} \mathsf{U}_1(x),
\end{align*} 
with 
$$\mathsf{U}_0(x)=-\frac{i(u_{0x}(x) +u_1(x))}{4}\sigma_2,\qquad 
  \mathsf{U}_1(x) = \frac{i\sin u_0(x)}{4} \sigma_1 + \frac{i(\cos{u_0(x)} -1)}{4}\sigma_3.$$
We define two $2 \times 2$-matrix valued solutions $X(x,k)$ and $Y(x,k)$ of (\ref{xpartF}) as the solutions of the linear Volterra integral equations
\begin{subequations}\label{XYdef}
\begin{align}  \label{XYdefa}
 & X(x,k) = I + \int_{\infty}^x e^{i\theta_1(x'-x)\hat{\sigma}_3} (\mathsf{U}X)(x',k) dx',
  	\\ \label{XYdefb}
&  Y(x,k) = I + \int_{0}^x e^{i\theta_1(x'-x)\hat{\sigma}_3} (\mathsf{U}Y)(x',k) dx'.
\end{align}
\end{subequations}

If $A$ is an $n \times m$ matrix, we define $|A| \geq 0$ by
$|A|^2 = \sum_{i,j} |A_{ij}|^2$. Then $|A + B| \leq |A| + |B|$ and $|AB| \leq |A| |B|$. The notation $k \in (D_i, D_j)$ indicates that the first and second columns of the preceding equation hold for $k \in D_i$ and $k \in D_j$, respectively.

\begin{theorem}[Basic properties of $X$ and $Y$]\label{xth1}
Let $n \geq 1$ and $N_x \in \Z$ be integers. Suppose $u_0, u_1:[0,\infty)\to \R$ satisfy
\begin{align}\label{xth1assump}
  (1+x)^n(|u_0 - 2\pi N_x| + |u_{0x}|+|u_1|)\in  L^1([0,\infty)).
\end{align}
Then the equations (\ref{XYdef}) uniquely define two $2 \times 2$-matrix valued solutions $X$ and $Y$ of (\ref{xpartF}) with the following properties:
\begin{enumerate}[$(a)$]
\item The function $X(x, k)$ is defined for $x \geq 0$ and $k \in (\bar{\C}_-, \bar{\C}_+) \setminus \{0\}$. For each $k \in (\bar{\C}_-, \bar{\C}_+)\setminus \{0\}$, $X(\cdot, k)$ is weakly differentiable and satisfies (\ref{xpartF}).

\item The function $Y(x, k)$ is defined for $x \geq 0$ and $k \in \C\setminus \{0\}$. For each $k \in \C\setminus\{0\}$, $Y(\cdot, k)$ is weakly differentiable and satisfies (\ref{xpartF}).

\item For each $x \geq 0$, the function $X(x,\cdot)$ is continuous for $k \in (\bar{\C}_-, \bar{\C}_+)\setminus \{0\}$ and analytic for $k \in (\C_-, \C_+)$.

\item For each $x \geq 0$, the function $Y(x,\cdot)$ is analytic for $k \in \C \setminus \{0\}$.

\item For each $x \geq 0$ and each $j = 1, \dots, n$, the partial derivative $\frac{\partial^j X}{\partial k^j}(x, \cdot)$ has a continuous extension to $(\bar{\C}_-, \bar{\C}_+)\setminus \{0\}$.

\item $X$ and $Y$ satisfy the following estimates:
\begin{subequations}\label{Fest}
\begin{align}\label{Festa}
& \bigg|\frac{\partial^j}{\partial k^j}\big(X(x,k) - I\big) \bigg| \leq
\frac{C}{(1+x)^{n-j}}, \qquad k \in (\bar{D}_4, \bar{D}_1),
	\\\label{Festb}
& \bigg|\frac{\partial^j}{\partial k^j}\Big(Y(x,k) - I\Big) \bigg| \leq
C(1 + x)^j, \qquad k \in (\bar{D}_1, \bar{D}_4),
	\\\label{Festc}
& \bigg|\frac{\partial^j}{\partial k^j}\Big((Y(x,k) - I)e^{-2i\theta_1x\sigma_3}\Big) \bigg| \leq
C(1 + x)^j, \qquad k \in (\bar{D}_4, \bar{D}_1),
\end{align}
\end{subequations}
for $x \geq 0$ and $ j = 0, 1, \dots, n$.

\item $\det X(x,k) = 1$ for $x \geq 0$, $k \in \R\setminus \{0\}$, and $\det Y(x,k) = 1$ for $x \geq 0$, $k \in \C \setminus \{0\}$.

\item $X$ and $Y$ obey the following symmetries for each $x \geq 0$:
\begin{align*}
& X(x,k) = \sigma_2 X(x,-k) \sigma_2 = \sigma_2 \overline{X(x,\bar{k})} \sigma_2, \qquad k \in (\bar{\C}_-, \bar{\C}_+) \setminus \{0\},
	\\
&Y(x,k) = \sigma_2 Y(x,-k) \sigma_2 = \sigma_2 \overline{Y(x,\bar{k})} \sigma_2, \qquad k \in \C \setminus \{0\}.
\end{align*}
\end{enumerate}
\end{theorem}
\begin{proof}
If we can prove $(a)$-$(f)$, the unit determinant conditions in $(g)$ follow easily from (\ref{xpartF}). Indeed, since $\tr \mathsf{U} = 0$, we obtain $(\det X)_x = 0$ and $(\det Y)_x = 0$; evaluation at $x = \infty$ and $x = 0$, respectively, gives $(g)$. On the other hand, the symmetries in $(h)$ are a consequence of the analogous symmetries for $\mathsf{U}$ together with uniqueness.

Given a $2 \times 2$ matrix $A$, we write $[A]_1$ and $[A]_2$ for the first and second columns of $A$.
We will prove $(a)$-$(f)$ for $[X]_2$ and $[Y]_2$. The results for the first columns follow from similar arguments.

{\bf Construction of $[X]_2$.}
We first use (\ref{XYdefa}) to construct $[X]_2$. 
The assumption (\ref{xth1assump}) implies that $(1+x)^n \mathsf{U}(x,k) \in L^1([0,\infty))$ for each $k \neq 0$. 
Letting $\Psi := [X]_2$, the second column of (\ref{XYdefa}) can be written as
\begin{align}\label{xPsiVolterra}
\Psi(x,k) = \begin{pmatrix} 0 \\ 1 \end{pmatrix} - \int_x^\infty E(x,x',k)\mathsf{U}(x',k)\Psi(x',k) dx', \qquad x \geq 0,
\end{align}
where
$$E(x,x',k) =  \begin{pmatrix} e^{\frac{i}{2} (k-\frac{1}{k})(x' - x)} & 0 \\ 0 & 1 \end{pmatrix}.$$
Let $K > 0$ be a small constant and let $\bar{\C}_\pm^K := \bar{\C}_\pm \cap \{|k| \geq K\}$.
We will show that the Volterra equation (\ref{xPsiVolterra}) has a unique solution $\Psi(x,k)$ for each $k \in \bar{\C}_+^K$.
Let $\Psi_0 = \big(\begin{smallmatrix} 0 \\ 1 \end{smallmatrix}\big)$ and define $\Psi_l$ inductively for $l \geq 1$ by
\begin{align*}
\Psi_{l+1}(x,k) = - \int_x^\infty E(x,x',k) \mathsf{U}(x',k) \Psi_l(x', k) dx', \qquad x \geq 0, \  k \in \bar{\C}_+^K.
\end{align*}
Then
\begin{align}\label{xPhiliterated}
\Psi_l(x,k) = (-1)^l \int_{x = x_{l+1} \leq x_l \leq \cdots \leq x_1 < \infty} \prod_{i = 1}^l E(x_{i+1}, x_i, k) \mathsf{U}(x_i,k) \Psi_0 dx_1 \cdots dx_l.
\end{align}
The estimates
\begin{align}\label{xE1bound}
\begin{cases}
|\partial_k^jE(x, x', k)| < C(1+|x' - x|)^j, & 0 \leq x \leq x' < \infty,
	\\
 \|\partial^j_k \mathsf{U}(\cdot,k)\|_{L^1([x,\infty))} < \frac{C}{(1+x)^{n}}, & 0 \leq x < \infty,
 \end{cases}  \  k \in \bar{\C}_+^K, \  j = 0, 1, \dots, n,
\end{align}
with $j = 0$ yield
\begin{align}\nonumber
|\Psi_l(x,k)| \leq \; & C^l \int_{x \leq x_l \leq \cdots \leq x_1 < \infty} \prod_{i = 1}^l  |\mathsf{U}(x_i,k)| |\Psi_0| dx_1 \cdots dx_l
	\\ \nonumber
\leq & \; \frac{C^l}{l!} \|\mathsf{U}(\cdot,k)\|_{L^1([x,\infty))}^l
	\\\label{xPhilestimate}
\leq &\; \frac{C^l}{l!} \bigg(\frac{C}{(1+x)^n}\bigg)^l, \qquad x \geq 0, \  k \in \bar{\C}_+^K, \ l=1,2, \dots.
\end{align}
Hence the series
$$\Psi(x,k) = \sum_{l=0}^\infty \Psi_l(x,k)$$
converges absolutely and uniformly for $x \geq 0$ and $k \in \bar{\C}_+^K$ to a bounded solution $\Psi(x,k)$ of (\ref{xPsiVolterra}).
Equation (\ref{xPsiVolterra}) implies that $\Psi(\cdot, k) \in W^{1,1}_{\text{loc}}([0,\infty))$ for each $k \in \bar{\C}_+^K$, showing that $\Psi(\cdot, k)$ is a weakly differentiable solution of (\ref{xpartF}). 
Using (\ref{xE1bound}) to differentiate under the integral sign in (\ref{xPhiliterated}), we see that $\Psi_l(x, \cdot)$ is analytic in the interior of $\bar{\C}_+^K$ for each $l$; the uniform convergence then proves that $\Psi$  is analytic in the interior of $\bar{\C}_+^K$. Since $K > 0$ is arbitrary, this proves $(a)$ and $(c)$.
On the other hand, by (\ref{xPhilestimate}),
\begin{align}\label{xPsiminusf}
|\Psi(x,k) - \Psi_0| \leq \sum_{l=1}^\infty | \Psi_l(x,k)| \leq \frac{C}{(1+x)^n}, \qquad x \geq 0, \  k \in \bar{\C}_+^K,
\end{align}
which proves the second column of (\ref{Festa}) for $j = 0$ and $k \in \bar{\C}_+^K$. 

We next show that $[X]_2 = \Psi$ satisfies $(e)$ and $(f)$ for $j = 1, \dots, n$ and $k \in \bar{\C}_+^K$.
Let
\begin{align}\label{xLambda0}
\Lambda_0(x,k) =- \int_x^\infty \partial_k [E(x,x',k) \mathsf{U}(x',k)] \Psi(x', k) dx'.
\end{align}
Differentiating the integral equation (\ref{xPsiVolterra}) with respect to $k$, we find that $\Lambda := \partial_k\Psi$ satisfies
\begin{align}\label{xintegraleq3}
\Lambda(x,k) = \Lambda_0(x,k) - \int_x^\infty E(x,x',k) \mathsf{U}(x',k) \Lambda(x', k) dx'
\end{align}
for each $k$ in the interior of $\bar{\C}_+^K$; the differentiation can be justified by dominated convergence using (\ref{xE1bound}) and a Cauchy estimate for $\partial_k\Psi$.
We have $\Lambda = \sum_{l=0}^\infty \Lambda_l$, where the functions $\Lambda_l$ are defined by replacing $\{\Psi_l\}$ by $\{\Lambda_l\}$ in (\ref{xPhiliterated}). Indeed, proceeding as in (\ref{xPhilestimate}), we find
\begin{align}\label{EE0}
|\Lambda_l(x,k)| \leq \frac{C^l}{l!}\|\Lambda_0(\cdot, k)\|_{L^\infty([x,\infty))}\bigg(\frac{C}{(1+x)^n}\bigg)^l, \qquad x \geq 0, \  k \in \bar{\C}_+^K,\  l=1,2,\dots.
\end{align}
Using (\ref{xE1bound}) and (\ref{xPsiminusf}) in (\ref{xLambda0}), we obtain
\begin{align}\label{EE1}
|\Lambda_0(x,k)| \leq \frac{C}{(1+x)^{n-1}}, \qquad x \geq 0, \  k \in \bar{\C}_+^K.
\end{align}
In particular $\|\Lambda_0(\cdot, k)\|_{L^\infty([x,\infty))}$ is bounded for $x \geq 0$ and $k \in \bar{\C}_+^K$.
Thus, $\sum_{l=0}^\infty \Lambda_l$ converges uniformly on $[0,\infty) \times \bar{\C}_+^K$ to a bounded solution of (\ref{xintegraleq3}). By uniqueness, this solution equals $\Lambda = \partial_k\Psi$ for $k$ in the interior of $\bar{\C}_+^K$. The following analog of (\ref{xPsiminusf}) follows:
\begin{align}\label{EE2}
|\Lambda(x,k) - \Lambda_0(x,k)| \leq
\frac{C}{(1+x)^{n}}, \qquad x \geq 0,  \  k \in \bar{\C}_+^K.
\end{align}
In view of equations (\ref{EE1}) and (\ref{EE2}), we conclude that $[X]_2 = \Psi$ satisfies $(e)$ and $(f)$ for $j = 1$ and $k \in \bar{\C}_+^K$.

Proceeding inductively, we find that $\Lambda^{(j)} := \partial_k^j\Psi$ satisfies an integral equation of the form
\begin{align*}
\Lambda^{(j)}(x,k) = \Lambda_{0,j}(x,k)
- \int_x^\infty E(x,x',k) \mathsf{U}(x',k) \Lambda^{(j)}(x', k) dx',
\end{align*}
where
$$\big|\Lambda_{0,j}(x, k)\big| \leq
\frac{C}{(1+x)^{n-j}}, \qquad x \geq 0, \  k \in \bar{\C}_+^K.$$
If $1 \leq j \leq n$, then $\|\Lambda_{0,j}(\cdot, k)\|_{L^\infty([x,\infty))}$ is bounded for $x \geq 0$ and $k \in \bar{\C}_+^K$; hence the associated series $ \Lambda^{(j)} = \sum_{l=0}^\infty \Lambda^{(j)}_l$ converges uniformly on $[0,\infty) \times \bar{\C}_+^K$ to a bounded solution which satisfies $(e)$ and $(f)$ for $k \in \bar{\C}_+^K$.
Since $K > 0$ was arbitrary, this completes the proof of the theorem for $[X]_2$.

{\bf Construction of $[Y]_2$.}
We next use similar arguments applied to (\ref{XYdefb}) to construct $[Y]_2$. In this case, we define $\{\Psi_l\}_0^\infty$ by $\Psi_0 = \big(\begin{smallmatrix} 0 \\ 1 \end{smallmatrix}\big)$ and
\begin{align*}
\Psi_l(x,k) = & \int_{0 \leq x_1 \leq  \cdots \leq x_l \leq x_{l+1}=x < \infty} \prod_{i = 1}^l E(x_{i+1}, x_i, k) \mathsf{U}(x_i,k) \Psi_0 dx_1 \cdots dx_l, \qquad x \geq 0.
\end{align*}
As in (\ref{xPhilestimate}), we find for $x \geq 0$, $K^{-1} < |k| < K$, and $l=1,2,\dots$,
\begin{align}\nonumber
|\Psi_l(x,k)| \leq \frac{(Ce^{Cx})^l}{l!} \|\mathsf{U}(\cdot,k)\|_{L^1([0,x])}^l
\leq \frac{(Ce^{Cx})^l}{l!},
\end{align}
which leads to the following analog of (\ref{xPsiminusf}) valid for $(x,k)$ in a compact subset of  $[0,\infty) \times (\C \setminus \{0\})$:
\begin{align}\label{YPsiPsi0}
|\Psi(x,k) - \Psi_0| \leq \sum_{l=1}^\infty | \Psi_l(x,k)| \leq C.
\end{align}
This leads to properties $(b)$ and $(d)$.
Since $|E(x, x', k)| < C$ for $0 \leq x' \leq x < \infty$ and $k \in \bar{\C}_-^K$, we see that (\ref{YPsiPsi0}) holds uniformly also for $(x, k) \in [0,\infty) \times \bar{\C}_-^K$, which yields (\ref{Festb}) for $j = 0$. Letting
\begin{align*}
\Lambda_0(x,k) = \;& \int_0^x \partial_k\big[E(x,x',k) \mathsf{U}(x',k)\big] \Psi(x', k) dx',
\end{align*}
we find that $\Lambda := \partial_k\Psi$ satisfies
\begin{align*}
\Lambda(x,k) = \Lambda_0(x,k) + \int_0^x E(x,x',k) \mathsf{U}(x',k) \Lambda(x', k) dx'
\end{align*}
The proof of (\ref{Festb}) for $j = 1$ follows from the following analogs of (\ref{EE0})-(\ref{EE2}):
\begin{align*}
& |\Lambda_l(x,k)| \leq \frac{C^l}{l!}\|\Lambda_0(\cdot, k)\|_{L^\infty([0,x])}, \qquad
	\\
& |\Lambda_0(x,k)| \leq C \int_0^x (1 + |x-x'|) |u_0(x') +u_{0x}(x')+u_1(x')| dx'
	\\
&\qquad \qquad  \leq C(1+x)\int_0^x |u_0(x')+u_{0x}(x')+u_1(x')| dx'
 \leq C(1+x),
	\\ \nonumber
& |\Lambda(x,k) - \Lambda_0(x,k)| \leq C(1+x),
\end{align*}
which are valid for $x \geq 0$ and $k \in \bar{\C}_-^K$; the proof of (\ref{Festb}) for $j \geq 2$ is similar.

Finally, the estimate (\ref{Festc}) follows by applying similar arguments to the integral equation
$$y(x,k) = e^{-2i\theta_1 x \sigma_3} + \int_0^x e^{i\theta_1 (x'-x) \sigma_3}(\mathsf{U} y)(x',k) e^{i\theta_1 (x'-x) \sigma_3}dx'$$
satisfied by $y(x,k) := Y(x,k) e^{-2i\theta_1 x\sigma_3}$. 
\end{proof}

\begin{remark}\upshape
A function $f$ is absolutely continuous on an interval $I \subset \R$ if and only if $f \in W^{1,1}_{\text{loc}}(I)$, and in that case the classical derivative of $f$ exists and equals the weak derivative at a.e. $x \in I$. Thus, in addition to being weak solutions of (\ref{xpartF}) in the distributional sense, $X$ and $Y$ also satisfy (\ref{xpartF}) in the classical sense for a.e. $x \geq 0$.
\end{remark}

\subsection{Asymptotics as $k \to \infty$}
We next consider the behavior of the eigenfunctions $X$ and $Y$ as $k \to \infty$. Our goal is to prove Theorem \ref{xth2} which essentially states that the asymptotics of $X$ and $Y$ as $k \to \infty$ can be obtained by considering formal power series solutions of (\ref{xpartF}). Let us first find these formal solutions.

Equation (\ref{xpartF}) admits formal power series solutions $X_{formal}$ and $Y_{formal}$, normalized at $x = \infty$ and $x = 0$ respectively, such that
\begin{align}\label{Xformaldef}
& X_{formal}(x,k) = I + \frac{X_1(x)}{k} + \frac{X_2(x)}{k^2} + \cdots,
	\\ \nonumber
& Y_{formal}(x,k) =  I + \frac{Z_1(x)}{k} + \frac{Z_2(x)}{k^2} + \cdots + \bigg(\frac{W_1(x)}{k} + \frac{W_2(x)}{k^2} + \cdots \bigg) e^{2i\theta_1 x\sigma_3},
\end{align}
where the coefficients $\{X_j(x), Z_j(x), W_j(x)\}_1^\infty$ satisfy
\begin{align}\label{Fjnormalization}
\lim_{x\to \infty} X_j(x) = 0, \qquad Z_j(0) + W_j(0) = 0, \qquad j \geq 1.
\end{align}
Indeed, substituting
$$X = I + \frac{X_1(x)}{k} + \frac{X_2(x)}{k^2} + \cdots$$
into (\ref{xpartF}), the off-diagonal terms of $O(k^{-j})$ and the diagonal terms of $O(k^{-j-1})$ yield the relations
\begin{align}\label{xrecursive}
\begin{cases}
X_{j+1}^{(o)} = -2 i \sigma_3\big(-\partial_x X_{j}^{(o)}+\frac{i}{2}\sigma_3 X_{j-1}^{(o)}+\mathsf{U}_0 X_j^{(d)}+\mathsf{U}_1^{(o)} X_{j-1}^{(d)}+\mathsf{U}_1^{(d)} X_{j-1}^{(o)}\big),
	\\
\partial_x X_{j+1}^{(d)} = \mathsf{U}_0 X_{j+1}^{(o)}+\mathsf{U}_1^{(o)} X_j^{(o)}+\mathsf{U}_1^{(d)} X_j^{(d)},
\end{cases}
\end{align}
where $A^{(d)}$ and $A^{(o)}$ denote the diagonal and off-diagonal parts of a $2 \times 2$ matrix $A$, respectively.

The coefficients $\{Z_j\}$ satisfy the equations obtained by replacing $\{X_j\}$ with $\{Z_j\}$ in \eqref{xrecursive}.
Similarly, substituting
$$X = \bigg(\frac{W_1(x)}{k} + \frac{W_2(x)}{k^2} + \cdots\bigg) e^{2i\theta_1 x \sigma_3}$$
into (\ref{xpartF}), the diagonal terms of $O(k^{-j})$ and the off-diagonal terms of $O(k^{-j-1})$ yield the relations
\begin{align}\label{xrecursive2}
\begin{cases}
W_{j+1}^{(d)} = -2 i \sigma_3\big(- \partial_xW_{j}^{(d)}+\frac{i}{2}\sigma_3 W_{j-1}^{(d)}+\mathsf{U}_0 W_j^{(o)}+\mathsf{U}_1^{(o)} W_{j-1}^{(o)}+\mathsf{U}_1^{(d)} W_{j-1}^{(d)}\big),
	\\
\partial_x W_{j+1}^{(o)} = \mathsf{U}_0 W_{j+1}^{(d)}+\mathsf{U}_1^{(o)} W_j^{(d)}+\mathsf{U}_1^{(d)} W_j^{(o)}.
\end{cases}
\end{align}
The coefficients $\{X_j(x), Z_j(x), W_j(x)\}$ are determined recursively from (\ref{Fjnormalization})-(\ref{xrecursive2}), the equations obtained by replacing $\{X_j\}$ with $\{Z_j\}$ in (\ref{xrecursive}), and the initial assignments
$$X_{-1} = 0, \qquad X_0 = I, \qquad Z_{-1} = 0, \qquad Z_0 = I, \qquad W_{-1}=W_0 = 0.$$
The first few coefficients are given by
\begin{align}\label{x1}
X_1(x) = &\; \frac {i(u_{0x}+u_1)}{2} \sigma_1+\sigma_3 \frac{i}{4} \int_{\infty}^{x}\bigg[-\frac{(u_{0x}+u_1)^2}{2}+\cos u_0-1 \bigg]dx',
	\\ \nonumber
X_2(x) = &\; i \bigg[-u_{0xx}-u_{1x}+\frac{i}{2}(u_{0x}+u_1)(X_1)_{22} + \frac{1}{2} \sin u_0\bigg]\sigma_2
	\\ \nonumber
 &\;+ I\int_{\infty}^x \bigg\{- \frac{i}{4}\Big[-\frac{(u_{0x} + u_1)^2}{2} + \cos{u_0} - 1\Big](X_1)_{22}
 	\\\nonumber
& - \frac{1}{4}(u_{0x} + u_1)(u_{0xx} + u_{1x})\bigg\} dx',
	\\ \label{z1}
Z_1(x) = &\; \frac {i(u_{0x}+u_1)}{2} \sigma_1+\sigma_3 \frac{i}{4} \int_{0}^{x}\bigg[-\frac{(u_{0x}+u_1)^2}{2}+\cos u_0-1 \bigg]dx',
	\\ \nonumber
Z_2(x) = &\; i \bigg[-u_{0xx}-u_{1x}+\frac{i}{2}(u_{0x}+u_1)(Z_1)_{22} + \frac{1}{2} \sin u_0\bigg]\sigma_2\\\nonumber
 &\;+ I\bigg\{\int_{0}^x \bigg[- \frac{i}{4}\Big[-\frac{(u_{0x} + u_1)^2}{2} + \cos{u_0} - 1\Big](Z_1)_{22}
 	\\\nonumber
& - \frac{1}{4}(u_{0x} + u_1)(u_{0xx} + u_{1x})\bigg] dx'-\frac{1}{4} (u_{0x}(0)+u_1(0))^2\bigg\},
\end{align}
and
\begin{align}\label{w1}
W_1(x) = & -\frac{i(u_{0x}(0)+u_1(0))}{2}\sigma_1,
	\\ \nonumber
W_2(x) = &\; \frac{i}{2}(u_{0x}+u_1) (W_1)_{12} I +\sigma_2 \bigg\{\int_0^x \frac{1}{4}\bigg[\frac{(u_{0x}+u_1)^2}{2} - \cos u_0 +1\bigg](W_1)_{12}  dx'\\ \nonumber &\; -i\Big[-(u_{0xx}(0)+u_{1x}(0))+\frac{1}{2} \sin u_0(0)\Big]\bigg\}.
\end{align}

We can now describe the behavior of $X$ and $Y$ for large $k$.

\begin{theorem}[Asymptotics of $X$ and $Y$ as $k \to \infty$]\label{xth2}
Let $m \geq 1$, $n \geq 1$, and $N_x \in \Z$ be integers.  
Let $u_0(x)$ and $u_1(x)$ be functions satisfying the following regularity and decay assumptions:\footnote{It is implicit in the assumption that $u_0(x)$ and $u_1(x)$  are $m+2$ and $m+1$ times weakly differentiable, respectively.} 
\begin{align}\label{ujassump}
\begin{cases}
(1+x)^{n} (u_0(x) - 2\pi N_x) \in L^1([0,\infty)),
	\\
(1+x)^{n}\partial^i u_0(x) \in L^1([0,\infty)), \qquad  i = 1, \dots, m+2,
	\\
(1+x)^{n}\partial^i u_1(x) \in L^1([0,\infty)), \qquad  i = 0,1, \dots, m+1,
\end{cases}
\end{align}

As $k \to \infty$, $X$ and $Y$ coincide to order $m$ with $X_{formal}$ and $Y_{formal}$, respectively, in the following sense: The functions
\begin{align}\nonumber
&X_p(x,k) := I + \frac{X_1(x)}{k} + \cdots + \frac{X_{m+1}(x)}{k^{m+1}},
	\\ \nonumber
&Y_p(x,k) = I + \frac{Z_1(x)}{k} + \cdots + \frac{Z_{m+1}(x)}{k^{m+1}}
 + \bigg(\frac{W_1(x)}{k} + \cdots + \frac{W_{m+1}(x)}{k^{m+1}}\bigg) e^{2i\theta_1 x\sigma_3},
\end{align}
are well-defined and, for each $j = 0, 1, \dots, n,$
\begin{subequations}\label{varphiasymptotics}
\begin{align}\label{varphiasymptoticsa}
& \bigg|\frac{\partial^j}{\partial k^j}\big(X - X_p\big) \bigg| \leq
\frac{C}{|k|^{m+1}(1+x)^{n-j}}, \qquad x \geq 0, \   k \in (\bar{D}_4, \bar{D}_1),
	\\ \label{varphiasymptoticsb}
& \bigg|\frac{\partial^j}{\partial k^j}\big(Y - Y_p\big) \bigg| \leq
\frac{C(1 + x)^{j+2}e^{\frac{Cx}{|k|^{m+1}}}}{|k|^{m+1}}, \qquad x \geq 0, \  k \in (\bar{D}_1, \bar{D}_4),
	\\ \label{varphiasymptoticsc}
& \bigg|\frac{\partial^j}{\partial k^j}\big((Y - Y_p)e^{-2i\theta_1 x\sigma_3}\big) \bigg| \leq
\frac{C(1 + x)^{j+2}e^{\frac{Cx}{|k|^{m+1}}}}{|k|^{m+1}}, \quad x \geq 0, \ k \in (\bar{D}_4, \bar{D}_1).
\end{align}
\end{subequations}

\end{theorem}
\begin{proof}
The proof of (\ref{varphiasymptotics}) is based on an analysis of the integral equations (\ref{XYdef}) in the limit $k \to \infty$. 
The proof proceeds through a series of lemmas.

\begin{lemma} \label{lemma1}
 $\{X_j(x)\}_1^{m+1}$ are weakly differentiable functions of $x \geq 0$ satisfying
\begin{align}\label{Fjbounds}
\begin{cases}
(1 + x)^n X_j(x) \in L^1([0,\infty)) \cap L^\infty([0,\infty)), \\
(1 + x)^n X_j'(x) \in L^1([0,\infty)),
\end{cases} \qquad j = 1, \dots, m+1.
\end{align}
\end{lemma}
\stepproofbegin
The assumption (\ref{ujassump}) implies
\begin{align}\label{uibound}
\begin{cases}
 |u_0^{(i)}(x)| \leq \frac{C}{(1 + x)^n}, \qquad i = 0, 1, \dots, m+1,
	\\
  |u_1^{(i)}(x)| \leq \frac{C}{(1 + x)^n}, \qquad i = 0, 1, \dots, m,
\end{cases} \ x \geq 0.
\end{align}
Indeed, if $i = 0, 1, \dots, m+1$ and $x \geq 0$, then 
\begin{align}\label{1xuiest}
\big|(1 + x)^n u_0^{(i)}(x)\big| \leq |u_0^{(i)}(0)| + \int_0^x \big|n(1 + x')^{n-1}u_0^{(i)}(x') + (1+x')^n u_0^{(i+1)}(x')\big|dx' < C,
\end{align}
and the proof for $u_1^{(i)}$ is similar.

Let $S_j$ refer to the statement
$$
\begin{cases}
\text{$(1 + x)^n \partial^i X_j^{(o)} \in L^1([0,\infty))$ for $i = 0,1, \dots, m+2-j$,}
	\\
\text{$(1 + x)^n \partial^i X_j^{(d)} \in L^1([0,\infty))$ for $i = 0,1, \dots, m+3-j$.}
\end{cases}$$
By (\ref{ujassump}) and (\ref{uibound}),
\begin{align*}
\lim_{x\to \infty} (1 + x)^{n+1} \int_x^\infty &[(u_{0x}(x')+u_{1}(x'))^2+|\cos u_0(x')-1|] dx' = 0,
\end{align*}
so an integration by parts yields
\begin{align}\label{xF1dest}
& \|(1 + x)^n X_1^{(d)}\|_{L^1([0,\infty))} < \infty.
 \end{align}
The bound (\ref{xF1dest}) together with the expression (\ref{x1}) for $X_1$ shows that $S_1$ holds.
Similar estimates together with the relations (\ref{xrecursive}) imply that if $1 \leq j \leq m$ and $S_j$ holds, then $S_{j+1}$ also holds. Thus, by induction, $\{S_j\}_{j=1}^{m+1}$ hold. This shows that $\{X_j\}_1^{m+1}$ are weakly differentiable functions satisfying $(1+x)^n\partial^iX_j(x) \in L^1([0,\infty))$ for $i =0,1$ and $j = 1, \dots, m+1$. The boundedness of $(1+x)^nX_j(x)$ follows by an estimate analogous to (\ref{1xuiest}).
\proofendcontinue

\begin{lemma}\label{lemma2}
There exists a $K> 1$ such that $X_p(x,k)^{-1}$ exists for all $k \in \C$ with $|k|\geq K$.
Moreover, letting $A := -i\theta_1 \sigma_3 + \mathsf{U}$ and
\begin{align}\label{Apdef}
A_p(x,k) := \big(\partial_x X_{p}(x,k) - i \theta_1 X_p(x,k) \sigma_3\big)X_p(x,k)^{-1}, \qquad x \geq 0, \  |k|\geq K,
\end{align}
the difference $\Delta(x,k) := A(x,k) - A_p(x,k)$ satisfies
\begin{align}\label{DeltaEbounda}
\big|\partial_k^j\Delta(x,k)\big| \leq \frac{Cf(x)}{|k|^{m+1}(1+x)^n}, \qquad x \geq 0, \  |k| \geq K, \  j = 0,1, \dots, n,
\end{align}
where $f$ is a function in $L^1([0,\infty))$.
In particular,
\begin{align}\label{ABsmall}
\|\partial_k^j\Delta(\cdot, k)\|_{L^1([x,\infty))} \leq \frac{C}{|k|^{m+1}(1+x)^n}, \qquad x \geq 0,  \  |k|\geq K, \ j = 0, 1, \dots, n.
\end{align}
\end{lemma}
\stepproofbegin
By Lemma \ref{lemma1}, there exists a bounded function $g \in L^1([0,\infty))$ such that
\begin{align}\label{Fjxbound}
|X_j(x)| \leq \frac{g(x)}{(m+1)(1+x)^n}, \qquad x \geq 0, \  j = 1, \dots, m+1.
\end{align}
In particular,
$$\bigg| \sum_{j=1}^{m+1} \frac{X_j(x)}{k^j}\bigg| \leq \frac{g(x)}{|k|(1+x)^n}, \qquad x \geq 0, \  |k| \geq 1.$$
Choose $K > \max(1, \|g\|_{L^\infty([0,\infty))})$. Then $X_p(x,k)^{-1}$ exists whenever $|k| \geq K$ and is given by the absolutely and uniformly convergent Neumann series
$$X_p(x,k)^{-1} = \sum_{l =0}^\infty \bigg(-\sum_{j=1}^{m+1} \frac{X_j(x)}{k^j}\bigg)^l, \qquad x \geq 0,\  |k| \geq K.$$
Furthermore,
\begin{align}\label{tailestimate}
\bigg|\sum_{l = m+2}^\infty \bigg(-\sum_{j=1}^{m+1} \frac{X_j(x)}{k^j}\bigg)^l\bigg|
\leq \sum_{l = m+2}^\infty \bigg(\frac{g(x)}{|k|(1+x)^n}\bigg)^l
\leq \frac{Cg(x)}{|k|^{m+2}(1+x)^n},
\end{align}
for $x \geq 0$ and $|k| \geq K$. Now let $\Phi_0(x) + \frac{\Phi_1(x)}{k} + \frac{\Phi_2(x)}{k^2} + \cdots$ be the formal power series expansion of $X_p(x,k)^{-1}$ as $k \to \infty$, i.e.
\begin{align*}
& \Phi_0(x) = I, \quad \Phi_1(x) = -X_1(x), \quad \Phi_2(x) = X_1(x)^2 - X_2(x),
	\\
& \Phi_3(x) = X_1(x)X_2(x) + X_2(x) X_1(x) - X_1(x)^3 - X_3(x), \quad \dots.
\end{align*}
Equation (\ref{Fjxbound}) and the inequality (\ref{tailestimate}) imply that the function $\mathcal{E}(x,k)$ defined by
$$\mathcal{E}(x,k) = X_p(x,k)^{-1} - \sum_{j=0}^{m+1} \frac{\Phi_j(x)}{k^j}$$
satisfies
\begin{align}\label{Eestimate}
|\mathcal{E}(x,k)| \leq \frac{Cg(x)}{|k|^{m+2}(1+x)^n}, \qquad x \geq 0, \  |k| \geq K.
\end{align}

Let $A_p(x,k)$ be given by (\ref{Apdef}).
Since $X_{formal}$ is a formal solution of (\ref{xpartF}), the coefficient of $k^{-j}$ in the formal expansion of $\Delta = A - A_p$ as $k \to \infty$  vanishes for $j \leq m$; hence, in view of Lemma \ref{lemma1} and (\ref{Eestimate}),
\begin{align}\nonumber
|\Delta|
& = \bigg|A - (\partial_xX_{p} - i \theta_1 X_p \sigma_3)\bigg(\sum_{j=0}^{m+1} \frac{\Phi_j}{k^j} + \mathcal{E}\bigg)\bigg|
	\\\nonumber
& \leq \frac{Cf(x)}{|k|^{m+1}(1+x)^n} + \bigg|(\partial_xX_{p} - i\theta_1 X_p \sigma_3)\mathcal{E}\bigg|
	\\ \nonumber
& \leq \frac{Cf(x)}{|k|^{m+1}(1+x)^n}, \qquad x \geq 0, \  |k| \geq K,
\end{align}
where $f$ is a function in $L^1([0,\infty))$. This proves (\ref{DeltaEbounda}) for $j = 0$. Differentiating $\Delta$ with respect to $k$ and estimating in a similar way, we obtain (\ref{DeltaEbounda}) also for $j = 1, \dots, n$. 
\proofendcontinue

In what follows, we let $K > 1$ be the constant from Lemma \ref{lemma2} and note that it is enough to prove the inequalities in (\ref{varphiasymptotics}) for $|k| \geq K$.

\begin{lemma}
$[X]_2$ satisfies (\ref{varphiasymptoticsa}) for $j = 0$.
\end{lemma}
\stepproofbegin
Using that $X(x,k)$ satisfies (\ref{xpartF}), we compute
\begin{align*}
  (X_p^{-1}X)_x & = - X_p^{-1}(\partial_xX_p) X_p^{-1} X + X_p^{-1} \partial_xX
  	\\
&  = - X_p^{-1}(A_{p}X_p +i\theta_1X_p\sigma_3) X_p^{-1} X + X_p^{-1} (AX +i\theta_1X\sigma_3)
	\\
&  = X_p^{-1}\Delta X - i\theta_1[\sigma_3, X_p^{-1}X],
\end{align*}
and hence
$$\Big(e^{i\theta_1x\hat{\sigma}_3}(X_p^{-1}X)\Big)_x = e^{i\theta_1x\hat{\sigma}_3}(X_p^{-1} \Delta X).$$
Integrating and using that $X \to I$ as $x \to \infty$, we conclude that $X$ satisfies the Volterra equation
\begin{align}\label{hatXhatXp}
X(x,k) = X_p(x,k) - \int_x^\infty X_p(x,k) e^{i\theta_1(x'-x)\hat{\sigma}_3} (X_p^{-1}\Delta X)(x',k) dx'.
\end{align}
Letting $\Psi = [X]_2$ and $\Psi_0(x,k) =  [X_p(x,k)]_2$, we can write the second column of (\ref{hatXhatXp}) as
\begin{align} \label{integraleq}
\Psi(x,k) = \Psi_0(x,k) - \int_x^\infty E(x,x',k) \Delta(x', k) \Psi(x', k) dx', \qquad x \geq 0, \ k \in \bar{\C}_+^K,
\end{align}
where
\begin{align}\label{E1E1E1def}
& E(x,x',k) =  X_p(x,k) \begin{pmatrix} e^{2i\theta_1(x' - x)} & 0 \\ 0 & 1 \end{pmatrix} X_p(x',k)^{-1}.
\end{align}
Define $\Psi_l$ for $l \geq 1$ inductively by
\begin{align*}
\Psi_{l+1}(x,k) = - \int_x^\infty E(x,x',k) \Delta(x', k) \Psi_l(x', k) dx', \qquad x \geq 0, \  k \in \bar{\C}_+^K.
\end{align*}
Then
\begin{align}\label{Philiterated}
\Psi_l(x,k) = & (-1)^l \int_{x = x_{l+1} \leq x_l \leq \cdots \leq x_1 < \infty} \prod_{i = 1}^l E(x_{i+1}, x_i, k) \Delta(x_i, k) \Psi_0(x_1, k) dx_1 \cdots dx_l.
\end{align}
The estimates
\begin{align}\label{DeltaEboundb}
  |\partial_k^jE(x, x', k)| < C(1+|x' - x|)^j, \qquad 0 \leq x \leq x' < \infty, \  k \in \bar{\C}_+^K,
\end{align}
and (\ref{ABsmall}) with $j = 0$, give, for $x \geq 0$, $k \in \bar{\C}_+^K$, and $l = 1, 2, \dots$,
\begin{align}\nonumber
|\Psi_l(x,k)| \leq \; & C^l \int_{x \leq x_l \leq \cdots \leq x_1 < \infty} \prod_{i = 1}^l  |\Delta(x_i, k)| |\Psi_0(x_1, k)| dx_1 \cdots dx_l
	\\ \nonumber
\leq & \; \frac{C^l}{l!}\|\Psi_0(\cdot, k)\|_{L^\infty([x,\infty))} \|\Delta(\cdot, k)\|_{L^1([x,\infty))}^l
	\\ \label{Philestimate}
\leq & \; \frac{C^l}{l!} \|\Psi_0(\cdot, k)\|_{L^\infty([x,\infty))}\bigg(\frac{C}{|k|^{m+1}(1+x)^n}\bigg)^l.
\end{align}
This implies
$$|\Psi_l(x,k)| \leq \frac{C^l}{l!}\bigg(\frac{C}{|k|^{m+1}(1+x)^n}\bigg)^l, \qquad x \geq 0, \  k \in \bar{\C}_+^K, \ l = 1, 2, \dots,$$
because, by Lemma \ref{lemma1},
$$|\Psi_0(x, k)| < C, \qquad x \geq 0, \ k \in \bar{\C}_+^K.$$
Hence the series $\sum_{l=0}^\infty \Psi_l(x,k)$
converges absolutely and uniformly for $x \geq 0$ and $k \in \bar{\C}_+^K$ to the solution $\Psi(x,k)$ of (\ref{integraleq}). Since
\begin{align}\label{Psiminusf}
|\Psi(x,k) - \Psi_0(x,k)| \leq \sum_{l=1}^\infty | \Psi_l(x,k)| \leq \frac{C}{|k|^{m+1}(1+x)^n}, \qquad x \geq 0, \  k \in \bar{\C}_+^K,
\end{align}
this proves (\ref{varphiasymptoticsa}) for $j = 0$.
\proofendcontinue

\begin{lemma}
$[X]_2$ satisfies (\ref{varphiasymptoticsa}) for $j = 1, \dots, n$.
\end{lemma}
\stepproofbegin
Let
\begin{align}\label{Lambda0}
\Lambda_0(x,k) = \;& [\partial_kX_p(x,k)]_2
- \int_x^\infty \frac{\partial}{\partial k}\big[E(x,x',k) \Delta(x', k)\big] \Psi(x', k) dx'.
\end{align}
Differentiating the integral equation (\ref{integraleq}) with respect to $k$, we find that $\Lambda := \partial_k\Psi$ satisfies
\begin{align}\label{integraleq3}
\Lambda(x,k) = \Lambda_0(x,k)
- \int_x^\infty E(x,x',k) \Delta(x', k) \Lambda(x', k) dx'
\end{align}
for each $k$ in the interior of $\bar{\C}_+^K$; the differentiation can be justified by dominated convergence using (\ref{DeltaEbounda}), (\ref{DeltaEboundb}), and a Cauchy estimate for $\partial_k\Psi$.
We seek a solution of (\ref{integraleq3}) of the form $\Lambda = \sum_{l=0}^\infty \Lambda_l$,  where the $\Lambda_l$ are defined by replacing $\{\Psi_l\}$ by $\{\Lambda_l\}$ in (\ref{Philiterated}). Proceeding as in (\ref{Philestimate}), we find
$$|\Lambda_l(x,k)| \leq \frac{C^l}{l!}\|\Lambda_0(\cdot, k)\|_{L^\infty([x,\infty))}\bigg(\frac{C}{|k|^{m+1}(1+x)^n}\bigg)^l, \qquad x \geq 0, \  k \in \bar{\C}_+^K.$$
Using (\ref{DeltaEbounda}), (\ref{DeltaEboundb}), and (\ref{Psiminusf}) in (\ref{Lambda0}), we obtain
\begin{align}\label{FF1}
|\Lambda_0(x,k) - [\partial_kX_p(x,k)]_2 | \leq \frac{C}{|k|^{m+1}(1+x)^{n-1}}, \qquad x \geq 0, \  k \in \bar{\C}_+^K.
\end{align}
In particular $\|\Lambda_0(\cdot, k)\|_{L^\infty([x,\infty))}$ is bounded for $k \in \bar{\C}_+^K$ and $x \geq 0$.
Thus, $\sum_{l=0}^\infty \Lambda_l$ converges uniformly on $[0,\infty) \times \bar{\C}_+^K$ to a solution $\Lambda$ of (\ref{integraleq3}), which satisfies the following analog of (\ref{Psiminusf}):
\begin{align}\label{FF2}
|\Lambda(x,k) - \Lambda_0(x,k)| \leq
\frac{C}{|k|^{m+1}(1+x)^{n}}, \qquad x \geq 0,  \  k \in \bar{\C}_+^K.
\end{align}
Equations (\ref{FF1}) and (\ref{FF2}) show that $[X]_2 = \Psi$ satisfies (\ref{varphiasymptoticsa}) for $j = 1$.

Proceeding inductively, we find that $\Lambda^{(j)} := \partial_k^j\Psi$ satisfies an integral equation of the form
\begin{align*}
\Lambda^{(j)}(x,k) = \Lambda_{0,j}(x,k)
- \int_x^\infty E(x,x',k) \Delta(x', k) \Lambda^{(j)}(x', k) dx',\qquad j = 1, \dots, n,
\end{align*}
where
$$\big|\Lambda_{0,j}(x, k) - [\partial_k^{j} X_p(x, k)]_2 \big| \leq
\frac{C}{|k|^{m+1}(1+x)^{n-j}}, \qquad x \geq 0, \  k \in \bar{\C}_+^K.$$
If $1 \leq j \leq n$, then $\|\Lambda_{0,j}(\cdot, k)\|_{L^\infty([x,\infty))}$ is bounded for $k \in \bar{\C}_+^K$ and $x \geq 0$; hence the associated series $ \Lambda^{(j)} = \sum_{l=0}^\infty \Lambda^{(j)}_l$ converges uniformly on $[0,\infty) \times \bar{\C}_+^K$ to a bounded solution with the desired properties.
\proofendcontinue

The above lemmas prove the theorem for $X$. We now consider $[Y]_2$.

\begin{lemma}\label{lemma5}
$\{Z_j(x), W_j(x)\}_1^{m+1}$ are weakly differentiable functions of $x \geq 0$ satisfying
\begin{align}\label{ZjWj}
\begin{cases}
Z_j, W_j \in L^\infty([0,\infty)), \\
Z_j', W_j' \in L^1([0,\infty)),
\end{cases} \qquad j = 1, \dots, m+1.
\end{align}
\end{lemma}
\stepproofbegin
Let $S_j$ and $\mathcal{S}_j$ refer to the statements
$$
\begin{cases}
\text{$\partial^i Z_j^{(o)} \in L^1([0,\infty))$ for $i = 1, \dots, m+2-j$,}
	\\
\text{$\partial^i Z_j^{(d)} \in L^1([0,\infty))$ for $i = 1, \dots, m+3-j$,}
\end{cases}$$
and
$$\begin{cases}
\text{$\partial^i W_j^{(d)} \in L^1([0,\infty))$ for $i = 1, \dots, m+2-j$,}
	\\
\text{$\partial^i W_j^{(o)} \in L^1([0,\infty))$ for $i = 1, \dots, m+3-j$,}
\end{cases}$$
respectively.
Using the expressions (\ref{z1}) and (\ref{w1}) for $Z_1$ and $W_1$, we conclude that $S_1$ and $\mathcal{S}_1$ hold. The relations (\ref{xrecursive}) and (\ref{xrecursive2}) imply by induction that $\{S_j, \mathcal{S}_j\}_{j=1}^{m+1}$ hold. This shows that $\{Z_j, W_j\}_1^{m+1}$ are weakly differentiable functions satisfying $Z_j', W_j' \in L^1([0,\infty))$ for $j = 1, \dots, m+1$. Integration shows that $Z_j, W_j \in L^\infty([0,\infty))$ for $j = 1, \dots, m+1$. \proofendcontinue

We write
$$Y_p(x,k) = Z_p(x,k) + W_p(x,k) e^{2i\theta_1 x\sigma_3},$$
where $Z_p$ and $W_p$ are defined by
$$Z_p(x,k) = I + \frac{Z_1(x)}{k} + \cdots + \frac{Z_{m+1}(x)}{k^{m+1}}, \qquad
W_p(x,k) = \frac{W_1(k)}{k} + \cdots + \frac{W_{m+1}(x)}{k^{m+1}}.$$

\begin{lemma}\label{lemma6}
There exists a $K > 1$ such that $Z_p(x,k)^{-1}$ and $W_p(x,k)^{-1}$ exist for all $k \in \C$ with $|k|\geq K$.
Moreover, letting $A :=-i\theta_1 \sigma_3 + \mathsf{U} $ and
\begin{align*}
\begin{cases}
A_{p,1}(x,k) := \big(\partial_xZ_{p}(x,k) - i\theta_1 Z_p(x,k) \sigma_3\big)Z_p(x,k)^{-1},
	\\
A_{p,2}(x,k) := \big(\partial_xW_{p}(x,k) + i\theta_1 W_p(x,k) \sigma_3\big)W_p(x,k)^{-1},
\end{cases}
\quad x \geq 0, \  |k| \geq K,
\end{align*}
the differences
$$\Delta_l(x,k) := A(x,k) - A_{p,l}(x,k), \qquad l = 1,2,$$
satisfy
\begin{align}\label{EEDeltaest1}
|\partial_k^j\Delta_l(x,k)| \leq \frac{C + f(x)}{|k|^{m+1}}, \qquad x \geq 0, \  |k| \geq K, \  j = 0, 1, \dots, n, \  l = 1,2,
\end{align}
where $f$ is a function in $L^1([0,\infty))$.
In particular,
\begin{align}\label{EEDeltaest}
\|\partial_k^j\Delta_l(\cdot, k)\|_{L^1([0,x])} \leq \frac{Cx}{|k|^{m+1}}, \qquad x \geq 0, \  |k| \geq K,  \  j = 0, 1, \dots, n, \  l = 1,2.
\end{align}
\end{lemma}
\stepproofbegin
The proof uses Lemma \ref{lemma5} and is similar to that of Lemma \ref{lemma2}.
\proofendcontinue

In the remainder of the proof, $K > 1$ will denote the constant from Lemma \ref{lemma6}.

\begin{lemma}
We have
\begin{subequations}
\begin{align}\label{ZYWa}
& \bigg| \frac{\partial^j}{\partial k^j}\Big[Z_p(x,k) e^{-i\theta_1x\hat{\sigma}_3} Z_p^{-1}(0,k) - Y_p(x,k)\Big]_2 \bigg| \leq C(1+x)^j\frac{1 + |e^{-2i\theta_1x}|}{|k|^{m+2}},
	\\ \label{ZYWb}
& \bigg| \frac{\partial^j}{\partial k^j}\Big[W_p(x,k) e^{i\theta_1x\hat{\sigma}_3} W_p^{-1}(0,k) - Y_p(x,k)e^{-2i\theta_1x\sigma_3}\Big]_2 \bigg| \leq C(1+x)^j\frac{1 + |e^{2i\theta_1x}|}{|k|^{m+2}},
\end{align}
for all $x \geq 0$, $|k|\geq K$, and $j = 0, 1, \dots, n$.
\end{subequations}
\end{lemma}
\stepproofbegin
We write
$$Y_{formal}(x,k) = Z_{formal}(x,k) + W_{formal}(x,k)e^{2i\theta_1x\sigma_3},$$
where
$$Z_{formal}(x,k) = I + \frac{Z_1(x)}{k} + \frac{Z_2(x)}{k^2} + \cdots, \qquad
W_{formal}(x,k) = \frac{W_1(x)}{k} + \frac{W_2(x)}{k^2} + \cdots.$$
Then
\begin{align}\label{ZZY}
Z_{formal}(x,k) e^{-i\theta_1x\hat{\sigma}_3} Z_{formal}^{-1}(0,k) = Y_{formal}(x,k)
\end{align}
formally to all orders in $k$ as $k \to \infty$. Indeed, both sides of (\ref{ZZY}) are formal solutions of (\ref{xpartF}) satisfying the same initial condition at $x = 0$. Truncating (\ref{ZZY}) at order $k^{-m-1}$, using Lemma \ref{lemma5}, and estimating the inverse $Z_p^{-1}(0,k)$ as in the proof of Lemma \ref{lemma2}, it follows that
$$Z_p(x,k) e^{-i\theta_1x\hat{\sigma}_3} Z_p^{-1}(0,k) = Y_p(x,k) + O(k^{-m-2}) + O(k^{-m-2})e^{2i\theta_1x\sigma_3}.$$
This gives (\ref{ZYWa}) for $j = 0$. Using the estimate $|\partial_k e^{\pm 2i\theta_1x}| \leq C |x e^{\pm 2i\theta_1x}|$ for $|k| \geq K$, we find (\ref{ZYWa}) also for $j \geq 1$.

Similarly, we have
$$W_{formal}(x,k) e^{i\theta_1x\hat{\sigma}_3} W_{formal}^{-1}(0,k) = Y_{formal}(x,k)e^{-2i\theta_1x\sigma_3}$$
to all orders in $k$ and truncation leads to (\ref{ZYWb}).
\proofendcontinue

\begin{lemma}\label{lemma8}
$[Y]_2$ satisfies (\ref{varphiasymptoticsb}).
\end{lemma}
\stepproofbegin
Using that $Y(x,k)$ satisfies (\ref{xpartF}), we compute
\begin{align}\nonumber
  (Z_p^{-1}Y)_x & = - Z_p^{-1}(\partial_xZ_p) Z_p^{-1} Y + Z_p^{-1} \partial_xY
  	\\\nonumber
&  = - Z_p^{-1}(A_{p,1}Z_p +i\theta_1Z_p\sigma_3) Z_p^{-1} Y + Z_p^{-1} (AY +i\theta_1Y\sigma_3)
	\\\nonumber
&  = Z_p^{-1}\Delta_1 Y - i\theta_1[\sigma_3, Z_p^{-1}Y].
\end{align}
Hence
$$\Big(e^{i\theta_1x\hat{\sigma}_3}(Z_p^{-1}Y)\Big)_x = e^{i\theta_1x\hat{\sigma}_3}(Z_p^{-1} \Delta_1 Y).$$
Integrating and using the initial condition $Y(0,k) = I$, we conclude that $Y$ satisfies the Volterra integral equation
\begin{align}\label{EEYVolterra}
Y(x,k) = Z_p(x,k) e^{-i\theta_1x\hat{\sigma}_3}Z_p^{-1}(0,k) + \int_0^x Z_p(x,k) e^{-i\theta_1(x-x')\hat{\sigma}_3} (Z_p^{-1}\Delta_1 Y)(x',k) dx'.
\end{align}
Letting $\Psi = [Y]_2$ and $\Psi_0(x,k) = [Z_p(x,k) e^{-i\theta_1x\hat{\sigma}_3}Z_p^{-1}(0,k) ]_2$, we can write the second column of (\ref{EEYVolterra}) as
\begin{align}\label{EEYVolterra2}
\Psi(x,k) = \Psi_0(x,k) + \int_0^x E(x,x',k) (\Delta_1 \Psi)(x',k) dx',
\end{align}
where
$$E(x,x',k) = Z_p(x,k) \begin{pmatrix} e^{-2i\theta_1(x-x')} & 0 \\ 0 & 1 \end{pmatrix} Z_p^{-1}(x',k).$$
We seek a solution $\Psi(x,k) = \sum_{l=0}^\infty \Psi_l(x,k)$ where
$$\Psi_l(x,k) = \int_{0 \leq x_1 \leq \cdots \leq x_l \leq x_{l+1}=x < \infty} \prod_{i = 1}^l E(x_{i+1}, x_i, k) \Delta_1(x_i, k) \Psi_0(x_1, k) dx_1 \cdots dx_l.$$
The estimates
\begin{align}\label{EEDeltaEboundb}
|\partial_k^jE(x, x', k)| < C(1+|x' - x|)^j, \qquad 0 \leq x' \leq x < \infty, \  k \in \bar{\C}_-^K, \  j = 0, 1, \dots, n,
\end{align}
and
$$|\Psi_0(x,k)| \leq C, \qquad x \geq 0, \  k \in \bar{\C}_-^K,$$
together with  (\ref{EEDeltaest}) yield
\begin{align*}\nonumber
|\Psi_l(x,k)| \leq \; & C^l \int_{0 \leq x_1 \leq \cdots \leq x_l \leq x < \infty} \prod_{i = 1}^l  |\Delta_1(x_i, k)| |\Psi_0(x_1, k)| dx_1 \cdots dx_l
	\\ \nonumber
\leq & \; \frac{C^l}{l!}\|\Psi_0(\cdot, k)\|_{L^\infty([0,x])} \|\Delta_1(\cdot, k)\|_{L^1([0,x])}^l
	\\
\leq &\; \frac{C^l}{l!} \bigg(\frac{Cx}{|k|^{m+1}}\bigg)^l, \qquad x \geq 0, \  k \in \bar{\C}_-^K.
\end{align*}
Hence
\begin{align}\label{EEPsiPsi0}
|\Psi(x,k) - \Psi_0(x,k)| \leq \sum_{l=1}^\infty | \Psi_l(x,k)| \leq \frac{Cx e^{\frac{Cx}{|k|^{m+1}}}}{|k|^{m+1}}, \qquad x \geq 0, \  k \in \bar{\C}_-^K.
\end{align}
Equations (\ref{ZYWa}) and (\ref{EEPsiPsi0}) prove the second column of (\ref{varphiasymptoticsb}) for $j = 0$.

Differentiating the integral equation (\ref{EEYVolterra2}) with respect to $k$, we find that $\Lambda := \partial_k\Psi$ satisfies
\begin{align}\label{EEintegraleq3}
\Lambda(x,k) = \Lambda_0(x,k)
+ \int_0^x E(x,x',k) \Delta_1(x', k) \Lambda(x', k) dx'
\end{align}
for each $k$ in the interior of $\bar{\C}_-^K$, where
\begin{align}\label{EELambda0}
\Lambda_0(x,k) = \;& [\partial_k\Psi_0(x,k)]_2
+ \int_0^x \frac{\partial}{\partial k}\big[E(x,x',k) \Delta_1(x', k)\big] \Psi(x', k) dx'.
\end{align}
We seek a solution of (\ref{EEintegraleq3}) of the form $\Lambda = \sum_{l=0}^\infty \Lambda_l$. Proceeding as above, we find
$$|\Lambda_l(x,k)| \leq \frac{C^l}{l!}\|\Lambda_0(\cdot, k)\|_{L^\infty([0, x])}\bigg(\frac{Cx}{|k|^{m+1}}\bigg)^l, \qquad x \geq 0, \  k \in \bar{\C}_-^K.$$
Using (\ref{Festb}), (\ref{EEDeltaest1}), and (\ref{EEDeltaEboundb}) in (\ref{EELambda0}), we obtain
\begin{align}\label{EEFF1}
\big|\Lambda_0(x,k) - [\partial_k\Psi_0(x,k)]_2 \big| \leq \frac{C(1+x)^2}{|k|^{m+1}}, \qquad x \geq 0, \  k \in \bar{\C}_-^K.
\end{align}
Thus $\sum_{l=0}^\infty \Lambda_l$ converges uniformly on compact subsets of $[0,\infty) \times \bar{\C}_-^K$ to a solution $\Lambda$ of (\ref{EEintegraleq3}), which satisfies
\begin{align}\label{EEFF2}
|\Lambda(x,k) - \Lambda_0(x,k)|
\leq \frac{C(1+x)^2xe^{\frac{Cx}{|k|^{m+1}}}}{|k|^{m+1}}, \qquad x \geq 0,  \  k \in \bar{\C}_-^K.
\end{align}
Equations (\ref{ZYWa}), (\ref{EEFF1}), and (\ref{EEFF2}) show that $[Y]_2 = \Psi$ satisfies (\ref{varphiasymptoticsb}) for $j = 1$.
Extending the above argument, we find that (\ref{varphiasymptoticsb}) holds also for $j = 2, \dots, n$.
\proofendcontinue

\begin{lemma} 
$[Y]_2$ satisfies (\ref{varphiasymptoticsc}).
\end{lemma}
\stepproofbegin
Let $y(x,k) = Y(x,k) e^{-2i\theta_1 x\sigma_3}$. Then $y$ satisfies $y_x = Ay - i\theta_1y\sigma_3$.
Thus
\begin{align}\nonumber
  (W_p^{-1}y)_x & = - W_p^{-1}(\partial_xW_p) W_p^{-1} y + W_p^{-1} \partial_x y
  	\\\nonumber
&  = - W_p^{-1}(A_{p,2}W_p - i\theta_1W_p\sigma_3) W_p^{-1} y + W_p^{-1} (Ay - i\theta_1 y\sigma_3)
	\\\label{Wpinvyx}
&  = W_p^{-1}\Delta_2 y + i\theta_1[\sigma_3, W_p^{-1}y].
\end{align}
Hence
$$\Big(e^{- i\theta_1x\hat{\sigma}_3}(W_p^{-1}y)\Big)_x = e^{- i\theta_1x\hat{\sigma}_3}(W_p^{-1} \Delta_2 y).$$
Integrating and using the initial condition $y(0,k) = I$, we conclude that $y$ satisfies the following Volterra integral equation:
\begin{align}\label{yWhatVolterra}
y(x,k) = W_p(x,k) e^{i\theta_1x\hat{\sigma}_3}W_p^{-1}(0,k) + \int_0^x W_p(x,k) e^{-i\theta_1(x'-x)\hat{\sigma}_3} (W_p^{-1}\Delta_2 y)(x',k) dx'.
\end{align}
Letting $\Psi = [y]_2$ and $\Psi_0(x,k) = [W_p(x,k) e^{i\theta_1x\hat{\sigma}_3}W_p^{-1}(0,k) ]_2$, we can write the second column of (\ref{yWhatVolterra}) as
$$\Psi(x,k) = \Psi_0(x,k) + \int_0^x E(x,x',k) (\Delta_2 \Psi)(x',k) dx',$$
where
$$E(x,x',k) = W_p(x,k) \begin{pmatrix} e^{-2i\theta_1(x'-x)} & 0 \\ 0 & 1 \end{pmatrix} W_p^{-1}(x',k).$$
As in the proof of Lemma \ref{lemma8}, the estimates
$$|\partial_k^jE(x, x', k)| < C(1+|x' - x|)^j, \qquad 0 \leq x' \leq x < \infty, \  k \in \bar{\C}_+^K, \  j = 0, 1, \dots, n,$$
and
$$|\Psi_0(x,k)| \leq C, \qquad x \geq 0, \  k \in \bar{\C}_+^K,$$
together with  (\ref{EEDeltaest}) yield
\begin{align}\label{HHPsiPsi0}
|\Psi(x,k) - \Psi_0(x,k)| \leq \frac{Cx e^{\frac{Cx}{|k|^{m+1}}}}{|k|^{m+1}}, \qquad x \geq 0, \  k \in \bar{\C}_+^K.
\end{align}
Equations (\ref{ZYWb}) and (\ref{HHPsiPsi0}) prove the second column of (\ref{varphiasymptoticsc}) for $j = 0$. Proceeding as in the proof of Lemma \ref{lemma8}, equation (\ref{varphiasymptoticsc}) follows also for $j = 1, \dots, n$.
This completes the proof of the theorem.
\end{proof}

\subsection{Asymptotics as $k \to 0$}
We next consider the behavior of $X$ and $Y$ as $k \to 0$. 
The matrix $\mathsf{U}$ is singular at $k = 0$, hence the integral equations (\ref{XYdef}) are not well-suited for determining the behavior of $X$ and $Y$ near $k = 0$. We therefore introduce new eigenfunctions $\hat{X}$ and $\hat{Y}$ as follows. 
Define $G_0(x)$ by (cf. (\ref{Gxtdef}))
\begin{align}\label{G0xdef}
G_0(x)= (-1)^{N_x}\begin{pmatrix}
  \cos \frac{u_0(x)}{2} & -\sin \frac{u_0(x)}{2} \\
  \sin \frac{u_0(x)}{2}  & \cos \frac{u_0(x)}{2}
 \end{pmatrix},
\end{align}
where the factor $(-1)^{N_x}$ has been inserted to ensure that $G_0(x) \to I$ as $x \to \infty$. 
If $X$ satisfies (\ref{xpartF}), then $\hat{X} = G_0^{-1}X$ satisfies
\begin{align}\label{xpartF-0}
\hat{X}_x + i \theta_1 [\sigma_3, \hat{X}] = \hat{\mathsf{U}}\hat{X},
\end{align}
where the function $\hat{\mathsf{U}}$ is regular at $k = 0$:
$$\hat{\mathsf{U}}(x,k) = G_0^{-1}(\mathsf{U}G_0 - G_{0x}) + i\theta_1(\sigma_3 - G_0^{-1}\sigma_3G_0)
= \hat{\mathsf{U}}_0(x)+ k\hat{\mathsf{U}}_1(x)$$
with
$$\hat{\mathsf{U}}_0(x)=\frac{i(u_{0x}-u_1)}{4}\sigma_2, \qquad 
  \hat{\mathsf{U}}_1(x) = \frac{i\sin u_0}{4}\sigma_1 - \frac{i(\cos u_0 -1)}{4}\sigma_3.$$
This leads to the following alternative representations for $X$ and $Y$:
\begin{align}\label{XXhatYYhat}
X(x,k) = G_0(x)\hat{X}(x,k), \qquad Y(x,k) = G_0(x)\hat{Y}(x,k)e^{-i\theta_1x\hat{\sigma}_3}G_0^{-1}(0),
\end{align}
where $\hat{X}$ and $\hat{Y}$ satisfy the Volterra integral equations
\begin{subequations}\label{hatXYdef}
\begin{align}  \label{hatXYdefa}
 & \hat{X}(x,k) = I + \int_{\infty}^x e^{i\theta_1(x'-x)\hat{\sigma}_3} (\hat{\mathsf{U}}\hat{X})(x',k) dx',
  	\\\label{hatXYdefb}
& \hat{Y}(x,k) = I + \int_{0}^x e^{i\theta_1(x'-x)\hat{\sigma}_3} (\hat{\mathsf{U}}\hat{Y})(x',k) dx'.
\end{align}
\end{subequations}
  
If we apply the same procedure that led to the estimates (\ref{Fest}) to the integral equations (\ref{hatXYdef}), we find that $\hat{X}$ and $\hat{Y}$ satisfy the estimates
\begin{subequations}\label{Fest-0}
\begin{align}\label{Festa-0}
& \bigg|\frac{\partial^j}{\partial k^j}\big(\hat{X}(x,k) - I\big) \bigg| \leq
\frac{C|k|^{-2j}}{(1+x)^{n-j}}, \qquad k \in (\bar{D}_3, \bar{D}_2)\setminus \{0\},
	\\\label{Festb-0}
& \bigg|\frac{\partial^j}{\partial k^j}\Big(\hat{Y}(x,k) - I\Big) \bigg| \leq
C(1 + x)^j |k|^{-2j}, \qquad k \in (\bar{D}_2, \bar{D}_3)\setminus \{0\},
	\\\label{Festc-0}
& \bigg|\frac{\partial^j}{\partial k^j}\Big((\hat{Y}(x,k) - I)e^{-2i\theta_1x\sigma_3}\Big) \bigg| \leq
C(1 + x)^j|k|^{-2j}, \qquad k \in (\bar{D}_3, \bar{D}_2) \setminus \{0\},
\end{align}
\end{subequations}
for $x \geq 0$ and $ j = 0, 1, \dots, n$.
It is possible to derive much more precise formulas for the behavior of $\hat{X}$ and $\hat{Y}$ for small $k$.
Indeed, our next theorem (Theorem \ref{xth3}) states that the asymptotic behavior of $\hat{X}$ and $\hat{Y}$ as $k \to 0$ can be obtained by considering formal power series solutions of equation (\ref{xpartF-0}). Since $X,Y$ are related to $\hat{X}, \hat{Y}$ by (\ref{XXhatYYhat}), this also gives the behavior of $X, Y$ as $k \to 0$ (see Corollary \ref{xcor}). Let us first find the formal solutions.

Equation (\ref{xpartF-0}) admits formal power series solutions $\hat{X}_{formal}$ and $\hat{Y}_{formal}$, normalized at $x = \infty$ and $x = 0$ respectively, such that
\begin{align}\label{Xformaldef-0}
& \hat{X}_{formal}(x,k) = I + \hat{X}_1(x)k + \hat{X}_2(x)k^2 + \cdots,
	\\\nonumber
& \hat{Y}_{formal}(x,k) =  I + \hat{Z}_1(x) k + \hat{Z}_2(x) k^2 + \cdots + \big(\hat{W}_1(x) k + \hat{W}_2(x) k^2 + \cdots \big) e^{2i\theta_1 x\sigma_3},
\end{align}
where the coefficients $\{\hat{X}_j(x), \hat{Z}_j(x), \hat{W}_j(x)\}_1^\infty$ satisfy
\begin{align}\label{Fjnormalization-0}
\lim_{x\to \infty} \hat{X}_j(x) = 0, \qquad \hat{Z}_j(0) + \hat{W}_j(0) = 0, \qquad j \geq 1.
\end{align}
Indeed, substituting
$$\hat{X} = I + {\hat{X}_1(x)}{k} + {\hat{X}_2(x)}{k^2} + \cdots$$
into (\ref{xpartF-0}), the off-diagonal terms of $O(k^{j})$ and the diagonal terms of $O(k^{j+1})$ yield the relations
\begin{align}\label{xrecursive-0}
\begin{cases}
\hat{X}_{j+1}^{(o)} = 2 i \sigma_3\big(-\partial_x \hat{X}_{j}^{(o)}-\frac{i}{2}\sigma_3 \hat{X}_{j-1}^{(o)}
+\hat{\mathsf{U}}_0 \hat{X}_j^{(d)} + \hat{\mathsf{U}}_1^{(o)} \hat{X}_{j-1}^{(d)}
+ \hat{\mathsf{U}}_1^{(d)} \hat{X}_{j-1}^{(o)}\big),
	\\
\partial_x \hat{X}_{j+1}^{(d)} = \hat{\mathsf{U}}_0 \hat{X}_{j+1}^{(o)}+ \hat{\mathsf{U}}_1^{(o)} \hat{X}_{j}^{(o)}+ \hat{\mathsf{U}}_1^{(d)} \hat{X}_j^{(d)}.
\end{cases}
\end{align}
The coefficients $\{\hat{Z}_j\}$ satisfy the equations obtained by replacing $\{\hat{X}_j\}$ with $\{\hat{Z}_j\}$ in \eqref{xrecursive-0}.
Similarly, substituting
$$\hat{X} = \big({\hat{W}_1(x)}{k} + {\hat{W}_2(x)}{k^2} + \cdots\big) e^{2i\theta_1 x\sigma_3}$$
into (\ref{xpartF-0}), the diagonal terms of $O(k^{j})$ and the off-diagonal terms of $O(k^{j+1})$ yield the relations
\begin{align}\label{xrecursive2-0}
\begin{cases}
\hat{W}_{j+1}^{(d)} = 2 i \sigma_3\big(-\partial_x \hat{W}_{j}^{(d)}-\frac{i}{2}\sigma_3 \hat{W}_{j-1}^{(d)}
+\hat{\mathsf{U}}_0 \hat{W}_j^{(o)} + \hat{\mathsf{U}}_1^{(o)} \hat{W}_{j-1}^{(o)}
+ \hat{\mathsf{U}}_1^{(d)} \hat{W}_{j-1}^{(d)}\big),
	\\
\partial_x \hat{W}_{j+1}^{(o)} = \hat{\mathsf{U}}_0 \hat{W}_{j+1}^{(d)}+ \hat{\mathsf{U}}_1^{(o)} \hat{W}_{j}^{(d)}+ \hat{\mathsf{U}}_1^{(d)} \hat{W}_j^{(o)}.
\end{cases}
\end{align}
The coefficients $\{\hat{X}_j(x), \hat{Z}_j(x), \hat{W}_j(x)\}$ are determined recursively from (\ref{Fjnormalization-0})-(\ref{xrecursive2-0}), the equations obtained by replacing $\{\hat{X}_j\}$ with $\{\hat{Z}_j\}$ in (\ref{xrecursive-0}), and the initial assignments
$$\hat{X}_{-1} = 0, \qquad \hat{X}_0 = I, \qquad \hat{Z}_{-1} = 0, \qquad \hat{Z}_0 = I, \qquad \hat{W}_{-1}=\hat{W}_0 = 0.$$
The first few coefficients are given by
\begin{align}\nonumber
\hat{X}_1(x) = &\;\frac{i(u_{0x}-u_1)}{2}\sigma_1 + \sigma_3 \frac{i}{4} \int_{\infty}^{x}\bigg[\frac{(u_{0x}-u_1)^2}{2} + 1 - \cos u_0\bigg] dx',
	\\ \label{x1-0}
\hat{X}_2(x) = &\; i\bigg[u_{0xx} - u_{1x} + i\frac{u_{0x} - u_1}{2}(\hat{X}_1)_{22} - \frac{1}{2}\sin u_0\bigg]\sigma_2
	\\ \nonumber
 & + I \int_{\infty}^{x}\bigg[\frac{i}{8}(2\cos u_0 - 2 - (u_{0x} - u_1)^2)(\hat{X}_1)_{22} + \frac{u_{0x}-u_1}{4}(u_{1x} - u_{0xx})\bigg] dx',
	\\ \nonumber
\hat{Z}_1(x) = &\; \frac{i(u_{0x}-u_1)}{2}\sigma_1 + \sigma_3 \frac{i}{4} \int_{0}^{x}\bigg[\frac{(u_{0x}-u_1)^2}{2} + 1 - \cos u_0\bigg] dx',
	\\ \nonumber
\hat{Z}_2(x) = &\; i\bigg[u_{0xx} - u_{1x} + i\frac{u_{0x} - u_1}{2}(\hat{Z}_1)_{22} - \frac{\sin u_0}{2}\bigg]\sigma_2
	\\ \nonumber
 & + I \bigg\{ \int_{0}^{x}\bigg[\frac{i}{8}(2\cos u_0 - 2 - (u_{0x} - u_1)^2)(\hat{Z}_1)_{22} + \frac{u_{0x}-u_1}{4}(u_{1x} - u_{0xx})\bigg] dx' 
 	\\
& - \frac{1}{4}(u_{0x}(0)-u_1(0))^2\bigg\},
\end{align}
and
\begin{align}\nonumber
\hat{W}_1(x) = & -\frac{i(u_{0x}(0)-u_1(0))}{2}\sigma_1,
	\\ \nonumber
\hat{W}_2(x) = &\; \frac{i}{2}(u_{0x}-u_1) (W_1)_{12} I +\sigma_2 \bigg\{\int_0^x \bigg[-\frac{1}{8}(u_{0x}-u_1)^2 + \frac{1}{4} (\cos u_0 -1)\bigg](W_1)_{12}  dx'\\ \nonumber &\;  -i\Big[u_{0xx}(0)-u_{1x}(0) - \frac{\sin u_0(0)}{2} \Big]\bigg\}.
\end{align}

We can now describe the asymptotic behavior of $\hat{X}$ and $\hat{Y}$ for small $k$.

\begin{theorem}[Asymptotics of $\hat{X}$ and $\hat{Y}$ as $k \to 0$]\label{xth3}
Let $u_0(x)$ and $u_1(x)$ be functions satisfying (\ref{ujassump}) for some given integers $m \geq 1$, $n \geq 1$, and $N_x \in \Z$.

As $k \to 0$, $\hat{X}$ and $\hat{Y}$ coincide to order $m$ with $\hat{X}_{formal}$ and $\hat{Y}_{formal}$, respectively, in the following sense: The functions
\begin{align*}
&\hat{X}_p(x,k) = I + \hat{X}_1(x)k + \cdots + \hat{X}_{m+1}(x) k^{m+1},
	\\ \nonumber
&\hat{Y}_p(x,k) = I + \hat{Z}_1(x)k + \cdots + \hat{Z}_{m+1}(x) k^{m+1}
 + \big(\hat{W}_1(x)k + \cdots + \hat{W}_{m+1}(x) k^{m+1}\big) e^{2i\theta_1 x\sigma_3},
\end{align*}
are well-defined and, for each $j = 0, 1, \dots, n$,
\begin{subequations}\label{varphiasymptotics-0}
\begin{align}\label{varphiasymptoticsa-0}
& \bigg|\frac{\partial^j}{\partial k^j}\big(\hat{X} - \hat{X}_p\big) \bigg| \leq
\frac{C|k|^{m+1-2j}}{(1+x)^{n-j}}, \qquad x \geq 0, \  k \in (\bar{D}_3, \bar{D}_2),
	\\ \label{varphiasymptoticsb-0}
& \bigg|\frac{\partial^j}{\partial k^j}\big(\hat{Y} - \hat{Y}_p\big) \bigg| \leq
C(1 + x)^{j+2}|k|^{m+1-2j}e^{Cx|k|^{m+1}}, \quad x \geq 0, \  k \in (\bar{D}_2, \bar{D}_3),
	\\ \label{varphiasymptoticsc-0}
& \bigg|\frac{\partial^j}{\partial k^j}\big((\hat{Y} - \hat{Y}_p)e^{-2i\theta_1 x\sigma_3}\big) \bigg| \leq
C(1 + x)^{j+2} |k|^{m+1-2j} e^{Cx|k|^{m+1}}, \quad  x \geq 0, \  k \in (\bar{D}_3, \bar{D}_2).
\end{align}
\end{subequations}
\end{theorem}
\begin{proof}
The proof is based on the equations (\ref{hatXYdef}) for $\hat{X}$ and $\hat{Y}$ and bears similarities with that of Theorem \ref{xth2}. However, since differentiation of the exponential $e^{i\theta_1 x}$ with respect to $k$ generates factors $\partial_k \theta_1 = \frac{1}{4}(1 + k^{-2})$ which are singular at $k = 0$ (but regular at $k = \infty$), the analysis is more complicated for small $k$ than for large $k$.

The proof proceeds through a series of lemmas. We first consider $[\hat{X}]_2$.

\begin{lemma} \label{lemma1-0}
 $\{\hat{X}_j(x)\}_1^{m+1}$ are weakly differentiable functions of $x \geq 0$ satisfying
\begin{align}\label{Fjbounds-0}
\begin{cases}
(1 + x)^n \hat{X}_j(x) \in L^1([0,\infty)) \cap L^\infty([0,\infty)), \\
(1 + x)^n \hat{X}_j'(x) \in L^1([0,\infty)),
\end{cases} \qquad j = 1, \dots, m+1.
\end{align}
\end{lemma}
\stepproofbegin
This is proved in the same way as Lemma \ref{lemma1}.
\proofendcontinue

\begin{lemma}\label{lemma2-0}
There exists a $K < 1$ such that $\hat{X}_p(x,k)^{-1}$ exists for all $k \in \C$ with $|k|\leq K$.
Moreover, letting $\hat{A} := -i\theta_1 \sigma_3 + \hat{\mathsf{U}}$ and
\begin{align}\label{Apdef-0}
\hat{A}_p(x,k) := \big(\partial_x \hat{X}_{p}(x,k) - i \theta_1 \hat{X}_p(x,k) \sigma_3\big)\hat{X}_p(x,k)^{-1}, \qquad x \geq 0, \  |k|\leq K,
\end{align}
the difference $\hat{\Delta}(x,k) := \hat{A}(x,k) - \hat{A}_p(x,k)$ satisfies
\begin{align}\label{DeltaEbounda-0}
|\partial_k^j \hat{\Delta}(x,k)| \leq \frac{f(x)|k|^{m+1-j}}{(1 + x)^n}, \qquad x \geq 0, \  |k|\leq K, \ j = 0,1, \dots, n,
\end{align}
where $f$ is a function in $L^1([0,\infty))$.
In particular,
\begin{align}\label{ABsmall-0}
\|\partial_k^j \hat{\Delta}(\cdot, k)\|_{L^1([x,\infty))} \leq \frac{C|k|^{m+1-j}}{(1+x)^n}, \qquad x \geq 0,  \  |k|\leq K, \ j = 0,1, \dots, n.
\end{align}
\end{lemma}
\stepproofbegin
By Lemma \ref{lemma1-0}, there exists a bounded function $g \in L^1([0,\infty))$ such that
\begin{align}\label{Fjxbound-0}
|\hat{X}_j(x)| \leq \frac{g(x)}{(m+1)(1+x)^n}, \qquad x \geq 0, \  j = 1, \dots, m+1.
\end{align}
In particular,
$$\bigg| \sum_{j=1}^{m+1} \hat{X}_j(x)k^j\bigg| \leq \frac{g(x)|k|}{(1+x)^n}, \qquad x \geq 0, \  |k| \leq 1.$$
Choose $K < \min(1, \|g\|_{L^\infty([0,\infty))}^{-1})$. Then $\hat{X}_p(x,k)^{-1}$ exists whenever $|k| \leq K$ and is given by the absolutely and uniformly convergent Neumann series
$$\hat{X}_p(x,k)^{-1} = \sum_{l =0}^\infty \bigg(-\sum_{j=1}^{m+1} \hat{X}_j(x)k^j\bigg)^l, \qquad x \geq 0,\  |k| \leq K.$$
Furthermore,
\begin{align}\label{tailestimate-0}
\bigg|\sum_{l = m+2}^\infty \bigg(-\sum_{j=1}^{m+1} \hat{X}_j(x)k^j\bigg)^l\bigg|
\leq \sum_{l = m+2}^\infty \bigg(\frac{g(x)|k|}{(1+x)^n}\bigg)^l
\leq \frac{Cg(x)|k|^{m+2}}{(1+x)^n},
\end{align}
for $x \geq 0$ and $|k| \leq K$. Now let $\Phi_0(x) + \frac{\Phi_1(x)}{k} + \frac{\Phi_2(x)}{k^2} + \cdots$ be the formal power series expansion of $\hat{X}_p(x,k)^{-1}$ as $k \to 0$, i.e.
\begin{align*}
& \Phi_0(x) = I, \quad \Phi_1(x) = -\hat{X}_1(x), \quad \Phi_2(x) = \hat{X}_1(x)^2 - \hat{X}_2(x),
	\\
& \Phi_3(x) = \hat{X}_1(x)\hat{X}_2(x) + \hat{X}_2(x) \hat{X}_1(x) - \hat{X}_1(x)^3 - \hat{X}_3(x), \quad \dots.
\end{align*}
The inequalities (\ref{Fjxbound-0}) and (\ref{tailestimate-0}) imply that the function $\mathcal{E}(x,k)$ defined by
\begin{align}\label{EhatX-0}
\mathcal{E}(x,k) = \hat{X}_p(x,k)^{-1} - \sum_{j=0}^{m+1} \Phi_j(x)k^j
\end{align}
satisfies
\begin{align}\label{Eestimate-0}
|\mathcal{E}(x,k)| \leq \frac{Cg(x)|k|^{m+2}}{(1+x)^n}, \qquad x \geq 0, \  |k| \leq K.
\end{align}

Let $\hat{A}_p(x,k)$ be given by (\ref{Apdef-0}).
Since $\hat{X}_{formal}$ is a formal solution of (\ref{xpartF-0}), the coefficient of $k^{j}$ in the formal expansion of $\hat{\Delta} = \hat{A} - \hat{A}_p$ as $k \to 0$  vanishes for $j \leq m$; hence, in view of Lemma \ref{lemma1-0} and (\ref{Eestimate-0}),
\begin{align}\nonumber
|\hat{\Delta}|
& = \bigg|\hat{A} - (\partial_x\hat{X}_{p} - i \theta_1 \hat{X}_p \sigma_3)\bigg(\sum_{j=0}^{m+1} \frac{\Phi_j}{k^j} + \mathcal{E}\bigg)\bigg|
	\\\nonumber
& \leq \frac{Cf(x)|k|^{m+1}}{(1+x)^n} + \bigg|(\partial_x\hat{X}_{p} - i\theta_1 \hat{X}_p \sigma_3)\mathcal{E}\bigg|
	\\ \label{Deltaestimate-0}
& \leq \frac{Cf(x)|k|^{m+1}}{(1+x)^n}, \qquad x \geq 0, \  |k| \leq K,
\end{align}
where $f$ is a function in $L^1([0,\infty))$. This proves (\ref{DeltaEbounda-0}) for $j = 0$.

Differentiation of (\ref{EhatX-0}) gives
$$\partial_k \mathcal{E} = -\bigg(\sum_{j=0}^{m+1} \Phi_jk^j + \mathcal{E}\bigg)(\partial_k\hat{X}_p)\bigg(\sum_{j=0}^{m+1} \Phi_jk^j + \mathcal{E}\bigg) - \frac{\partial}{\partial k}\sum_{j=0}^{m+1} \Phi_j(x)k^j.$$
By means of (\ref{Fjxbound-0}) and (\ref{Eestimate-0}), we infer that
\begin{align}\label{partialEestimate-0}
|\partial_k \mathcal{E}(x,k)| \leq \frac{Cg(x)|k|^{m+1}}{(1+x)^n}, \qquad x \geq 0, \  |k| \leq K.
\end{align}
The coefficient of $k^{j}$ in the formal expansion of $\partial_k \hat{\Delta}$ as $k \to 0$  vanishes for $j \leq m-1$. Therefore, an estimate as in (\ref{Deltaestimate-0}) gives, using (\ref{Eestimate-0}) and (\ref{partialEestimate-0}), 
$$|\partial_k \hat{\Delta}| \leq \frac{Cf(x)|k|^{m}}{(1+x)^n}, \qquad x \geq 0, \  |k| \leq K.$$
This proves (\ref{DeltaEbounda-0}) for $j = 1$. The argument is easily extended to $j = 2, \dots, n$.
\proofendcontinue

In the following, we let $K < 1$ be the constant from Lemma \ref{lemma2-0} and note that it is enough to prove the inequalities in (\ref{varphiasymptotics-0}) for $|k| \leq K$. We write, for $j = 2,3$, $\bar{D}_j^K := (\bar{D}_j \setminus \{0\}) \cap \{|k| \leq K\}$.

\begin{lemma}\label{lemma3-0}
$[\hat{X}]_2$ satisfies (\ref{varphiasymptoticsa-0}) for $j = 0$.
\end{lemma}
\stepproofbegin
A computation similar to the one leading to (\ref{hatXhatXp}) shows that $\hat{X}$ satisfies the Volterra equation
\begin{align}\label{hatXhatXp-0}
\hat{X}(x,k) = \hat{X}_p(x,k) - \int_x^\infty \hat{X}_p(x,k) e^{i\theta_1(x'-x)\hat{\sigma}_3} (\hat{X}_p^{-1}\hat{\Delta} \hat{X})(x',k) dx'.
\end{align}
Letting $\hat{\Psi} = [\hat{X}]_2$ and $\hat{\Psi}_0(x,k) =  [\hat{X}_p(x,k)]_2$, we can write the second column of (\ref{hatXhatXp-0}) as
\begin{align} \label{integraleq-0}
\hat{\Psi}(x,k) = &\; \hat{\Psi}_0(x,k) - \int_x^\infty \hat{E}(x,x',k) \hat{\Delta}(x', k) \hat{\Psi}(x', k) dx', \qquad x \geq 0, \ k \in \bar{D}_2^K,
\end{align}
where 
\begin{align}\label{E1E1E1def-0}
& \hat{E}(x,x',k) =  \hat{X}_p(x,k) \begin{pmatrix} e^{2i\theta_1(x' - x)} & 0 \\ 0 & 1 \end{pmatrix} \hat{X}_p(x',k)^{-1}.
\end{align}
Define $\hat{\Psi}_l$ for $l \geq 1$ inductively by
\begin{align*}
\hat{\Psi}_{l+1}(x,k) = - \int_x^\infty \hat{E}(x,x',k) \hat{\Delta}(x', k) \hat{\Psi}_l(x', k) dx', \qquad x \geq 0, \  k \in \bar{D}_2^K.
\end{align*}
Then
\begin{align}\label{Philiterated-0}
\hat{\Psi}_l(x,k) = & (-1)^l \int_{x = x_{l+1} \leq x_l \leq \cdots \leq x_1 < \infty} \prod_{i = 1}^l \hat{E}(x_{i+1}, x_i, k) \hat{\Delta}(x_i, k) \hat{\Psi}_0(x_1, k) dx_1 \cdots dx_l.
\end{align}
The estimates
\begin{align}\label{DeltaEboundb-0}
|\partial_k^j\hat{E}(x, x', k)| < C(1+|x' - x|)^j|k|^{-2j}, \qquad 0 \leq x \leq x' < \infty, \  k \in \bar{D}_2^K,
\end{align}
and (\ref{ABsmall-0}) with $j = 0$ give, for $x \geq 0$, $k \in \bar{D}_2^K$, and $l = 1,2, \dots$,
\begin{align}\nonumber
|\hat{\Psi}_l(x,k)| \leq \; & C^l \int_{x \leq x_l \leq \cdots \leq x_1 < \infty} \prod_{i = 1}^l  |\hat{\Delta}(x_i, k)| |\hat{\Psi}_0(x_1, k)| dx_1 \cdots dx_l
	\\ \nonumber
\leq & \; \frac{C^l}{l!}\|\hat{\Psi}_0(\cdot, k)\|_{L^\infty([x,\infty))} \|\hat{\Delta}(\cdot, k)\|_{L^1([x,\infty))}^l
	\\ \label{Philestimate-0}
\leq & \; \frac{C^l}{l!} \|\hat{\Psi}_0(\cdot, k)\|_{L^\infty([x,\infty))}\bigg(\frac{C|k|^{m+1}}{(1+x)^n}\bigg)^l.
\end{align}
This implies
$$|\hat{\Psi}_l(x,k)| \leq \frac{C^l}{l!}\bigg(\frac{C|k|^{m+1}}{(1+x)^n}\bigg)^l, \qquad x \geq 0, \  k \in \bar{D}_2^K, \ l = 1,2, \dots,$$
because, by Lemma \ref{lemma1-0},
\begin{align}\label{Psihat0bound-0}
|\hat{\Psi}_0(x, k)| < C, \qquad x \geq 0, \ k \in \bar{D}_2.
\end{align}
Hence the series $\sum_{l=0}^\infty \hat{\Psi}_l(x,k)$
converges absolutely and uniformly for $x \geq 0$ and $k \in \bar{D}_2^K$ to the solution $\hat{\Psi} = [\hat{X}]_2$ of (\ref{integraleq-0}).
Since
\begin{align}\label{Psiminusf-0}
|\hat{\Psi}(x,k) - \hat{\Psi}_0(x,k)| \leq \sum_{l=1}^\infty | \hat{\Psi}_l(x,k)| \leq \frac{C|k|^{m+1}}{(1+x)^n}, \qquad x \geq 0, \  k \in \bar{D}_2^K,
\end{align}
this proves (\ref{varphiasymptoticsa-0}) for $j = 0$. 
\proofendcontinue

\begin{lemma}\label{lemma4-0}
$[\hat{X}]_2$ satisfies (\ref{varphiasymptoticsa-0}) for $j = 1, \dots, n$.
\end{lemma}
\stepproofbegin
By (\ref{Psihat0bound-0}) and (\ref{Psiminusf-0}), we have $|\hat{\Psi}(x,k)| < C$ for $x \geq 0$ and $k \in \bar{D}_2^K$.
Let
\begin{align}\label{Lambda0-0}
\hat{\Lambda}_0(x,k) = \;& [\partial_k\hat{X}_p(x,k)]_2
- \int_x^\infty \frac{\partial}{\partial k}\big[\hat{E}(x,x',k) \hat{\Delta}(x', k)\big] \hat{\Psi}(x', k) dx'.
\end{align}
Differentiating the integral equation (\ref{integraleq-0}) with respect to $k$, we find that $\hat{\Lambda} := \partial_k\hat{\Psi}$ satisfies
\begin{align}\label{integraleq3-0}
\hat{\Lambda}(x,k) = \hat{\Lambda}_0(x,k)
- \int_x^\infty \hat{E}(x,x',k) \hat{\Delta}(x', k) \hat{\Lambda}(x', k) dx'
\end{align}
for each $k$ in the interior of $\bar{D}_2^K$.
We seek a solution of (\ref{integraleq3-0}) of the form $\hat{\Lambda} = \sum_{l=0}^\infty \hat{\Lambda}_l$,  where the $\hat{\Lambda}_l$ are defined by replacing $\{\hat{\Psi}_l\}$ by $\{\hat{\Lambda}_l\}$ in (\ref{Philiterated-0}). Proceeding as in (\ref{Philestimate-0}), we find
$$|\hat{\Lambda}_l(x,k)| \leq \frac{C^l}{l!}\|\hat{\Lambda}_0(\cdot, k)\|_{L^\infty([x,\infty))}\bigg(\frac{C|k|^{m+1}}{(1+x)^n}\bigg)^l, \qquad x \geq 0, \  k \in \bar{D}_2^K, \ l = 1,2, \dots.$$

Using (\ref{DeltaEbounda-0}) and (\ref{DeltaEboundb-0}) in (\ref{Lambda0-0}), we obtain
\begin{align}\label{FF1-0}
|\hat{\Lambda}_0(x,k) - [\partial_k\hat{X}_p(x,k)]_2 | \leq \frac{C|k|^{m-1}}{(1+x)^{n-1}}, \qquad x \geq 0, \  k \in \bar{D}_2^K.
\end{align}
In particular, $\|\hat{\Lambda}_0(\cdot, k)\|_{L^\infty([0,\infty))} < C(1 + |k|^{m-1})$ for $k \in \bar{D}_2^K$.
Thus, $\sum_{l=0}^\infty \hat{\Lambda}_l$ converges absolutely and uniformly on $[0,\infty) \times \bar{D}_2^K$ to a solution $\hat{\Lambda}$ of (\ref{integraleq3-0}), which satisfies the following analog of (\ref{Psiminusf-0}):
\begin{align}\label{FF2-0}
|\hat{\Lambda}(x,k) - \hat{\Lambda}_0(x,k)| \leq
\frac{C|k|^{m+1}}{(1+x)^{n}}, \qquad x \geq 0,  \  k \in \bar{D}_2^K.
\end{align}
Equations (\ref{FF1-0}) and (\ref{FF2-0}) show that $[\hat{X}]_2 = \hat{\Psi}$ satisfies (\ref{varphiasymptoticsa-0}) for $j = 1$.

Proceeding to higher values of $j$, we find that $\hat{\Lambda}^{(j)} := \partial_k^j\hat{\Psi}$ satisfies an integral equation of the form
\begin{align*}
\hat{\Lambda}^{(j)}(x,k) = \hat{\Lambda}_{0,j}(x,k)
- \int_x^\infty \hat{E}(x,x',k) \hat{\Delta}(x', k) \hat{\Lambda}^{(j)}(x', k) dx',\qquad j = 1, \dots, n,
\end{align*}
where
\begin{align}\label{hatLambda0j-0}
\big|\hat{\Lambda}_{0,j}(x, k) - [\partial_k^{j} \hat{X}_p(x, k)]_2 \big| \leq
\frac{C|k|^{m+1-2j}}{(1+x)^{n-j}}, \qquad x \geq 0, \  k \in \bar{D}_2^K.
\end{align}
In particular, $\|\hat{\Lambda}_{0,j}(\cdot, k)\|_{L^\infty([0,\infty))} < C(1 + |k|^{m+1-2j})$ for $k \in \bar{D}_2^K$ and $1 \leq j \leq n$. Consequently, the associated series $ \hat{\Lambda}^{(j)} = \sum_{l=0}^\infty \hat{\Lambda}^{(j)}_l$ converges absolutely and uniformly on $[0,\infty) \times \bar{D}_2^K$ to a solution satisfying
\begin{align}\nonumber
|\hat{\Lambda}^{(j)}(x,k) - \hat{\Lambda}_{0,j}(x,k)| & \leq \sum_{l=1}^\infty | \hat{\Lambda}_l^{(j)}(x,k)| 
\leq \|\hat{\Lambda}_{0,j}(\cdot, k)\|_{L^\infty([x,\infty))}\frac{C|k|^{m+1}}{(1+x)^n}
	\\ \label{hatLambdaj-0}
& \leq \frac{C|k|^{m+1-2j}}{(1+x)^n}, \qquad x \geq 0, \  k \in \bar{D}_2^K, \ j = 1, \dots, n.
\end{align}
Equations (\ref{hatLambda0j-0}) and (\ref{hatLambdaj-0}) show that $[\hat{X}]_2$ satisfies (\ref{varphiasymptoticsa-0}) for $j = 1, \dots, n$.
\proofendcontinue

The above lemmas prove the theorem for $\hat{X}$. We now consider $[\hat{Y}]_2$.

\begin{lemma}\label{lemma5-0}
$\{\hat{Z}_j(x), \hat{W}_j(x)\}_1^{m+1}$ are weakly differentiable functions of $x \geq 0$ satisfying
\begin{align}\label{ZjWj-0}
\begin{cases}
\hat{Z}_j, \hat{W}_j \in L^\infty([0,\infty)), \\
\hat{Z}_j', \hat{W}_j' \in L^1([0,\infty)),
\end{cases} \qquad j = 1, \dots, m+1.
\end{align}
\end{lemma}
\stepproofbegin
This is proved in the same way as Lemma \ref{lemma5}.
\proofendcontinue

We write
$$\hat{Y}_p(x,k) = \hat{Z}_p(x,k) + \hat{W}_p(x,k) e^{2i\theta_1 x\sigma_3},$$
where $\hat{Z}_p$ and $\hat{W}_p$ are defined by
$$\hat{Z}_p(x,k) = I + \hat{Z}_1(x)k + \cdots + \hat{Z}_{m+1}(x)k^{m+1}, \quad
\hat{W}_p(x,k) = \hat{W}_1(k)k + \cdots + \hat{W}_{m+1}(x)k^{m+1}.$$

\begin{lemma}\label{lemma6-0}
There exists a $K > 0$ such that $\hat{Z}_p(x,k)^{-1}$ and $\hat{W}_p(x,k)^{-1}$ exist for all $k \in \C$ with $|k|\leq K$.
Moreover, letting $\hat{A} :=-i\theta_1 \sigma_3 + \hat{\mathsf{U}} $ and
\begin{align*}
\begin{cases}
\hat{A}_{p,1}(x,k) := \big(\partial_x\hat{Z}_{p}(x,k) - i\theta_1 \hat{Z}_p(x,k) \sigma_3\big)\hat{Z}_p(x,k)^{-1},
	\\
\hat{A}_{p,2}(x,k) := \big(\partial_x\hat{W}_{p}(x,k) + i\theta_1 \hat{W}_p(x,k) \sigma_3\big)\hat{W}_p(x,k)^{-1},
\end{cases}
\quad x \geq 0, \  |k| \leq K,
\end{align*}
the differences
$$\hat{\Delta}_l(x,k) := \hat{A}(x,k) - \hat{A}_{p,l}(x,k), \qquad l = 1,2,$$
satisfy
\begin{align}\label{EEDeltaest1-0}
|\partial_k^j\hat{\Delta}_l(x,k)| \leq (C + f(x))|k|^{m+1-j}, \qquad x \geq 0, \  |k| \leq K, \  j = 0, 1, \dots, n, \  l = 1,2,
\end{align}
where $f$ is a function in $L^1([0,\infty))$.
In particular,
\begin{align}\label{EEDeltaest-0}
\|\partial_k^j \hat{\Delta}_l(\cdot, k)\|_{L^1([0,x])} \leq Cx|k|^{m+1-j}, \qquad x \geq 0, \  |k| \leq K, \  j = 0, 1, \dots, n, \ l = 1,2.
\end{align}
\end{lemma}
\stepproofbegin
The proof uses Lemma \ref{lemma5-0} and is similar to that of Lemma \ref{lemma2-0}.
\proofendcontinue

In the following, we let $K < 1$ be the constant from Lemma \ref{lemma6-0}.

\begin{lemma}\label{lemma7-0}
We have
\begin{subequations}
\begin{align}\label{ZYWa-0}
& \bigg| \frac{\partial^j}{\partial k^j}\Big[\hat{Z}_p(x,k) e^{-i\theta_1x\hat{\sigma}_3} \hat{Z}_p^{-1}(0,k) - \hat{Y}_p(x,k)\Big]_2 \bigg| \leq C(1+x)^j|k|^{m+2-2j}(1 + |e^{-2i\theta_1x}|),
	\\ \label{ZYWb-0}
& \bigg| \frac{\partial^j}{\partial k^j}\Big[\hat{W}_p(x,k) e^{i\theta_1x\hat{\sigma}_3} \hat{W}_p^{-1}(0,k) - \hat{Y}_p(x,k)e^{-2i\theta_1x\sigma_3}\Big]_2 \bigg| \leq C(1+x)^j|k|^{m+2-2j}(1 + |e^{2i\theta_1x}|),
\end{align}
for $x \geq 0$, $|k|\leq K$, and $j = 0, 1, \dots, n$.
\end{subequations}
\end{lemma}
\stepproofbegin
We write
$$\hat{Y}_{formal}(x,k) = \hat{Z}_{formal}(x,k) + \hat{W}_{formal}(x,k)e^{2i\theta_1x\sigma_3},$$
where
$$\hat{Z}_{formal}(x,k) = I + \hat{Z}_1(x)k + \hat{Z}_2(x)k^2 + \cdots, \qquad
\hat{W}_{formal}(x,k) = \hat{W}_1(x)k + \hat{W}_2(x)k^2 + \cdots.$$
Then
\begin{align}\label{ZZY-0}
\hat{Z}_{formal}(x,k) e^{-i\theta_1x\hat{\sigma}_3} \hat{Z}_{formal}^{-1}(0,k) = \hat{Y}_{formal}(x,k)
\end{align}
formally to all orders in $k$ as $k \to \infty$. Indeed, both sides of (\ref{ZZY-0}) are formal solutions of (\ref{xpartF-0}) satisfying the same initial condition at $x = 0$. Truncating (\ref{ZZY-0}) at order $k^{m+1}$, using Lemma \ref{lemma5-0}, and estimating the inverse $\hat{Z}_p^{-1}(0,k)$ as in the proof of Lemma \ref{lemma2-0}, it follows that
$$\hat{Z}_p(x,k) e^{-i\theta_1x\hat{\sigma}_3} \hat{Z}_p^{-1}(0,k) = \hat{Y}_p(x,k) + O(k^{m+2}) + O(k^{m+2})e^{2i\theta_1x\sigma_3}.$$
This gives (\ref{ZYWa-0}) for $j = 0$. Using the estimate $|\partial_k e^{\pm 2i\theta_1x}| \leq C x |k|^{-2} |e^{\pm 2i\theta_1x}|$ for $|k| \leq K$, we find (\ref{ZYWa-0}) also for $j \geq 1$.

Similarly, we have
$$\hat{W}_{formal}(x,k) e^{i\theta_1x\hat{\sigma}_3} \hat{W}_{formal}^{-1}(0,k) = \hat{Y}_{formal}(x,k)e^{-2i\theta_1x\sigma_3}$$
to all orders in $k$ and truncation leads to (\ref{ZYWb-0}).
\proofendcontinue

\begin{lemma}\label{lemma8-0}
$[\hat{Y}]_2$ satisfies (\ref{varphiasymptoticsb-0}).
\end{lemma}
\stepproofbegin
A computation similar to the one leading to (\ref{EEYVolterra}) shows that $\hat{Y}$ satisfies the Volterra integral equation
\begin{align}\label{EEYVolterra-0}
\hat{Y}(x,k) = \hat{Z}_p(x,k) e^{-i\theta_1x\hat{\sigma}_3}\hat{Z}_p^{-1}(0,k) + \int_0^x \hat{Z}_p(x,k) e^{-i\theta_1(x-x')\hat{\sigma}_3} (\hat{Z}_p^{-1}\hat{\Delta}_1 \hat{Y})(x',k) dx'.
\end{align}
Letting $\hat{\Psi} = [\hat{Y}]_2$ and $\hat{\Psi}_0(x,k) = [\hat{Z}_p(x,k) e^{-i\theta_1x\hat{\sigma}_3}\hat{Z}_p^{-1}(0,k) ]_2$, we can write the second column of (\ref{EEYVolterra-0}) as
\begin{align}\label{EEYVolterra2-0}
\hat{\Psi}(x,k) = \hat{\Psi}_0(x,k) + \int_0^x \hat{E}(x,x',k) (\hat{\Delta}_1 \hat{\Psi})(x',k) dx', \qquad x \geq 0, \ k \in \bar{D}_3^K,
\end{align}
where
$$\hat{E}(x,x',k) = \hat{Z}_p(x,k) \begin{pmatrix} e^{-2i\theta_1(x-x')} & 0 \\ 0 & 1 \end{pmatrix} \hat{Z}_p^{-1}(x',k).$$
We seek a solution $\hat{\Psi}(x,k) = \sum_{l=0}^\infty \hat{\Psi}_l(x,k)$ where
$$\hat{\Psi}_l(x,k) = \int_{0 \leq x_1 \leq \cdots \leq x_l \leq x_{l+1}=x < \infty} \prod_{i = 1}^l \hat{E}(x_{i+1}, x_i, k) \hat{\Delta}_1(x_i, k) \hat{\Psi}_0(x_1, k) dx_1 \cdots dx_l.$$
The estimates
\begin{align}\label{EEDeltaEboundb-0}
|\partial_k^j\hat{E}(x, x', k)| < \frac{C(1+|x' - x|)^j}{|k|^{2j}}, \qquad 0 \leq x' \leq x < \infty, \  k \in \bar{D}_3^K, \  j = 0, 1, \dots, n,
\end{align}
and
$$|\hat{\Psi}_0(x,k)| \leq C, \qquad x \geq 0, \  k \in \bar{D}_3^K,$$
together with  (\ref{EEDeltaest-0}) yield
\begin{align*}\nonumber
|\hat{\Psi}_l(x,k)| \leq \; & C^l \int_{0 \leq x_1 \leq \cdots \leq x_l \leq x < \infty} \prod_{i = 1}^l  |\hat{\Delta}_1(x_i, k)| |\hat{\Psi}_0(x_1, k)| dx_1 \cdots dx_l
	\\ \nonumber
\leq & \; \frac{C^l}{l!}\|\hat{\Psi}_0(\cdot, k)\|_{L^\infty([0,x])} \|\hat{\Delta}_1(\cdot, k)\|_{L^1([0,x])}^l
	\\
\leq &\; \frac{C^l}{l!} \big(Cx|k|^{m+1}\big)^l, \qquad x \geq 0, \  k \in \bar{D}_3^K, \ l = 1, 2, \dots.
\end{align*}
Hence
\begin{align}\label{EEPsiPsi0-0}
|\hat{\Psi}(x,k) - \hat{\Psi}_0(x,k)| \leq \sum_{l=1}^\infty | \hat{\Psi}_l(x,k)| \leq Cx |k|^{m+1} e^{Cx|k|^{m+1}}, \qquad x \geq 0, \  k \in \bar{D}_3^K.
\end{align}
Equations (\ref{ZYWa-0}) and (\ref{EEPsiPsi0-0}) prove the second column of (\ref{varphiasymptoticsb-0}) for $j = 0$.

Differentiating the integral equation (\ref{EEYVolterra2-0}) with respect to $k$, we find that $\hat{\Lambda} := \partial_k\hat{\Psi}$ satisfies
\begin{align}\label{EEintegraleq3-0}
\hat{\Lambda}(x,k) = \hat{\Lambda}_0(x,k)
+ \int_0^x \hat{E}(x,x',k) \hat{\Delta}_1(x', k) \hat{\Lambda}(x', k) dx'
\end{align}
for each $k$ in the interior of $\bar{D}_3^K$, where
\begin{align}\label{EELambda0-0}
\hat{\Lambda}_0(x,k) = \;& \partial_k\hat{\Psi}_0(x,k)
+ \int_0^x \frac{\partial}{\partial k}\big[\hat{E}(x,x',k) \hat{\Delta}_1(x', k)\big] \hat{\Psi}(x', k) dx'.
\end{align}
Seeking a solution of (\ref{EEintegraleq3-0}) of the form $\hat{\Lambda} = \sum_{l=0}^\infty \hat{\Lambda}_l$ and proceeding as above, we find
$$|\hat{\Lambda}_l(x,k)| \leq \frac{C^l}{l!}\|\hat{\Lambda}_0(\cdot, k)\|_{L^\infty([0, x])}\big(Cx|k|^{m+1}\big)^l, \qquad x \geq 0, \  k \in \bar{D}_3^K, \ l = 1, 2, \dots.$$
Using (\ref{Festb-0}), (\ref{EEDeltaest1-0}), and (\ref{EEDeltaEboundb-0}) in (\ref{EELambda0-0}), we obtain
\begin{align}\label{EEFF1-0}
\big|\hat{\Lambda}_0(x,k) - \partial_k\hat{\Psi}_0(x,k) \big| \leq C(1+x)^2|k|^{m-1}, \qquad x \geq 0, \  k \in \bar{D}_3^K.
\end{align}
On the other hand, Lemma \ref{lemma5-0} implies
$$\big| \partial_k^j [\hat{Y}_p(x,k)]_2\big| 
\leq C(1+x)^j|k|^{1-2j}, \qquad x \geq 0, \ k \in \bar{D}_3, \ j = 0, 1, \dots, n,$$
so, by (\ref{ZYWa-0}),
\begin{align*}
\big| \partial_k^j \hat{\Psi}_0(x,k)\big|
& \leq \big| \partial_k^j [\hat{Y}_p(x,k)]_2\big| 
+ C(1+x)^j|k|^{m+2-2j}(1 + |e^{-2i\theta_1x}|)
	\\
& \leq C(1+x)^j|k|^{1-2j}, \qquad x \geq 0, \ k \in \bar{D}_3, \ j = 0, 1, \dots, n.
\end{align*}
Thus (\ref{EEFF1-0}) gives
$$\|\hat{\Lambda}_0(\cdot, k)\|_{L^\infty([0, x])}
\leq C(1+x) |k|^{-1} + C(1+x)^2|k|^{m-1}, \qquad x \geq 0, \  k \in \bar{D}_3^K.$$
This shows that $\sum_{l=0}^\infty \hat{\Lambda}_l$ converges absolutely and uniformly on compact subsets of $[0,\infty) \times \bar{D}_3^K$ to the solution $\hat{\Lambda}$ of (\ref{EEintegraleq3-0}) and that
\begin{align}\nonumber
|\hat{\Lambda}(x,k) - \hat{\Lambda}_0(x,k)|
& \leq C \|\hat{\Lambda}_0(\cdot, k)\|_{L^\infty([0, x])}x|k|^{m+1}e^{Cx|k|^{m+1}}
	\\ \label{EEFF2-0}
& \leq C(1+x)^2x|k|^{m}e^{Cx|k|^{m+1}}, \qquad x \geq 0,  \  k \in \bar{D}_3^K.
\end{align}
Equations (\ref{ZYWa-0}), (\ref{EEFF1-0}), and (\ref{EEFF2-0}) show that $[\hat{Y}]_2 = \hat{\Psi}$ satisfies (\ref{varphiasymptoticsb-0}) for $j = 1$.
Extending the above argument, we find that (\ref{varphiasymptoticsb-0}) holds also for $j = 2, \dots, n$.
\proofendcontinue

\begin{lemma} \label{lemma9-0}
$[\hat{Y}]_2$ satisfies (\ref{varphiasymptoticsc-0}).
\end{lemma}
\stepproofbegin
Let $\hat{y}(x,k) = \hat{Y}(x,k) e^{-2i\theta_1 x\sigma_3}$. A computation similar to (\ref{Wpinvyx}) shows that $\hat{y}$ satisfies the Volterra integral equation
\begin{align}\label{yWhatVolterra-0}
\hat{y}(x,k) = \hat{W}_p(x,k) e^{i\theta_1x\hat{\sigma}_3}\hat{W}_p^{-1}(0,k) + \int_0^x \hat{W}_p(x,k) e^{i\theta_1(x-x')\hat{\sigma}_3} (\hat{W}_p^{-1}\hat{\Delta}_2 \hat{y})(x',k) dx'.
\end{align}
Letting $\hat{\Psi} = [\hat{y}]_2$ and $\hat{\Psi}_0(x,k) = [\hat{W}_p(x,k) e^{i\theta_1x\hat{\sigma}_3}\hat{W}_p^{-1}(0,k) ]_2$, we can write the second column of (\ref{yWhatVolterra-0}) as
$$\hat{\Psi}(x,k) = \hat{\Psi}_0(x,k) + \int_0^x \hat{E}(x,x',k) (\hat{\Delta}_2 \hat{\Psi})(x',k) dx',$$
where
$$\hat{E}(x,x',k) = \hat{W}_p(x,k) \begin{pmatrix} e^{2i\theta_1(x-x')} & 0 \\ 0 & 1 \end{pmatrix} \hat{W}_p^{-1}(x',k).$$
As in the proof of Lemma \ref{lemma8-0}, the estimates
$$|\partial_k^j\hat{E}(x, x', k)| < C(1+|x' - x|)^j|k|^{-2j}, \qquad 0 \leq x' \leq x < \infty, \  k \in \bar{D}_2^K, \  j = 0, 1, \dots, n,$$
and
$$|\hat{\Psi}_0(x,k)| \leq C, \qquad x \geq 0, \  k \in \bar{D}_2^K,$$
together with  (\ref{EEDeltaest-0}) yield
\begin{align}\label{HHPsiPsi0-0}
|\hat{\Psi}(x,k) - \hat{\Psi}_0(x,k)| \leq Cx |k|^{m+1} e^{Cx|k|^{m+1}}, \qquad x \geq 0, \  k \in \bar{D}_2^K.
\end{align}
Equations (\ref{ZYWb-0}) and (\ref{HHPsiPsi0-0}) prove the second column of (\ref{varphiasymptoticsc-0}) for $j = 0$. Proceeding as in the proof of Lemma \ref{lemma8-0}, equation (\ref{varphiasymptoticsc-0}) follows also for $j = 1, \dots, n$.
This completes the proof of the theorem.
\end{proof}

\begin{corollary}[Asymptotics of $X$ and $Y$ as $k \to 0$]\label{xcor}
Let $u_0(x)$ and $u_1(x)$ be functions satisfying (\ref{ujassump}) for some given integers $m \geq 1$, $n \geq 1$, and $N_x \in \Z$.

As $k \to 0$, $X$ and $Y$ coincide to order $m$ with $G_0\hat{X}_{formal}$ and $G_0\hat{Y}_{formal} e^{-i\theta_1 x \hat{\sigma}_3}G_0^{-1}(0)$ respectively in the following sense: 
For each $j = 0, 1, \dots, n$,
\begin{subequations}
\begin{align}\label{asymptoticszeroa}
& \bigg|\frac{\partial^j}{\partial k^j}\big(X - G_0\hat{X}_p\big) \bigg| \leq
\frac{C|k|^{m+1-2j}}{(1+x)^{n-j}}, \qquad x \geq 0, \  k \in (\bar{D}_3, \bar{D}_2),
	\\ \nonumber
& \bigg|\frac{\partial^j}{\partial k^j}\big(Y - G_0\hat{Y}_p e^{-i\theta_1 x \hat{\sigma}_3}G_0^{-1}(0)\big) \bigg| \leq C(1 + x)^{j+2}|k|^{m+1-2j}e^{Cx|k|^{m+1}}, 
	\\ \label{asymptoticszerob}
&\hspace{9cm} x \geq 0, \  k \in (\bar{D}_2, \bar{D}_3),
	\\ \nonumber
& \bigg|\frac{\partial^j}{\partial k^j}\Big(\big(Y - G_0\hat{Y}_p e^{-i\theta_1 x \hat{\sigma}_3}G_0^{-1}(0)\big)e^{-2i\theta_1 x\sigma_3}\Big) \bigg| \leq
C(1 + x)^{j+2} |k|^{m+1-2j} e^{Cx|k|^{m+1}}, 
	\\ \label{asymptoticszeroc}
&\hspace{9cm} x \geq 0, \  k \in (\bar{D}_3, \bar{D}_2),
\end{align}
\end{subequations}
where $G_0(x)$ is defined by (\ref{G0xdef}).
In particular, for each $x \geq 0$, it holds that $X(x,\cdot)$ is bounded for $k \in (\bar{\C}_-, \bar{\C}_+)$; $Y(x,\cdot)$ is bounded for $k \in (\bar{\C}_+, \bar{\C}_-)$; and $Y(x,\cdot)e^{-2i\theta_1 x\sigma_3}$ is bounded for $k \in (\bar{\C}_-, \bar{\C}_+)$.
\end{corollary}
\begin{proof}
The estimate (\ref{varphiasymptoticsa-0}) yields, for $j = 0, 1,\dots, n$,
$$\bigg|\frac{\partial^j}{\partial k^j}\big(X - G_0\hat{X}_p\big)\bigg| = \bigg|G_0 \frac{\partial^j}{\partial k^j}(\hat{X} - \hat{X}_p)\bigg| 
\leq \frac{C|k|^{m+1-2j}}{(1+x)^{n-j}}, \qquad x \geq 0, \  k \in (\bar{D}_3, \bar{D}_2),$$
which proves (\ref{asymptoticszeroa}).
On the other hand, using that $Y = G_0(x) \hat{Y} e^{-i\theta_1 x \hat{\sigma}_3}G_0^{-1}(0)$, we find
\begin{align*}
&[Y - G_0\hat{Y}_p e^{-i\theta_1 x \hat{\sigma}_3}G_0^{-1}(0)]_2
 = G_0[(\hat{Y} - \hat{Y}_p)e^{-i\theta_1 x \hat{\sigma}_3}G_0^{-1}(0)]_2
	\\
& = (-1)^{N_x}G_0[\hat{Y} - \hat{Y}_p]_1 e^{-2i\theta_1 x}  \sin\frac{u_0(0)}{2} + (-1)^{N_x}G_0[\hat{Y} - \hat{Y}_p]_2 \cos \frac{u_0(0)}{2}.
\end{align*}
Using (\ref{varphiasymptoticsc-0}) and (\ref{varphiasymptoticsb-0}) to estimate the first and second terms on the right-hand side, respectively, we find
$$\bigg|\frac{\partial^j}{\partial k^j}[Y - G_0\hat{Y}_p e^{-i\theta_1 x \hat{\sigma}_3}G_0^{-1}(0)]_2\bigg|
\leq C(1 + x)^{j+2} |k|^{m+1-2j} e^{Cx|k|^{m+1}}, \qquad x \geq 0, \  k \in \bar{D}_3,$$
for $j = 0, 1, \dots, n$. This establishes the second column of (\ref{asymptoticszerob}). The first column of (\ref{asymptoticszerob}) and the two columns of (\ref{asymptoticszeroc}) are established in a similar way. 

We have $[G_0(x) \hat{X}_p]_2 \leq C[\hat{X}_p]_2 \leq C$ for $k \in \bar{D}_2$; hence (\ref{asymptoticszeroa}) shows that $[X]_2$ is bounded for $k \in \bar{D}_2$ (this follows also from (\ref{Festa-0})). In view of (\ref{Festa}), $[X]_2$ is then bounded for all $k \in \bar{\C}_+$.
On the other hand, since  $\hat{Y}_p = \hat{Z}_p + \hat{W}_p e^{2i\theta_1 x \sigma_3}$, we find
\begin{align}\nonumber
(-1)^{N_x} [G_0\hat{Y}_p e^{-i\theta_1 x \hat{\sigma}_3}G_0^{-1}(0)]_2
= &\; G_0[\hat{Y}_p]_1 e^{-2i\theta_1 x}  \sin\frac{u_0(0)}{2} + G_0[\hat{Y}_p]_2 \cos \frac{u_0(0)}{2}
	\\\nonumber
= &\; G_0\bigg([\hat{Z}_p]_1  \sin\frac{u_0(0)}{2} + [\hat{W}_p]_2 \cos \frac{u_0(0)}{2}\bigg)e^{-2i\theta_1 x} 
	\\\label{G0hatYpe}
& + G_0\bigg([\hat{W}_p]_1 \sin\frac{u_0(0)}{2} + [\hat{Z}_p]_2 \cos \frac{u_0(0)}{2}\bigg).
\end{align}
The functions $e^{-2i\theta_1 x}$, $\hat{Z}_p$, and $\hat{W}_p$ are bounded for $x \geq 0$ and $k \in \bar{D}_3$ (see Lemma \ref{lemma5-0}). Hence equations (\ref{asymptoticszerob}) and (\ref{G0hatYpe}) imply that the column $[Y]_2$ is bounded for $x \geq 0$ and $k \in \bar{D}_3$; by (\ref{Festb}) it is then bounded for all $k \in \bar{\C}_-$. The boundedness of $[Y]_1$ and the two columns of $Ye^{-2i\theta_1 x\sigma_3}$ can be established in a similar way. 
\end{proof}

\section{Spectral analysis of the $t$-part}\label{tpartapp}
Consider the $t$-part of the Lax pair (\ref{lax}) evaluated at $x = 0$:
\begin{align}\label{tpartF}
  T_t + i\theta_2 [\sigma_3, T] = \mathsf{V} T,
\end{align}
where $\theta_2 := \theta_2(k) = \frac{1}{4}(k + \frac{1}{k})$ and $\mathsf{V}(t,k)$ is given by (\ref{mathsfVdef}), i.e.,
\begin{align*}
  \mathsf{V}(t,k) = \mathsf{V}_0(t)+\frac{1}{k} \mathsf{V}_1(t),
\end{align*} 
with 
$$\mathsf{V}_0(t) = -\frac{i(g_1(t) + g_{0t}(t))}{4}\sigma_2,\qquad 
  \mathsf{V}_1(t) = -\frac{i\sin g_0(t)}{4} \sigma_1 - \frac{i(\cos{g_0(t)} -1)}{4}\sigma_3.$$
We define two $2 \times 2$-matrix valued solutions $T(t,k)$ and $U(t,k)$ of (\ref{tpartF}) as the solutions of the linear Volterra integral equations
\begin{subequations}\label{TUdef}
\begin{align}\label{TUdefa}
 & T(t,k) = I + \int_{\infty}^t e^{i\theta_2(t'-t)\hat{\sigma}_3} (\mathsf{V}T)(t',k) dt',
  	\\
&  U(t,k) = I + \int_{0}^t e^{i\theta_2(t'-t)\hat{\sigma}_3} (\mathsf{V}U)(t',k) dt'.
\end{align}
\end{subequations}
Let $D_+ = D_1 \cup D_3$ and $D_- = D_2 \cup D_4$. 

\begin{theorem}[Basic properties of $T$ and $U$]\label{tth1}
Let $n \geq 1$ and $N_t \in \Z$ be integers. Suppose $g_0, g_1:[0,\infty)\to \R$ satisfy
\begin{align*}
  (1+t)^n(|g_0 - 2\pi N_t| + |g_{0t}|+|g_1|)\in  L^1([0,\infty)).
\end{align*}
Then the equations (\ref{TUdef}) uniquely define two $2 \times 2$-matrix valued solutions $T$ and $U$ of (\ref{tpartF}) with the following properties:
\begin{enumerate}[$(a)$]
\item The function $T(t, k)$ is defined for $t \geq 0$ and $k \in (\bar{D}_-, \bar{D}_+) \setminus \{0\}$. For each $k \in (\bar{D}_-, \bar{D}_+)\setminus \{0\}$, $T(\cdot, k)$ is weakly differentiable and satisfies (\ref{tpartF}).

\item The function $U(t, k)$ is defined for $t \geq 0$ and $k \in \C\setminus \{0\}$. For each $k \in \C\setminus \{0\}$, $U(\cdot, k)$ is weakly differentiable and satisfies (\ref{tpartF}).

\item For each $t \geq 0$, the function $T(t,\cdot)$ is continuous for $k \in (\bar{D}_-, \bar{D}_+)\setminus \{0\}$ and analytic for $k \in (D_-, D_+)$.

\item For each $t \geq 0$, the function $U(t,\cdot)$ is analytic for $k \in \C \setminus \{0\}$. 

\item For each $t \geq 0$ and each $j = 1, \dots, n$, the partial derivative $\frac{\partial^j T}{\partial k^j}(t, \cdot)$ has a continuous extension to $(\bar{D}_-, \bar{D}_+)\setminus \{0\}$.

\item $T$ and $U$ satisfy the following estimates:
\begin{subequations}\label{Fest-t}
\begin{align}\label{Festa-t}
& \bigg|\frac{\partial^j}{\partial k^j}\big(T(t,k) - I\big) \bigg| \leq
\frac{C}{(1+t)^{n-j}}, \qquad k \in (\bar{D}_4, \bar{D}_1),
	\\\label{Festb-t}
& \bigg|\frac{\partial^j}{\partial k^j}\Big(U(t,k) - I\Big) \bigg| \leq
C(1 + t)^j, \qquad k \in (\bar{D}_1, \bar{D}_4), 
	\\\label{Festc-t}
& \bigg|\frac{\partial^j}{\partial k^j}\Big((U(t,k) - I)e^{-2i\theta_2t\sigma_3}\Big) \bigg| \leq
C(1 + t)^j, \qquad k \in (\bar{D}_4, \bar{D}_1), 
\end{align}
\end{subequations}
for $t \geq 0$ and $j = 0, 1, \dots, n$.

\item $\det T(t,k) = 1$ for $t \geq 0$, $k \in (\bar{D}_+ \cap \bar{D}_-)\setminus \{0\}$, and $\det U(t,k) = 1$ for $t \geq 0$, $k \in \C \setminus \{0\}$.

\item $T$ and $U$ obey the following symmetries for each $t \geq 0$:
\begin{align*}
& T(t,k) = \sigma_2 T(t,-k) \sigma_2 = \sigma_2 \overline{T(t,\bar{k})} \sigma_2, \qquad k \in (\bar{D}_-, \bar{D}_+) \setminus \{0\},
	\\
&U(t,k) = \sigma_2 U(t,-k) \sigma_2 = \sigma_2 \overline{U(t,\bar{k})} \sigma_2, \qquad k \in \C \setminus \{0\}.
\end{align*}

\end{enumerate}
\end{theorem}
\begin{proof}
The proof is similar to that of Theorem \ref{xth1} and will be omitted.
\end{proof}

\subsection{Asymptotics as $k \to \infty$}
We next consider the behavior of the eigenfunctions $T$ and $U$ as $k \to \infty$. Equation (\ref{tpartF}) admits formal power series solutions $T_{formal}$ and $U_{formal}$, normalized at $t = \infty$ and $t = 0$ respectively, such that
\begin{align*}
& T_{formal}(t,k) = I + \frac{T_1(t)}{k} + \frac{T_2(t)}{k^2} + \cdots,
	\\
& U_{formal}(t,k) =  I + \frac{Z_1(t)}{k} + \frac{Z_2(t)}{k^2} + \cdots + \bigg(\frac{W_1(t)}{k} + \frac{W_2(t)}{k^2} + \cdots \bigg) e^{2i\theta_2t\sigma_3},
\end{align*}
where the coefficients $\{T_j(t), Z_j(t), W_j(t)\}_1^\infty$ satisfy
\begin{align}\label{Fjnormalization-t}
\lim_{t\to \infty} T_j(t) = 0, \qquad Z_j(0) + W_j(0) = 0, \qquad j \geq 1.
\end{align}
Indeed, substituting
$$T = I + \frac{T_1(t)}{k} + \frac{T_2(t)}{k^2} + \cdots$$
into (\ref{tpartF}), the off-diagonal terms of $O(k^{-j})$ and the diagonal terms of $O(k^{-j-1})$ yield the relations
\begin{align}\label{trecursive}
\begin{cases}
T_{j+1}^{(o)} = -2 i \sigma_3\big(-\partial_tT_{j}^{(o)}-\frac{i}{2}\sigma_3 T_{j-1}^{(o)}+\mathsf{V}_0 T_j^{(d)}+\mathsf{V}_1^{(o)} T_{j-1}^{(d)}+\mathsf{V}_1^{(d)} T_{j-1}^{(o)}\big),
	\\
\partial_t T_{j+1}^{(d)} = \mathsf{V}_0 T_{j+1}^{(o)}+\mathsf{V}_1^{(o)} T_j^{(o)}+\mathsf{V}_1^{(d)} T_j^{(d)}.
\end{cases}
\end{align}
The coefficients $\{Z_j\}$ satisfy the equations obtained by replacing $\{T_j\}$ with $\{Z_j\}$ in \eqref{trecursive}.
Similarly, substituting
$$T = \bigg(\frac{W_1(t)}{k} + \frac{W_2(t)}{k^2} + \cdots\bigg) e^{2i\theta_2t\sigma_3}$$
into (\ref{tpartF}), the diagonal terms of $O(k^{-j})$ and the off-diagonal terms of $O(k^{-j-1})$ yield the relations
\begin{align}\label{trecursive2}
\begin{cases}
W_{j+1}^{(d)} = -2 i \sigma_3\big(-\partial_tW_{j}^{(d)}-\frac{i}{2}\sigma_3 W_{j-1}^{(d)}+\mathsf{V}_0 W_j^{(o)}+\mathsf{V}_1^{(o)} W_{j-1}^{(o)}+\mathsf{V}_1^{(d)} W_{j-1}^{(d)}\big),
	\\
\partial_t W_{j+1}^{(o)} = \mathsf{V}_0 W_{j+1}^{(d)}+\mathsf{V}_1^{(o)} W_j^{(d)}+\mathsf{V}_1^{(d)} W_j^{(o)}.
\end{cases}
\end{align}
The coefficients $\{T_j(t), Z_j(t), W_j(t)\}$ are determined recursively from (\ref{Fjnormalization-t})-(\ref{trecursive2}), the equations obtained by replacing $\{T_j\}$ with $\{Z_j\}$ in (\ref{trecursive}), and the initial assignments
$$T_{-1} = 0, \qquad T_0 = I, \qquad Z_{-1} = 0, \qquad Z_0 = I, \qquad W_{-1}=W_0 = 0.$$
The first few coefficients are given by
\begin{align*}\nonumber
T_1(t) = &\; \frac {i(g_1+g_{0t})}{2} \sigma_1- \sigma_3 \frac{i}{4} \int_{\infty}^{t}\bigg[\frac{(g_1+g_{0t})^2}{2}+\cos g_0-1 \bigg]dt',
	\\ \nonumber
T_2(t) = &\; i \bigg[-g_{0tt}-g_{1t}+\frac{i}{2}(g_1+g_{0t})(T_1)_{22} - \frac{1}{2} \sin g_0\bigg]\sigma_2
	\\ \nonumber
 &\;+ I\int_{\infty}^t \bigg\{\frac{i}{4}\Big[\frac{(g_{1} + g_{0t})^2}{2} + \cos{g_0} - 1\Big](T_1)_{22}
 - \frac{1}{4}(g_{1} + g_{0t})(g_{1t} + g_{0tt})\bigg\} dt',
	\\ 
Z_1(t) = &\; \frac {i(g_1+g_{0t})}{2} \sigma_1+\sigma_3 \frac{i}{4} \int_{0}^{t}\bigg[\frac{(g_1+g_{0t})^2}{2}+\cos g_0-1 \bigg]dt',
	\\ \nonumber
Z_2(t) = &\; i \bigg[-g_{0tt}-g_{1t}+\frac{i}{2}(g_1+g_{0t})(Z_1)_{22} - \frac{1}{2} \sin g_0\bigg]\sigma_2\\\nonumber
 &\;+ I\bigg\{\int_{0}^t \bigg[\frac{i}{4}\Big[\frac{(g_{1} + g_{0t})^2}{2} + \cos{g_0} - 1\Big](Z_1)_{22}
 - \frac{1}{4}(g_{1} + g_{0t})(g_{1t} + g_{0tt})\bigg] dt'
 	\\
& -\frac{1}{4} (g_1(0)+g_{0t}(0))^2\bigg\},
\end{align*}
and
\begin{align*}
W_1(t) = & -\frac{i(g_1(0)+g_{0t}(0))}{2}\sigma_1,
	\\ \nonumber
W_2(t) = &\; \frac{i(g_1+g_{0t})}{2} (W_1)_{12} I +\sigma_2 \bigg\{\int_0^t \frac{1}{4}\bigg[\frac{(g_1+g_{0t})^2}{2} + \cos g_0 -1\bigg](W_1)_{12}dt' 
	\\
&+i\Big[g_{0tt}(0) + g_{1t}(0) + \frac{1}{2} \sin g_0(0)\Big]\bigg\}.
\end{align*}

We following theorem gives the behavior of $T$ and $U$ for large $k$.

\begin{theorem}[Asymptotics of $T$ and $U$ as $k \to \infty$]\label{tth2}
Let $m \geq 1$, $n \geq 1$, and $N_t \in \Z$ be integers.  
Let $g_0(t)$ and $g_1(t)$ be functions satisfying the following regularity and decay assumptions:
\begin{align}\label{gjassump}
\begin{cases}
(1+t)^n(g_0(t) - 2\pi N_t) \in L^1([0,\infty)),
	\\
(1+t)^n\partial^i g_0(t) \in L^1([0,\infty)), \qquad i = 1, \dots, m+2,
	\\
(1+t)^n\partial^i g_1(t) \in L^1([0,\infty)), \qquad i = 0,1, \dots, m+1.
\end{cases}
\end{align}

As $k \to \infty$, $T$ and $U$ coincide to order $m$ with $T_{formal}$ and $U_{formal}$ respectively in the following sense: The functions
\begin{align*}
&T_p(t,k) := I + \frac{T_1(t)}{k} + \cdots + \frac{T_{m+1}(t)}{k^{m+1}},
	\\ \nonumber
&U_p(t,k) = I + \frac{Z_1(t)}{k} + \cdots + \frac{Z_{m+1}(t)}{k^{m+1}}
 + \bigg(\frac{W_1(k)}{k} + \cdots + \frac{W_{m+1}(t)}{k^{m+1}}\bigg) e^{2i\theta_2t\sigma_3},
\end{align*}
are well-defined and, for each $j = 0, 1, \dots, n$,
\begin{subequations}\label{varphiasymptotics-t}
\begin{align}\label{varphiasymptoticsa-t}
& \bigg|\frac{\partial^j}{\partial k^j}\big(T - T_p\big) \bigg| \leq
\frac{C}{|k|^{m+1}(1+t)^{n-j}}, \qquad t \geq 0, \   k \in (\bar{D}_4, \bar{D}_1),
	\\ \label{varphiasymptoticsb-t}
& \bigg|\frac{\partial^j}{\partial k^j}\big(U - U_p\big) \bigg| \leq
\frac{C(1 + t)^{j+2}e^{\frac{Ct}{|k|^{m+1}}}}{|k|^{m+1}}, \qquad t \geq 0, \  k \in (\bar{D}_1, \bar{D}_4),
	\\ \label{varphiasymptoticsc-t}
& \bigg|\frac{\partial^j}{\partial k^j}\big((U - U_p)e^{-2i\theta_2t\sigma_3}\big) \bigg| \leq
\frac{C(1 + t)^{j+2}e^{\frac{Ct}{|k|^{m+1}}}}{|k|^{m+1}}, \quad t \geq 0, \ k \in (\bar{D}_4, \bar{D}_1).
\end{align}
\end{subequations}
\end{theorem}
\begin{proof}
The proof is similar to that of Theorem \ref{xth2} and will be omitted.
\end{proof}

\subsection{Asymptotics as $k \to 0$}
We now turn to the behavior of $T$ and $U$ as $k \to 0$. Let
\begin{align}\label{G0tdef}
\mathcal{G}_0(t)= (-1)^{N_t} \begin{pmatrix}
  \cos \frac{g_0(t)}{2} & -\sin \frac{g_0(t)}{2} \\
  \sin \frac{g_0(t)}{2}  & \cos \frac{g_0(t)}{2}
 \end{pmatrix},
\end{align}
where the factor $(-1)^{N_t}$ has been inserted to ensure that $\mathcal{G}_0(t) \to I$ as $t \to \infty$. 
If $X$ satisfies (\ref{tpartF}), then $\hat{T} = \mathcal{G}_0^{-1}T$ satisfies
\begin{align}\label{tpartF-0}
\hat{T}_t + i\theta_2[\sigma_3, \hat{T}] = \hat{\mathsf{V}}\hat{T},
\end{align}
where the function $\hat{\mathsf{V}}$ is regular at $k = 0$:
$$\hat{\mathsf{V}}(t,k) = \hat{\mathsf{V}}_0(t)+ k\hat{\mathsf{V}}_1(t)$$
with
$$\hat{\mathsf{V}}_0(t)=-\frac{i(g_1- g_{0t})}{4}\sigma_2, \qquad 
  \hat{\mathsf{V}}_1(t) = \frac{i\sin g_0}{4}\sigma_1 - \frac{i(\cos{g_0} -1)}{4}\sigma_3.$$
This leads to the following alternative representations for $T$ and $U$:
\begin{align}\label{TThatUUhat}
T(t,k) = \mathcal{G}_0(t)\hat{T}(t,k), \qquad U(t,k) = \mathcal{G}_0(t)\hat{U}(t,k)e^{-i\theta_2t\hat{\sigma}_3}\mathcal{G}_0^{-1}(0),
\end{align}
where $\hat{T}$ and $\hat{U}$ satisfy the Volterra integral equations
\begin{subequations}\label{TUhatdef}
\begin{align}  \label{TUhatdefa}
 & \hat{T}(t,k) = I + \int_{\infty}^t e^{i\theta_2(t'-t)\hat{\sigma}_3} (\hat{\mathsf{V}}\hat{T})(t',k) dt',
  	\\\label{TUhatdefb}
& \hat{U}(t,k) = I + \int_{0}^t e^{i\theta_2(t'-t)\hat{\sigma}_3} (\hat{\mathsf{V}}\hat{U})(t',k) dt'.
\end{align}
\end{subequations}
Equation (\ref{tpartF-0}) admits formal power series solutions $\hat{T}_{formal}$ and $\hat{U}_{formal}$, normalized at $t = \infty$ and $t = 0$ respectively, such that
\begin{align*}
& \hat{T}_{formal}(t,k) = I + \hat{T}_1(t)k + \hat{T}_2(t)k^2 + \cdots,
	\\
& \hat{U}_{formal}(t,k) =  I + \hat{Z}_1(t) k + \hat{Z}_2(t) k^2 + \cdots + \big(\hat{W}_1(t) k + \hat{W}_2(t) k^2 + \cdots \big) e^{2i\theta_2t\sigma_3},
\end{align*}
where the coefficients $\{\hat{T}_j(t), \hat{Z}_j(t), \hat{W}_j(t)\}_1^\infty$ satisfy
\begin{align}\label{Fjnormalization-t0}
\lim_{t\to \infty} \hat{T}_j(t) = 0, \qquad \hat{Z}_j(0) + \hat{W}_j(0) = 0, \qquad j \geq 1.
\end{align}
Indeed, substituting
$$\hat{T} = I + {\hat{T}_1(t)}{k} + {\hat{T}_2(t)}{k^2} + \cdots$$
into (\ref{tpartF-0}), the off-diagonal terms of $O(k^{j})$ and the diagonal terms of $O(k^{j+1})$ yield the relations
\begin{align}\label{trecursive-0}
\begin{cases}
\hat{T}_{j+1}^{(o)} = -2 i \sigma_3\big(-\partial_t\hat{T}_{j}^{(o)}-\frac{i}{2}\sigma_3 \hat{T}_{j-1}^{(o)}
+\hat{\mathsf{V}}_0 \hat{T}_j^{(d)}+ \hat{\mathsf{V}}_1^{(o)} \hat{T}_{j-1}^{(d)}
+ \hat{\mathsf{V}}_1^{(d)} \hat{T}_{j-1}^{(o)}\big),
	\\
\partial_t \hat{T}_{j+1}^{(d)} = \hat{\mathsf{V}}_0 \hat{T}_{j+1}^{(o)}
+ \hat{\mathsf{V}}_1^{(o)} \hat{T}_j^{(o)} + \hat{\mathsf{V}}_1^{(d)} \hat{T}_j^{(d)}.
\end{cases}
\end{align}
The coefficients $\{\hat{Z}_j\}$ satisfy the equations obtained by replacing $\{\hat{T}_j\}$ with $\{\hat{Z}_j\}$ in \eqref{trecursive-0}.
Similarly, substituting
$$\hat{T} = \big({\hat{W}_1(t)}{k} + {\hat{W}_2(t)}{k^2} + \cdots\big) e^{2i\theta_2t\sigma_3}$$
into (\ref{tpartF-0}), the diagonal terms of $O(k^{j})$ and the off-diagonal terms of $O(k^{j+1})$ yield the relations
\begin{align}\label{trecursive2-0}
\begin{cases}
\hat{W}_{j+1}^{(d)} = -2 i \sigma_3\big(-\partial_t\hat{W}_{j}^{(d)}-\frac{i}{2}\sigma_3 \hat{W}_{j-1}^{(d)} + \hat{\mathsf{V}}_0 \hat{W}_j^{(o)}
+ \hat{\mathsf{V}}_1^{(o)} \hat{W}_{j-1}^{(o)}
+ \hat{\mathsf{V}}_1^{(d)} \hat{W}_{j-1}^{(d)}\big),
	\\
\partial_t \hat{W}_{j+1}^{(o)} = \hat{\mathsf{V}}_0 \hat{W}_{j+1}^{(d)}+\hat{\mathsf{V}}_1^{(o)} \hat{W}_j^{(d)}+\hat{\mathsf{V}}_1^{(d)} \hat{T}_j^{(o)}.
\end{cases}
\end{align}
The coefficients $\{\hat{T}_j(t), \hat{Z}_j(t), \hat{W}_j(t)\}$ are determined recursively from (\ref{Fjnormalization-t0})-(\ref{trecursive2-0}), the equations obtained by replacing $\{\hat{T}_j\}$ with $\{\hat{Z}_j\}$ in (\ref{trecursive-0}), and the initial assignments
$$\hat{T}_{-1} = 0, \qquad \hat{T}_0 = I, \qquad \hat{Z}_{-1} = 0, \qquad \hat{Z}_0 = I, \qquad \hat{W}_{-1}=\hat{W}_0 = 0.$$
The first few coefficients are given by
\begin{align*}\nonumber
\hat{T}_1(t) = &\; \frac {i(g_1-g_{0t})}{2} \sigma_1- \sigma_3 \frac{i}{4} \int_{\infty}^{t}\bigg[\frac{(g_1-g_{0t})^2}{2}+\cos g_0-1 \bigg]dt',
	\\ \nonumber
\hat{T}_2(t) = &\; i \bigg[g_{0tt}-g_{1t}+\frac{i}{2}(g_1-g_{0t})(\hat{T}_1)_{22} + \frac{1}{2} \sin g_0\bigg]\sigma_2
	\\ \nonumber
 &\;+ I\int_{\infty}^t \bigg\{\frac{i}{4}\Big[\frac{(g_{1} - g_{0t})^2}{2} + \cos{g_0} - 1\Big](\hat{T}_1)_{22}
 - \frac{1}{4}(g_{1} - g_{0t})(g_{1t} - g_{0tt})\bigg\} dt',
	\\ 
\hat{Z}_1(t) = &\; \frac {i(g_1-g_{0t})}{2} \sigma_1- \sigma_3 \frac{i}{4} \int_{0}^{t}\bigg[\frac{(g_1-g_{0t})^2}{2}+\cos g_0-1 \bigg]dt',
	\\ \nonumber
\hat{Z}_2(t) = &\; i \bigg[g_{0tt}-g_{1t}+\frac{i}{2}(g_1-g_{0t})(\hat{Z}_1)_{22} + \frac{1}{2} \sin g_0\bigg]\sigma_2
	\\\nonumber
 &\;+ I\bigg\{\int_{0}^t \bigg[\frac{i}{4}\Big[\frac{(g_{1} - g_{0t})^2}{2} + \cos{g_0} - 1\Big](\hat{Z}_1)_{22}
 - \frac{1}{4}(g_{1} - g_{0t})(g_{1t} - g_{0tt})\bigg] dt'
 	\\
& -\frac{1}{4} (g_1(0)-g_{0t}(0))^2\bigg\},
\end{align*}
and
\begin{align*}
\hat{W}_1(t) = & -\frac{i(g_1(0)-g_{0t}(0))}{2}\sigma_1,
	\\ \nonumber
\hat{W}_2(t) = &\; \frac{i(g_1-g_{0t})}{2} (\hat{W}_1)_{12} I 
+\sigma_2 \bigg\{\int_0^t \frac{1}{4}\bigg[\frac{(g_1-g_{0t})^2}{2} + \cos g_0 -1\bigg](\hat{W}_1)_{12}dt' 
	\\
&-i\Big[g_{0tt}(0) - g_{1t}(0) + \frac{1}{2} \sin g_0(0)\Big]\bigg\}.
\end{align*}

\begin{theorem}[Asymptotics of $\hat{T}$ and $\hat{U}$ as $k \to 0$]\label{tth3}
Let $g_0(t)$ and $g_1(t)$ be functions satisfying (\ref{gjassump}) for some given integers $m \geq 1$, $n \geq 1$, and $N_t \in \Z$.

As $k \to 0$, $\hat{T}$ and $\hat{U}$ coincide to order $m$ with $\hat{T}_{formal}$ and $\hat{U}_{formal}$ respectively in the following sense: The functions
\begin{align*}
&\hat{T}_p(t,k) = I + \hat{T}_1(t)k + \cdots + \hat{T}_{m+1}(t) k^{m+1},
	\\ \nonumber
&\hat{U}_p(t,k) = I + \hat{Z}_1(t)k + \cdots + \hat{Z}_{m+1}(t) k^{m+1}
 + \big(\hat{W}_1(t)k + \cdots + \hat{W}_{m+1}(t) k^{m+1}\big) e^{2i\theta_2t\sigma_3},
\end{align*}
are well-defined and, for each $j = 0, 1, \dots, n$,
\begin{align*}
& \bigg|\frac{\partial^j}{\partial k^j}\big(\hat{T} - \hat{T}_p\big) \bigg| \leq
\frac{C|k|^{m+1-2j}}{(1+t)^{n-j}}, \qquad t \geq 0, \  k \in (\bar{D}_2, \bar{D}_3),
	\\ 
& \bigg|\frac{\partial^j}{\partial k^j}\big(\hat{U} - \hat{U}_p\big) \bigg| \leq
C(1 + t)^{j+2}|k|^{m+1-2j}e^{Ct|k|^{m+1}}, \qquad t \geq 0, \  k \in (\bar{D}_3, \bar{D}_2)
	\\ 
& \bigg|\frac{\partial^j}{\partial k^j}\big((\hat{U} - \hat{U}_p)e^{-2i\theta_2t\sigma_3}\big) \bigg| \leq
C(1 + t)^{j+2} |k|^{m+1-2j} e^{Ct|k|^{m+1}}, \qquad t \geq 0, \  k \in (\bar{D}_2, \bar{D}_3).
\end{align*}
\end{theorem}
\begin{proof}
The proof is similar to that of Theorem \ref{xth3} and will be omitted.
\end{proof}

\begin{corollary}[Asymptotics of $T$ and $U$ as $k \to 0$]\label{tcor}
Let $g_0(t)$ and $g_1(t)$ be functions satisfying (\ref{gjassump}) for some given integers $m \geq 1$, $n \geq 1$, and $N_t \in \Z$.

As $k \to 0$, $T$ and $U$ coincide to order $m$ with $\mathcal{G}_0\hat{T}_{formal}$ and $\mathcal{G}_0\hat{U}_{formal} e^{-i\theta_2 t \hat{\sigma}_3}\mathcal{G}_0^{-1}(0)$ respectively in the following sense: 
For each $j = 0, 1, \dots, n$,
\begin{subequations}
\begin{align}
& \bigg|\frac{\partial^j}{\partial k^j}\big(T - \mathcal{G}_0\hat{T}_p\big) \bigg| \leq
\frac{C|k|^{m+1-2j}}{(1+t)^{n-j}}, \qquad t \geq 0, \  k \in (\bar{D}_2, \bar{D}_3),
	\\ \nonumber
& \bigg|\frac{\partial^j}{\partial k^j}\big(U - \mathcal{G}_0\hat{U}_p e^{-i\theta_2 t \hat{\sigma}_3}\mathcal{G}_0^{-1}(0)\big) \bigg| \leq C(1 + t)^{j+2}|k|^{m+1-2j}e^{Ct|k|^{m+1}}, 
	\\ 
&\hspace{9cm} t \geq 0, \  k \in (\bar{D}_3, \bar{D}_2),
	\\ \nonumber
& \bigg|\frac{\partial^j}{\partial k^j}\Big(\big(U - \mathcal{G}_0\hat{U}_p e^{-i\theta_2 t \hat{\sigma}_3}\mathcal{G}_0^{-1}(0)\big)e^{-2i\theta_2 t\sigma_3}\Big) \bigg| \leq
C(1 + t)^{j+2} |k|^{m+1-2j} e^{Ct|k|^{m+1}}, 
	\\ 
&\hspace{9cm} t \geq 0, \  k \in (\bar{D}_2, \bar{D}_3),
\end{align}
\end{subequations}
where $\mathcal{G}_0(t)$ is defined by (\ref{G0tdef}).
In particular, for each $t \geq 0$, it holds that $T(t,\cdot)$ is bounded for $k \in (\bar{D}_-, \bar{D}_+)$; $U(t,\cdot)$ is bounded for $k \in (\bar{D}_+, \bar{D}_-)$; and $U(t,\cdot)e^{-2i\theta_2 t\sigma_3}$ is bounded for $k \in (\bar{D}_-, \bar{D}_+)$.
\end{corollary}
\begin{proof}
The proof is similar to that of Corollary \ref{xcor} and will be omitted.
\end{proof}

\section{Spectral functions}

The next two theorems establish several properties of the spectral functions $a,b,A,B$.

\begin{theorem}[Properties of $a(k)$ and $b(k)$]\label{abth}
Suppose $\{u_0(x), u_1(x)\}$ satisfy (\ref{ujassump}) for some integers $m \geq 1$, $n \geq 1$, and $N_x \in \Z$.
Then the spectral functions  $a(k)$ and $b(k)$ defined in (\ref{abABdef}) have the following properties:
\begin{enumerate}[$(a)$]
\item $a(k)$ and $b(k)$ are continuous for $\im k \geq 0$ and analytic for $\im k > 0$.

\item For $j = 1, \dots n$, the derivatives $a^{(j)}(k)$ and $b^{(j)}(k)$ are well-defined and continuous on $\bar{\C}_+ \setminus \{0\}$.

\item There exist complex constants $\{a_i, b_i, \hat{a}_i, \hat{b}_i\}_{i=1}^m$ such that the following expansions hold uniformly for each $j = 0, 1, \dots n$:
\begin{align*}\nonumber
\begin{pmatrix} b^{(j)}(k) \\ a^{(j)}(k) \end{pmatrix}
= \begin{pmatrix} \frac{d^j}{dk^j}\big(\frac{b_1}{k} + \cdots + \frac{b_m}{k^m}\big) \\
\frac{d^j}{dk^j}\big(1 + \frac{a_1}{k} + \cdots + \frac{a_m}{k^m}\big)
\end{pmatrix} + O\bigg(\frac{1}{k^{m+1}}\bigg),\quad k \to \infty, \ \im k \geq 0,
\end{align*}
and
\begin{align*}
\begin{pmatrix} b^{(j)}(k)  \\ a^{(j)}(k) \end{pmatrix}
= G_0(0) \begin{pmatrix} \frac{d^j}{dk^j}\big(\hat{b}_1k + \cdots + \hat{b}_mk^m\big) \\
\frac{d^j}{dk^j}\big(1 + \hat{a}_1k + \cdots + \hat{a}_mk^m\big)
\end{pmatrix} + O\big(k^{m+1-2j}\big),\quad k \to 0, \ \im k \geq 0,
\end{align*}
where the matrix $G_0(x)$ is defined in (\ref{G0xdef}).
 
\item $a(k) = \overline{a(-\bar{k})}$ and $b(k) = \overline{b(-\bar{k})}$ for $\im k \geq 0$.

\item $|a(k)|^2 +|b(k)|^2 = 1$ for $k \in \R$.

\end{enumerate}
\end{theorem}
\begin{proof} 
The theorem follow upon setting $x = 0$ in Theorem \ref{xth1}, Theorem \ref{xth2}, and Corollary \ref{xcor}. 
\end{proof}

Evaluation of the coefficients $X_j(x)$ at $x = 0$ gives explicit formulas for the coefficients $\{a_j, b_j\}$ of Theorem \ref{abth} in terms of $u_0(x)$ and $u_1(x)$. For example, we find from (\ref{x1}) that
\begin{align}\nonumber 
& b_1 = \frac{i(u_{0x}(0) + u_1(0))}{2}, 
	\\\nonumber
& a_1 = \frac{i}{4} \int_0^\infty \bigg(-\frac{(u_{0x} + u_1)^2}{2} + \cos{u_0} -1\bigg)dx,
	\\\label{bjajexplicit}
& b_2 = \bigg(-u_{0xx} - u_{1x} + \frac{i}{2}(u_{0x} + u_1)a_1 + \frac{\sin u_0}{2}\bigg)\bigg|_{x=0}.
\end{align}
Similarly, evaluation of the $\hat{X}_j(x)$ at $x= 0$ gives formulas for the coefficients $\{\hat{a}_j, \hat{b}_j\}$. By (\ref{x1-0}), we have
\begin{align}\nonumber
& \hat{b}_1 = \frac{i(u_{0x}(0) - u_1(0))}{2}, 
	\\ \nonumber
& \hat{a}_1 = \frac{i}{4} \int_0^\infty \bigg(\frac{(u_{0x} - u_1)^2}{2} + 1 - \cos{u_0}  \bigg)dx,
	\\ \label{bjajhatexplicit}
& \hat{b}_2 = \bigg(u_{0xx} - u_{1x} + \frac{i}{2}(u_{0x} - u_1)\hat{a}_1 - \frac{\sin u_0}{2}\bigg)\bigg|_{x=0}.
\end{align}

\begin{theorem}[Properties of $A(k)$ and $B(k)$]\label{ABth}
Suppose $\{g_0(t), g_1(t)\}$ satisfy (\ref{gjassump}) for some integers $m \geq 1$, $n \geq 1$, and $N_t \in \Z$.
Then the spectral functions  $A(k)$ and $B(k)$ defined in (\ref{abABdef}) have the following properties:
\begin{enumerate}[$(a)$]
\item $A(k)$ and $B(k)$ are continuous for $k \in \bar{D}_+$ and analytic for $k \in D_+$.

\item For $j = 1, \dots n$, the derivatives $A^{(j)}(k)$ and $B^{(j)}(k)$ are well-defined and continuous on $\bar{D}_+ \setminus \{0\}$.

\item There exist complex constants $\{A_i, B_i, \hat{A}_i, \hat{B}_i\}_{i=1}^m$ such that the following expansions hold uniformly for each $j = 0, 1, \dots n$:
\begin{align*}\nonumber
\begin{pmatrix} B^{(j)}(k) \\ A^{(j)}(k) \end{pmatrix}
= \begin{pmatrix} \frac{d^j}{dk^j}\big(\frac{B_1}{k} + \cdots + \frac{B_m}{k^m}\big)
\\
\frac{d^j}{dk^j}\big(1 + \frac{A_1}{k} + \cdots + \frac{A_m}{k^m}\big)
\end{pmatrix} + O\bigg(\frac{1}{k^{m+1}}\bigg),\quad k \to \infty, \ k \in \bar{D}_+,
\end{align*}
and
\begin{align*}
\begin{pmatrix} B^{(j)}(k) \\ A^{(j)}(k) \end{pmatrix}
= \mathcal{G}_0(0) \begin{pmatrix} \frac{d^j}{dk^j}\big(\hat{B}_1k + \cdots + \bar{B}_mk^m\big) \\
\frac{d^j}{dk^j}\big(1 + \hat{A}_1k + \cdots + \hat{A}_mk^m\big)
\end{pmatrix} + O\big(k^{m+1-2j}\big),\quad k \to 0, \ k \in \bar{D}_+,
\end{align*}
where the matrix $\mathcal{G}_0(t)$ is defined in (\ref{G0tdef}).

\item $A(k) = \overline{A(-\bar{k})}$ and $B(k) = \overline{B(-\bar{k})}$ for $k \in \bar{D}_+$.

\item $A(k)\overline{A(\bar k)}+B(k)\overline{B(\bar k)}= 1$ for $k \in \R \cup \{|k| = 1\}$.

\end{enumerate}
\end{theorem}
\begin{proof} 
The theorem follows upon setting $t = 0$ in Theorem \ref{tth1}, Theorem \ref{tth2}, and Corollary \ref{tcor}.
\end{proof}

Evaluation of the coefficients $T_j(t)$ and $\hat{T}_j(t)$ at $t = 0$ gives explicit formulas for the coefficients $\{A_j, B_j, \hat{A}_j, \hat{B}_j\}$ of Theorem \ref{ABth} in terms of $g_0(t)$ and $g_1(t)$. The first few coefficients are given by 
\begin{align}\nonumber
& B_1 = \frac{i(g_1(0) + g_{0t}(0))}{2}, 
	\\\nonumber
& A_1 = -\frac{i}{4} \int_0^\infty \bigg(\frac{(g_1 + g_{0t})^2}{2} + \cos{g_0} -1\bigg)dt, 
	\\ \label{BjAjexplicit}
& B_2 = \bigg(-g_{0tt} - g_{1t} + \frac{i}{2}(g_1 + g_{0t})A_1 - \frac{\sin g_0}{2}\bigg)\bigg|_{t=0},
\end{align}
and
\begin{align}\nonumber
& \hat{B}_1 = \frac{i(g_1(0) - g_{0t}(0))}{2}, 
	\\\nonumber
& \hat{A}_1 = -\frac{i}{4} \int_0^\infty \bigg(\frac{(g_1 - g_{0t})^2}{2} + \cos{g_0} -1\bigg)dt,
	\\ \label{BjAjhatexplicit}
& \hat{B}_2 = \bigg(g_{0tt} - g_{1t} + \frac{i}{2}(g_1 - g_{0t})\hat{A}_1 + \frac{\sin g_0}{2}\bigg)\bigg|_{t=0}. 
\end{align}

Our last result is concerned with the spectral functions $c(k)$ and $d(k)$ defined in (\ref{cddef}). Note that (\ref{cddef}) implies
$$S(k)^{-1}s(k) = \begin{pmatrix}
\overline{d(\bar{k})} 	&	c(k)	\\
-\overline{c(\bar{k})}	&	d(k)
\end{pmatrix}, \qquad k \in \R.$$

\begin{definition}\upshape
We say that the initial and boundary values $u_0, u_1, g_0, g_1$ are {\it compatible with equation (\ref{sg}) to order $m$ at $x=t=0$} if all of the following relations which involve derivatives of at most $m$th order of $u_0, g_0$ and at most $(m-1)$th order of $u_1, g_1$ hold:
\begin{align*}
& g_0(0) = u_0(0), \quad g_0'(0) = u_1(0), \quad g_1(0) = u_0'(0), \quad g_1'(0) = u_1'(0), 
	\\
&g_0''(0) - u_0''(0) + \sin u_0(0) = 0, \quad g_1''(0) - u_0'''(0) + (\sin u_0)'(0) = 0, \quad \text{etc.}
\end{align*}
\end{definition}

\begin{theorem}[Properties of $c(k)$ and $d(k)$]\label{cdth}
Suppose $\{u_j(x), g_j(t)\}_0^1$ satisfy (\ref{ujassump}) and (\ref{gjassump}) for some given integers $m \geq 1$, $n \geq 1$, $N_x \in \Z$, and $N_t \in \Z$. 
Suppose $\{u_j(x), g_j(t)\}_0^1$ are compatible with equation (\ref{sg}) to order $m+1$ at $x=t=0$.
Then the spectral functions  $c(k)$ and $d(k)$ defined in (\ref{cddef}) have the following properties:
\begin{enumerate}[$(a)$]
\item $c(k)$ is continuous for $k \in \bar{D}_1 \cup \R$ and analytic for $k \in D_1$.

\item $d(k)$ is continuous for $k \in \bar{D}_2 \cup \R$ and analytic for $k \in D_2$.

\item For $j = 1, \dots n$, the derivatives $c^{(j)}(k)$ and $d^{(j)}(k)$ are well-defined and continuous on $(\bar{D}_1 \cup \R) \setminus \{0\}$ and $(\bar{D}_2 \cup \R) \setminus \{0\}$, respectively.

\item There exist complex constants $\{d_i, \hat{d}_i\}_{i=1}^m$ such that the following expansions hold uniformly for each  $j = 0, 1, \dots, n$:
\begin{subequations}\label{cdexpansions}
\begin{align}\label{cdexpansionsa}
& c^{(j)}(k) = O\big(k^{-m-1}\big), \qquad k \to \infty, \ k \in \bar{D}_1,
	\\\label{cdexpansionsb}
& d^{(j)}(k) = \frac{d^j}{dk^j}\bigg(1 + \frac{d_1}{k} + \cdots + \frac{d_m}{k^m}\bigg) + O\big(k^{-m-1}\big), \qquad |k| \to \infty, \ k \in \R,	
\end{align}
and
\begin{align}\label{cdexpansionsc}
& c^{(j)}(k) = O\big(k^{m+1-2j}\big), \qquad |k| \to 0, \ k \in \R,
	\\\label{cdexpansionsd}
& d^{(j)}(k) = (-1)^{N_x - N_t}\frac{d^j}{dk^j}\big(1 + \hat{d}_1k + \cdots + \hat{d}_mk^m\big) + O\big(k^{m+1-2j}\big), \qquad k \to 0, \ k \in \bar{D}_2.	
\end{align}
\end{subequations}

\item $c(k) = \overline{c(-\bar{k})}$ for $k \in \bar{D}_1 \cup \R$ and $d(k) = \overline{d(-\bar{k})}$ for $k \in \bar{D}_2 \cup \R$.

\item $|c(k)|^2 + |d(k)|^2 = 1$ for $k \in \R$.

\end{enumerate}
\end{theorem}
\begin{proof} 
We only have to prove the expansions in (\ref{cdexpansions}), because all other properties are direct consequences of Theorem \ref{abth} and Theorem \ref{ABth}. 

Let $\{u_j(x), g_j(t)\}_0^1$ be functions satisfying (\ref{ujassump}) and (\ref{gjassump}) for some $m,n \geq 1$, which are compatible with equation (\ref{sg}) to order $m+1$ at $x=t=0$.
Define $\hat{a},\hat{b},\hat{A},\hat{B}$ by
\begin{align*}
\hat{s}(k) = \begin{pmatrix}
\overline{\hat{a}(\bar{k})} 	&	\hat{b}(k)	\\
-\overline{\hat{b}(\bar{k})}	&	\hat{a}(k)
\end{pmatrix}, \qquad
\hat{S}(k) = \begin{pmatrix}
\overline{\hat{A}(\bar{k})} 	&	\hat{B}(k)	\\
-\overline{\hat{B}(\bar{k})}	&	\hat{A}(k)
\end{pmatrix},
\end{align*}
where $\hat{s}(k) := \hat{X}(0,k)$ and $\hat{S}(k) := \hat{T}(0,k)$.
The compatibility relation $u_0(0) = g_0(0)$ implies that $\mathcal{G}_0(0) = (-1)^{N_x - N_t}G_0(0)$, where $G_0(x)$ and $\mathcal{G}_0(t)$ are the matrices defined in (\ref{G0xdef}) and (\ref{G0tdef}). We find
$S^{-1}s = (\mathcal{G}_0(0)\hat{S})^{-1}G_0(0)\hat{s} = (-1)^{N_x - N_t}\hat{S}^{-1}\hat{s}$ and hence 
$$c(k) = (-1)^{N_x - N_t}\big[\hat{b}(k)\hat{A}(k) - \hat{B}(k)\hat{a}(k)\big], \quad
d(k) = (-1)^{N_x - N_t}\big[\hat{a}(k)\overline{\hat{A}(\bar{k})} + \hat{b}(k)\overline{\hat{B}(\bar{k})}\big].$$ Using that 
$$\hat{b}(k) = O(k), \quad \hat{a}(k) = 1 + O(k), \quad \hat{B}(k) = O(k), \quad \hat{A}(k) = 1 + O(k),$$
as $k \to 0$, this gives $\lim_{k\to0}d(k) = (-1)^{N_x - N_t}$ and $\lim_{k\to0}c(k) = 0$. 
The asymptotic formulas (\ref{cdexpansionsb}) and (\ref{cdexpansionsd}) for $d(k)$ now follow from the expansions of $a,b,A,B$ established in Theorems \ref{abth} and \ref{ABth}. These expansions also imply that that there exist complex constants $\{c_i, \hat{c}_i\}_{i=1}^m$ such that, for $j = 0, 1, \dots, n$,
\begin{align}\label{cexpansioninfty}
& c^{(j)}(k) = \frac{d^j}{dk^j}\bigg(\frac{c_1}{k} + \cdots + \frac{c_m}{k^m}\bigg) + O\bigg(\frac{1}{k^{m+1}}\bigg) \quad
 \text{uniformly as $k \to \infty$, $k \in \bar{D}_1$},
\end{align}
and
\begin{align}\label{cexpansionzero}
& c^{(j)}(k) = \frac{d^j}{dk^j}\bigg(\hat{c}_1k + \cdots + \hat{c}_m k^m\bigg) + O(k^{m+1-2j}) \quad \text{as $k \to 0$, $k \in \R$}.
\end{align}
It only remains to prove that the coefficients $c_i$ and $\hat{c}_i$ vanish.

 Let us first consider the $c_i$. 
The coefficients $c_i$ can be expressed in terms of the coefficients of the expansions of $a,b,A,B$.
For example, the computation
\begin{align*}
c(k) = &\; b(k)A(k) - B(k)a(k) = (b_1k^{-1} + b_2 k^{-2} + \cdots)(1 + A_1k^{-1} + \cdots)
	\\
& - (B_1k^{-1} + B_2 k^{-2} + \cdots)(1 + a_1k^{-1} + \cdots)
	\\
= &\; (b_1 - B_1)k^{-1} + (b_2 - B_2 + b_1A_1 - B_1 a_1) k^{-2} + \cdots, \qquad k \to \infty,
\end{align*}
shows that 
$$c_1 = b_1 - B_1, \qquad c_2 = b_2 - B_2 + b_1A_1 - B_1 a_1.$$
It follows immediately from the explicit expressions in (\ref{bjajexplicit}) and (\ref{BjAjexplicit}) together with the assumption of compatibility at $x =t=0$ that $c_1 = c_2 = 0$. Using the more general formula
\begin{align}\label{cjexplicit}
c_j = b_j - B_j + \sum_{r= 1}^{j-1}(b_{j-r}A_r - B_{j-r}a_r), \qquad j = 1, 2, \dots.
\end{align}
we can show that $c_j = 0$ for some higher values of $j$; however, the computations quickly become exceedingly complicated as $j$ increases. In order to prove (\ref{cdexpansionsa}) when $m \geq 3$, we therefore use a different argument. 

Let $X,Y,T,U$ denote the eigenfunctions defined in (\ref{XYdef}) and (\ref{TUdef}).
Define a solution $F(t,k)$ of the $t$-part (\ref{tpartF}) for $k \in \R$ by
$$F(t,k) = U(t,k)e^{-i\theta_2 t\hat{\sigma}_3}s(k).$$
Then $F(0,k) = s(k) = X(0,k)$. Let
$$F_{formal}(t,k) = I + \frac{F_1(t)}{k} + \frac{F_2(t)}{k^2} + \cdots 
+ \bigg(\frac{G_1(t)}{k} + \frac{G_2(t)}{k^2} + \cdots \bigg) e^{2i\theta_2t\sigma_3}$$
denote a formal power series solution of (\ref{tpartF}) normalized to equal $s(k)$ at $t = 0$. The coefficients $\{F_j(t)\}$ and $\{G_j(t)\}$ satisfy the recursive relations (\ref{trecursive}) and (\ref{trecursive2}) respectively, and the initial condition
$F_j(0) + G_j(0) = X_j(0)$, where $X_j(x)$ are the coefficients in the formal power series (\ref{Xformaldef}) approximating $X(x,k)$.
That is, $\{F_j(t)\}$ and $\{G_j(t)\}$ are determined by the system
\begin{align}\label{FjGjsystem}
\begin{cases}
F_{j+1}^{(o)} = -2 i \sigma_3\big(-\partial_tF_{j}^{(o)}-\frac{i}{2}\sigma_3 F_{j-1}^{(o)}+\mathsf{V}_0 F_j^{(d)}+\mathsf{V}_1^{(o)} F_{j-1}^{(d)}+\mathsf{V}_1^{(d)} F_{j-1}^{(o)}\big),
	\\
\partial_t F_{j+1}^{(d)} = \mathsf{V}_0 F_{j+1}^{(o)}+\mathsf{V}_1^{(o)} F_j^{(o)}+\mathsf{V}_1^{(d)} F_j^{(d)},
	\\
G_{j+1}^{(d)} = -2 i \sigma_3\big(-\partial_tG_{j}^{(d)}-\frac{i}{2}\sigma_3 G_{j-1}^{(d)}+\mathsf{V}_0 G_j^{(o)}+\mathsf{V}_1^{(o)} G_{j-1}^{(o)}+\mathsf{V}_1^{(d)} G_{j-1}^{(d)}\big),
	\\
\partial_t G_{j+1}^{(o)} = \mathsf{V}_0 G_{j+1}^{(d)}+\mathsf{V}_1^{(o)} G_j^{(d)}+\mathsf{V}_1^{(d)} T_j^{(o)},
	\\
F_j(0) + G_j(0) = X_j(0),
\end{cases} \quad  j \geq 1,
\end{align}
and the assignments $F_{-1} = G_{-1} = G_0 = 0$ and $F_0 = I$.

\begin{lemma}
$F_{formal}$ approximates $F$ to order $m$ as $k \to \infty$ in the sense that the function
\begin{align}\label{Fpdef}
&F_p(t,k) = I + \frac{F_1(t)}{k} + \cdots + \frac{F_{m}(t)}{k^{m}}
 + \bigg(\frac{G_1(t)}{k} + \cdots + \frac{G_{m}(t)}{k^{m}}\bigg) e^{2i\theta_2t\sigma_3},
\end{align}
is well-defined and
\begin{align}\label{FFpestimate}
& |F(t,k) - F_p(t,k)| \leq \frac{C}{|k|^{m+1}}, \qquad t \geq 0, \  k \in \R, \ |k| > 1.
\end{align}
\end{lemma}
\stepproofbegin
This can be proved directly by the same type of arguments that gave Theorems \ref{xth2} and \ref{tth2}. We instead choose to relate $F$ to $T$ and then appeal to the results of Theorem \ref{tth2} for $T$.  We have
\begin{align*}
F(t,k) = T(t,k) e^{-i\theta_2 t\hat{\sigma}_3}(S(k)^{-1}s(k)), \qquad t \geq 0, \ k \in \R.
\end{align*}
That is,
$$F(t,k) = \mathcal{F}(t,k) + \mathcal{G}(t,k) e^{2it\theta_2\sigma_3},$$
where 
$$\mathcal{F}(t,k) = \begin{pmatrix}[T]_1\overline{d(\bar{k})} & [T]_2d(k) \end{pmatrix}, \qquad
\mathcal{G}(t,k) = \begin{pmatrix}-[T]_2\overline{c(\bar{k})} & [T]_1c(k) \end{pmatrix}.$$
Note that $\mathcal{F}(0,k) + \mathcal{G}(0,k) = s(k)$. Thus, by uniqueness of the solution of (\ref{FjGjsystem}), we have
\begin{align}\label{FTidentity}
F_{formal}(t,k) =  \mathcal{F}_{formal}(t,k) + \mathcal{G}_{formal}(t,k) e^{2it\theta_2\sigma_3},
\end{align}
where $\mathcal{F}_{formal}$ and $\mathcal{G}_{formal}$ denote the formal expansions of $\mathcal{F}$ and $\mathcal{G}$ as $k \to \infty$. Since the expansion $T_p(t,k)$ of $T$ is well-defined, so is $F_p(t,k)$.

By (\ref{FTidentity}), the second columns of $I + \frac{F_1(t)}{k} + \cdots $ and $\frac{G_1(k)}{k} + \cdots $ are given by the formal expansions as $k \to \infty$ of $[T(t,k)]_2 d(k)$ and $[T(t,k)]_1c(k)$, respectively; hence
\begin{align}\label{FjtGjt}
[F_j(t)]_2 = \sum_{i=0}^j [T_i(t)]_2 d_{j-i}, \quad [G_j(t)]_2 = \sum_{i = 0}^{j-1} [T_i(t)]_1 c_{j-i}, \qquad j = 1, \dots, m,
\end{align}
where $T_0(t) = I$ and $d_0 =1$. Setting $F_0(t) = I$, we can write
$$[F - F_p]_2 = \bigg([T]_1c(k) - \sum_{j=1}^m \frac{[G_j(t)]_2}{k^j}\bigg)e^{-2i\theta_2 t} 
+ [T]_2d(k) - \sum_{j=0}^m \frac{[F_j(t)]_2}{k^j},$$
which leads to the estimate
\begin{align}\nonumber
|[F - F_p]_2| \leq &\; |[T - T_p]_1| |c(k)| + \bigg|[T_p]_1c(k) - \sum_{j=1}^m \frac{[G_j(t)]_2}{k^j}\bigg| 
+  |[T - T_p]_2| |d(k)|
	\\\label{FFp2est}
&  +  \bigg|[T_p]_2d(k) -  \sum_{j=0}^m \frac{[F_j(t)]_2}{k^j}\bigg|, \qquad t \geq 0, \ k \in \R, \ |k| > 1.
\end{align}
By (\ref{varphiasymptoticsa-t}), 
\begin{align}\label{T1Tp1}
|[T - T_p]_1| |c(k)| + |[T - T_p]_2| |d(k)| \leq \frac{C}{|k|^{m+1}(1+t)^n}, \qquad t \geq 0, \ k \in \R, \ |k| > 1.
\end{align}
Since $|T_j(t)| \leq C$ for $t \geq 0$ and $j = 1, \dots, m+1$ (cf. (\ref{Fjbounds})), the second column of (\ref{FFpestimate}) follows from equations (\ref{FFp2est}) and (\ref{T1Tp1}). The proof of the first column is similar.
\proofendcontinue

\begin{lemma}
All the coefficients $G_j(t)$ in (\ref{Fpdef}) vanish identically, i.e., $G_1(t) = \cdots = G_m(t) = 0$.
\end{lemma}
\stepproofbegin
Suppose $G_i(t) = 0$ for $i = 1, \dots, j$. Then (\ref{FjGjsystem}) implies that $G_{j+1}(t)$ is off-diagonal and constant, so that $G_{j+1}(t) = G_{j+1}(0)$. It is therefore enough to show that $G_j(0) = 0$, or, equivalently, that $F_j(0) = X_j(0)$ for $j = 1, \dots, m$. 

Starting with $F_{-1} = G_{-1} = G_0 = 0$ and $F_0 = I$, the system (\ref{FjGjsystem}) with $j = 0$ gives $F_1^{(o)}(t) = - 2i\sigma_3\mathsf{V}_0(t)$ and $G_1^{(d)}(t) =  0$. In particular, since $F_1(0) + G_1(0) = X_1(0)$, we have $F_1^{(d)}(0) = X_1^{(d)}(0)$. Since $X_1^{(o)}(0) = - 2i\sigma_3\mathsf{U}_0(0)$, the requirement $F_1(0) = X_1(0)$ therefore reduces to the compatibility condition $\mathsf{U}_0(0) = \mathsf{V}_0(0)$, which holds because $u_{0x}(0) = g_1(0)$ and $u_1(0) = g_{0t}(0)$.

The system (\ref{FjGjsystem}) with $j = 1$ now gives
$$F_2^{(o)}(0) = - 2i\sigma_3(2i\sigma_3\partial_t\mathsf{V}_0(0) + \mathsf{V}_0(0)X_1^{(d)}(0) + \mathsf{V}_1^{(o)}(0)), \qquad
G_2^{(d)}(t) = 0.$$
In particular, since $F_2(0) + G_2(0) = X_2(0)$, we have $F_2^{(d)}(0) = X_2^{(d)}(0)$.
Since
$$X_2^{(o)}(0) = - 2i\sigma_3(2i\sigma_3\partial_x\mathsf{U}_0(0) + \mathsf{U}_0(0)X_1^{(d)}(0) + \mathsf{U}_1^{(o)}(0)),$$
the requirement $F_2^{(o)}(0) = X_2^{(o)}(0)$ reduces to the compatibility condition 
$$2i\sigma_3\partial_t\mathsf{V}_0(0) + \mathsf{V}_0(0)X_1^{(d)}(0) + \mathsf{V}_1^{(o)}(0)
= 2i\sigma_3\partial_x\mathsf{U}_0(0) + \mathsf{U}_0(0)X_1^{(d)}(0) + \mathsf{U}_1^{(o)}(0),$$
which holds because $u_0(0) = g_0(0)$, $u_{0x}(0) = g_1(0)$, $u_1(0) = g_{0t}(0)$, $u_{1x}(0) = g_{1t}(0)$, and $g_{0tt}(0) - u_{0xx}(0)  + \sin u_0(0) = 0$.

This completes the proof if $m = 1$ or $m = 2$. In order to treat higher values of $m$, we let $u(x,t)$ denote a $C^{m+1}$-solution of equation (\ref{sg}) defined for all small $x \geq 0$ and $t \geq 0$ such that $u(x,0) = u_0(x)$ and $u_t(x,0) = u_1(x)$. Such a solution can be obtained, for example, by extending $u_0(x)$ and $u_1(x)$, respectively, to the negative real axis and appealing to the existence results for the initial-value problem for (\ref{sg}) (see Appendix B of \cite{BM2008}). 
In general, we have $u(0,t) \neq g_0(t)$ and $u_x(0,t) \neq g_1(t)$ for $t > 0$. However, since $\{u_j(x), g_j(t)\}_0^1$ are compatible with equation (\ref{sg}) to order $m+1$ at $x=t=0$, we have 
\begin{align}\label{u00g00g10}
\begin{cases}
\partial_t^j u(0,0) =  \partial_t^jg_0(0), & j = 1, \dots, m +1,\\ 
\partial_t^j u_x(0,0) =  \partial_t^jg_1(0), & j = 1, \dots, m. 
\end{cases} 
\end{align}

It is possible to define a formal power series solution
\begin{align}\label{muformaldef}
\mu_{formal}(x,t,k) = I + \frac{\mu_1(x,t)}{k} + \frac{\mu_2(x,t)}{k^2} + \cdots,
\end{align}
of the Lax pair (\ref{lax}) such that  $\mu_j(0,0) = X_j(0)$ for $j = 1, \dots, m$. Indeed, substituting (\ref{muformaldef}) into (\ref{lax}) the off-diagonal terms of $O(k^{-j})$ and the diagonal terms of $O(k^{-j-1})$ yield (cf. (\ref{xrecursive}) and (\ref{trecursive}))
\begin{align}\label{murecursive}
\begin{cases}
\mu_{j+1}^{(o)} = -2 i \sigma_3\big(-\partial_x \mu_{j}^{(o)}+\frac{i}{2}\sigma_3 \mu_{j-1}^{(o)}+Q_0 \mu_j^{(d)}+Q_1^{(o)} \mu_{j-1}^{(d)}+Q_1^{(d)} \mu_{j-1}^{(o)}\big),
	\\
\partial_x \mu_{j+1}^{(d)} = Q_0 \mu_{j+1}^{(o)}+Q_1^{(o)} \mu_j^{(o)}+Q_1^{(d)} \mu_j^{(d)}.
	\\
\partial_t \mu_{j+1}^{(d)} = Q_0 \mu_{j+1}^{(o)}+Q_1^{(o)} \mu_j^{(o)}+Q_1^{(d)} \mu_j^{(d)}.
\end{cases}
\end{align}
The two equations for $\mu_{j+1}^{(d)}$ are compatible, because the Lax pair equations in (\ref{lax}) are compatible for every $k$ as a consequence of the fact that $u(x,t)$ satisfies (\ref{sg}); hence the system (\ref{murecursive}) can be solved recursively. Using that $u(x,t)$ is $C^{m+1}$, we find that the coefficients $\mu_j(x,t)$ are well-defined and continuous functions of $x \geq 0$ and $t \geq 0$ near the origin for $j = 1, \dots, m$. 

By evaluating the system (\ref{FjGjsystem}) at $t = 0$, we see that the values of $\partial^l F_{j+1}(0)$ and $\partial^l G_{j+1}(0)$, $l \geq 0$, are uniquely determined from the values of the functions $\{F_{i}(t), G_{i}(t)\}_{i=1}^j$ and their derivatives {\it evaluated at $t = 0$}. Hence it makes sense to say that $\{\partial_t^l F_{j}(0), \partial_t^l G_{j}(0)\}$ satisfies the system (\ref{FjGjsystem}) restricted to $t = 0$ and that this solution is unique. 
Now $\mu_{formal}(0,t,k) = I + \frac{\mu_1(0,t)}{k} + \frac{\mu_2(0,t)}{k^2} + \cdots$, is a formal solution of the $t$-part (\ref{tpartF}) such that $\mu_j(0,0) = X_j(0)$ for each $j$. In view of (\ref{u00g00g10}), it follows that $\partial_t^lF_j(0) = \partial_t^l\mu_j(0,0)$, $\partial^l G_j(0) = 0$ is a solution of the system (\ref{FjGjsystem}) restricted to $t = 0$. By uniqueness, we conclude that $F_j(0) = \mu_j(0,0) = X_j(0)$ and $G_j(0) = 0$ for $j = 1, \dots, m$. 
\proofendcontinue

\begin{lemma}
The coefficients $c_i$ in (\ref{cexpansioninfty}) all vanish, that is, $c_1 = \cdots = c_m = 0$.
\end{lemma}
\stepproofbegin
Using that $G_j(t) = 0$, the first entry of the second equation in (\ref{FjtGjt}) yields
$$0 = c_j + \sum_{i = 1}^{j-1} (T_i(t))_{11} c_{j-i}, \qquad t \geq 0, \ j = 1, \dots, m,$$
For $j = 1$, this equation gives $c_1 = 0$; the equation with $j = 2$ then gives $c_2 = 0$ and so on. 
\proofend

The following lemma completes the proof of the theorem.

\begin{lemma}
The coefficients $\hat{c}_i$ in (\ref{cexpansionzero}) all vanish, that is, $\hat{c}_1 = \cdots = \hat{c}_m = 0$.
\end{lemma}
\stepproofbegin
The proof is analogous to the above proof that the $c_i$ vanish, except that it employs the eigenfunctions $\hat{X}$ and $\hat{T}$ instead of $X$ and $T$. We omit details. 
\end{proof}

\begin{remark}\upshape
The coefficients $d_j$ and $\hat{d}_j$ in Theorem \ref{cdth} are generally nonzero. 
For example, from (\ref{bjajexplicit}) and (\ref{BjAjexplicit}) we see that
\begin{align}\nonumber
d_1 = &\; a_1 + \bar{A}_1
	\\\nonumber
=&\; \frac{i}{4} \int_0^\infty \bigg(-\frac{(u_{0x} + u_1)^2}{2}+ \cos{u_0} -1\bigg)dx
	\\ \label{d1explicit}
& + \frac{i}{4} \int_0^\infty \bigg(\frac{(g_1 + g_{0t})^2}{2} + \cos{g_0} -1\bigg)dt.
\end{align}
In the case when $u_0,u_1,g_0,g_1$ are the initial and boundary values for a quarter-plane solution of (\ref{sg}), then the $d_j$ and $\hat{d}_j$ encode the values of the infinite sequence of conservation laws associated with (\ref{sg}). For example, letting $u(x,t)$  denote the associated solution, we can write $d_1 = \int_\gamma \omega_1$, where the one-form  $\omega_1$  is defined by
$$\omega_1(x,t) = -\frac{i}{4} \bigg(-\frac{(u_x + u_t)^2}{2} + \cos{u} -1\bigg)dx
+ \frac{i}{4} \bigg(\frac{(u_x + u_t)^2}{2} + \cos{u} -1\bigg)dt,$$
and $\gamma$  consists of the half-line $\R_+$ traversed to the left followed by the positive imaginary axis $i\R_+$ traversed upwards. A computation using (\ref{sg}) shows that $\omega_1$  is closed and therefore encodes a conservation law for sine-Gordon. In particular, the contour $\gamma$  can be deformed into any contour going from $(\infty, 0)$ to $(0,\infty)$ without affecting the value of $\int_\gamma \omega_1$.
\end{remark}

\bigskip
\noindent
{\bf Acknowledgement}  {\it Support is acknowledged from the G\"oran Gustafsson Foundation, the European Research Council, Consolidator Grant No. 682537, the Swedish Research Council, Grant No. 2015-05430, and the National Science Foundation of China, Grant No. 11671095.}

\bibliographystyle{plain}
\bibliography{is}

\end{document}